\documentclass[10pt]{amsart}
      \usepackage[mathscr]{eucal}
      \usepackage{amsmath,amsfonts, amssymb}
\def\mapright#1{\smash{
\mathop{\rg}\limits^{#1}}}
\def\mapdown#1{\bigg\downarrow
\rlap{$\vcenter{\hbox{$\scriptstyle#1$}}$}}

\def\rg{\hbox to 30pt{\rightarrowfill}}
\def\lg{\hbox to 30pt{\leftarrowfill}}

      \parskip\smallskipamount
          \newtheorem{theorem}{Theorem}[section]
      
      \newtheorem{proposition}[theorem]{Proposition}
      \newtheorem{corollary}[theorem]{Corollary}
      \newtheorem{lemma}[theorem]{Lemma}

      \makeatletter
      \@addtoreset{equation}{section}
      \makeatother

      \newcommand{\BB}{{\mathbb B}}
      \newcommand{\CC}{{\mathbb C}}
      \newcommand{\NN}{{\mathbb N}}
      \newcommand{\HH}{{\mathbb H}}

      \newcommand{\FF}{{\mathbb F}}
      \newcommand{\TT}{{\mathbb T}}

      \newcommand{\cA}{{\mathcal A}}
      
      \newcommand{\cC}{{\mathcal C}}
      \newcommand{\cD}{{\mathcal D}}
      \newcommand{\cE}{{\mathcal E}}
      \newcommand{\cF}{{\mathcal F}}
      \newcommand{\cG}{{\mathcal G}}
      \newcommand{\cH}{{\mathcal H}}
      \newcommand{\cI}{{\mathcal I}}
      \newcommand{\cK}{{\mathcal K}}
      
      \newcommand{\cM}{{\mathcal M}}
      \newcommand{\cN}{{\mathcal N}}
      \newcommand{\cP}{{\mathcal P}}
      \newcommand{\cR}{{\mathcal R}}
      \newcommand{\cS}{{\mathcal S}}

      \newcommand{\cV}{{\mathcal V}}
      \newcommand{\cY}{{\mathcal Y}}

      \newdimen\expt
      \def\boxit#1{\setbox0\hbox{$\displaystyle{#1}$}
            \hbox{\lower.4\expt
       \hbox{\lower3\expt\hbox{\lower\dp0
            \hbox{\vbox{\hrule height.4\expt
       \hbox{\vrule width.4\expt\hskip3\expt
            \vbox{\vskip3\expt\box0\vskip2\expt}%
       \hskip3\expt\vrule width.4\expt}\hrule height.4\expt}}}}}}
      
\begin{document}
       \pagestyle{myheadings}
      \markboth{ Gelu Popescu}{    Free biholomorphic functions and operator model theory II }

      \title [ Free biholomorphic functions and operator model theory, II ]
{ Free biholomorphic functions and operator model theory, II }
        \author{Gelu Popescu}
\date{November 11, 2011}
      \thanks{Research supported in part by an NSF grant}

      \address{Department of Mathematics, The University of Texas
      at San Antonio \\ San Antonio, TX 78249, USA}
      \email{\tt gelu.popescu@utsa.edu}

\bigskip

\begin{abstract}

In a companion to this paper, we introduced the class of $n$-tuples
$f=(f_1,\ldots, f_n)$  of formal power series in noncommutative
indeterminates $Z_1,\ldots, Z_n$  with  the {\it model property} and
developed an operator model theory for {\it pure} $n$-tuples of
operators in noncommutative domains $\BB_f(\cH)\subset B(\cH)^n$,
where the associated universal model is an $n$-tuple
$(M_{Z_1},\ldots, M_{Z_n})$ of multiplication operators on a Hilbert
space $\HH^2(f)$ of formal powers series.

In the present paper, several results concerning the noncommutative
multivariable operator theory on
 the unit ball  $[B(\cH)^n]_1^-$  are extended to
 noncommutative
varieties   $\cV_{f,J}(\cH)\subseteq \BB_f(\cH)$ defined by
$$
\cV_{f,J}(\cH):=\left\{ (T_1,\ldots, T_n)\in \BB_f(\cH):\
\psi(T_1,\ldots,T_n)=0 \ \text{ for any } \ \psi\in J\right\},
$$
for an appropriate evaluation $\psi(T_1,\ldots,T_n)$, and associated
with    $n$-tuples $f$ with the model property and  WOT-closed
two-sided ideals $J$ of the Hardy algebra $H^\infty(\BB_f)$, the
WOT-closure of all
  noncommutative polynomials in $M_{Z_1},\ldots, M_{Z_n}$ and the
  identity. We develop an operator model theory and dilation
theory for $\cV_{f,J} (\cH)$,
  where the associated universal model  is  an $n$-tuple $(B_1,\ldots, B_n)$ of
     operators  acting on a Hilbert space $\cN_{f, J}$
  of formal power series. We  study the representations of the algebras
    generated by $B_1,\ldots, B_n$  and the identity: the variety algebra $\cA(\cV_{f,J})$,
    the Hardy algebra $H^\infty(\cV_{f,J}))$, and
    the $C^*$-algebra $C^*(B_1,\ldots, B_n)$. A constrained characteristic function $\Theta_{f,X,J}$,
      associated with each  $n$-tuple
  $X\in \cV_{f,J} (\cH)$,
  is used  to provide an operator
   model  for   the  class of
    completely non-coisometric (c.n.c) elements in the noncommutative variety
    $\cV_{f,J} (\cH)$. As a consequence, we show that $\Theta_{f,X,J}$ is a complete unitary invariant
    for the c.n.c. part of $\cV_{f,J} (\cH)$.  A  Beurling type theorem characterizing the joint invariant
   subspaces  under $B_1,\ldots, B_n$ and a  commutant lifting for  pure  $n$-tuples of operators in
$\cV_{f,J}(\cH)$ is also provided. In particular, when $J$ is the
$WOT$-closed two-sided ideal generated by the commutators $
M_{Z_i}M_{Z_j}-M_{Z_j}M_{Z_i}$, $i,j\in \{1,\ldots, n\}$, we obtain
commutative versions for all the results.

For special  classes of  $n$-tuples of formal power series
  $f=(f_1,\ldots, f_n)$ and $J=\{0\}$, we  obtain several results
  regarding the dilation and model theory for the noncommutative
  domain $\BB_f(\cH)$ or the c.n.c. part of it.

\end{abstract}

      \maketitle

\bigskip

\bigskip
\bigskip





\bigskip

\section*{Introduction}

This paper is  a continuation of \cite{Po-multi} in our  attempt to
to {\it transfer}
   the free analogue of Nagy-Foia\c s theory (see \cite{SzF-book}, \cite{F},
\cite{B}, \cite{Po-isometric}, \cite{Po-charact}, \cite{Po-von},
\cite{Po-funct}, \cite{Po-analytic}, \cite{Po-disc},
\cite{Po-interpo}, \cite{Po-poisson}, \cite{Po-curvature},
 \cite{Po-entropy}, \cite{Po-varieties},
\cite{Po-varieties2}, \cite{Po-unitary}, \cite{Po-domains},
\cite{Po-multi},
 \cite{Arv},  \cite{BT1},   \cite{BES1}) from
the  closed  unit ball
$$
[B(\cH)^n]_1^-:=\{(X_1,\ldots, X_n)\in B(\cH)^n: \ X_1X_1^*+\cdots
+ X_nX_n^*\leq  I\}
$$
 to other
    noncommutative domains in $B(\cH)^n$, where $B(\cH)$ is the algebra
     of bounded linear operators  an a
Hilbert space $\cH$. More precisely, we want to find large classes
$\cG$  of free holomorphic
 functions $g :\Omega\subseteq [B(\cH)^n]_1^-\to B(\cH)^n$ for which  a reasonable
  operator model theory  and dilation theory  can be developed  for
   the   noncommutative domain    $g(\Omega)$.

 Section 1  contains some preliminaries on  the class $\cM$ of $n$-tuples  of formal
power series $f=(f_1,\ldots,f_n)$ with the {\it model property}. An
$n$-tuple  $f$ has the model property if it is either one of the
following: an $n$-tuple of polynomials with  property $(\cA)$,
 an $n$-tuple of  formal power series  with  $f(0)=0$ and    property
 $(\mathcal{S})$, or
  an $n$-tuple of  free holomorphic functions  with property $(\mathcal{F})$.
  We associate with each $f\in \cM$ a Hilbert space
$\HH^2(f)$ of formal power series and the noncommutative domain
$$
\BB_f(\cH):=\{X=(X_1,\ldots, X_n)\in B(\cH)^n: \ g(f(X))=X \text{
and } \|f(X)\|\leq 1\},
$$
where $g=(g_1,\ldots, g_n)$ is the inverse   of $f$ with respect to
the composition of power series, and the evaluations are
well-defined.

 The
characteristic  function  of  an $n$-tuple $T=(T_1,\ldots, T_n)$ in
the noncommutative domain $\BB_f(\cH)$
  is
  the  operator   $ \Theta_{f,T}:\HH^2(f)\otimes \cD_{f,T^*}\to
\HH^2(f)\otimes \cD_{f,T}$ having  the formal Fourier representation
\begin{equation*}
\begin{split}
  -I_{\HH^2(f)}\otimes f(T)+
\left(I_{\HH^2(f)}\otimes \Delta_{f,T}\right)&\left(I_{\HH^2(f)\otimes \cH}-\sum_{i=1}^n \Lambda_i\otimes f_i(T)^*\right)^{-1}\\
&\quad\qquad\qquad\left[\Lambda_1\otimes I_\cH,\ldots, \Lambda_n\otimes I_\cH \right]
\left(I_{\HH^2(f)}\otimes \Delta_{f,T^*}\right),
\end{split}
\end{equation*}
where $\Lambda_1,\ldots, \Lambda_n$ are the right multiplication
operators  by the power series $f_i$  on the Hardy space $\HH^2(f)$
and $\Delta_{f,T}$, $\Delta_{f,T^*}$ are certain   defect operators,
while  $\cD_{f,T}$,   $\cD_{f,T^*}$ are  the corresponding  defect
spaces  associated with $T\in \BB_f(\cH)$. We remark that the
characteristic  function  is a contractive multi-analytic operator
with respect to
  the
universal model $(M_{Z_1},\ldots, M_{Z_n})$ associated with  the
noncommutative domain $\BB_f$.

In Section 2, we   present  operator models
  for  completely non-coisometric (c.n.c.) $n$-tuples of operators $T:=(T_1,\ldots, T_n)$
   in   noncommutative
  domains
  $\BB_f(\cH)$, generated by by an $n$-tuples of formal power series
  $f=(f_1,\ldots, f_n)$  of class $\cM^b$,
in which the characteristic function $\Theta_{f,T}$ occurs explicitly. More
precisely, we show that  $T:=(T_1,\ldots, T_n)$   is unitarily
equivalent to a c.n.c.
 $n$-tuple ${\bf T}:=({\bf T}_1,\ldots, {\bf T}_n)$ of operators in $\BB_f ({\bf H})$
 on the Hilbert space
$$
{\bf H}:=[ (\HH^2(f)\otimes \cD_{f,T
})\oplus\overline{\Delta_{\Theta_{f,T}} (\HH^2(f)\otimes
\cD_{f,T^*})}]\ominus
\{\Theta_{f,T} x\oplus\Delta_{\Theta_{f,T}} x:\ \ x\in
\HH^2(f)\otimes \cD_{f,T^*}\},
$$
where $~\Delta_{\Theta_{f,T}}:=(I-\Theta_{f,T}^*\Theta_{f,T})^{1/2}$
and
  the operator ${\bf T}_i$  is defined by
$$
{\bf T}_{i}^*[x\oplus\Delta_{\Theta_{f,T}}   y]:= (M_{Z_i}^*\otimes
I_{\cD_{f,T}}) x\oplus D_i^*(\Delta_{\Theta_{f,T}}y), \qquad
i=1,\ldots,n,
$$
for   $x\in \HH^2(f)\otimes \cD_{f,T} $,  $y\in \HH^2(f)\otimes
\cD_{f,T^*}$,   where
$D_i(\Delta_{\Theta_{f,T}}y):=\Delta_{\Theta_{f,T}} (M_{Z_i} \otimes
I_{\cD_{f,T^*}}) y$. Moreover, $~T$ is a pure  $n$-tuple of
operators in $\BB_f(\cH)$ if and only  if
 the characteristic function  $~\Theta_{f,T}~$ is an isometry.
  In this  case, the model reduces to
$$
{\bf H}=\left(\HH^2(f)\otimes \cD_{f,T}\right)\ominus
\Theta_{f,T}(\HH^2(f)\otimes \cD_{f,T^*}),\qquad{\bf T}_{i}^*
x=(M_{Z_i}^*\otimes I_{\cD_{f,T}}) x,\qquad
x\in \bf H.
$$
This   result is used to show that the characteristic function is a
complete unitary invariant for the c.n.c. $n$-tuples of operators in
$\BB_f(\cH)$.
 We
also show that
 any contractive multi-analytic operator with respect to
 $M_{Z_1},\ldots, M_{Z_n}$
   generates a c.n.c.
$n$-tuple of operators  in $\BB_f( {\bf H})$, for an appropriate
Hilbert space ${\bf H}$.

In Section 3, under natural conditions on the $n$-tuple
$f=(f_1,\ldots, f_n)$, we study the $*$-representations of the
$C^*$-algebra $C^*(M_{Z_1},\ldots, M_{Z_n})$ and obtain a Wold type
decomposition for the nondegenerate $*$-representations, where
$(M_{Z_1},\ldots, M_{Z_n})$  is  the universal model associated with
the noncommutative domain $\BB_f$.  We  also show that any $n$-tuple
$T=(T_1,\ldots,T_n)$ of operators
 in the noncommutative domain
 $\BB_f(\cH)$  has  a minimal dilation which is unique up to an
 isomorphism, i.e.,
there is an $n$-tuple $V:=(V_1,\ldots, V_n)$ of operators on a
Hilbert space $\cK\supseteq \cH$ such that
\begin{enumerate}
\item[(i)] $(V_1,\ldots, V_n)\in \BB_f(\cK)$;
\item[(ii)] there  is a $*$-representation
$\pi:C^*(M_{Z_1},\ldots, M_{Z_n})\to B(\cK)$ such that
$\pi(M_{Z_i})=V_i$, $i=1,\ldots,n$;
\item[(iii)]$V_i^*|_\cH=T_i^*$, $i=1,\ldots,n$;
\item[(iv)] $\cK=\bigvee_{\alpha\in \FF_n^+} V_\alpha \cH$.
\end{enumerate}
 A commutant lifting theorem for $\BB_f(\cH)$ (see
Theorem \ref{CLT}) is also provided.

If  $f=(f_1,\ldots, f_n)$ has the model property, we introduce the
Hardy algebra $H^\infty(\BB_f)$ to be  the WOT-closure of all
  noncommutative polynomials in $M_{Z_1},\ldots, M_{Z_n}$ and the
  identity, where
 $(M_{Z_1},\ldots, M_{Z_n})$  is the universal model associated with the
noncommutative domain $\BB_f$. In Section 4, we extend the model
theory   to c.n.c. $n$-tuples of operators  in
 noncommutative
varieties    defined by
$$
\cV_{f,J}(\cH):=\left\{ (T_1,\ldots, T_n)\in \BB_f(\cH):\
\psi(T_1,\ldots,T_n)=0 \ \text{ for any } \ \psi\in J\right\},
$$
for an appropriate evaluation $\psi(T_1,\ldots,T_n)$, and associated
with    $n$-tuples $f$ with the model property and  WOT-closed
two-sided ideals $J$ of the Hardy algebra $H^\infty(\BB_f)$. We also
show that the constrained characteristic function $\Theta_{f,T,J}$
is a complete unitary invariant for  the c.n.c. part of $\cV_{f,J}
(\cH)$.

In  Section 5, we develop a dilation theory for $n$-tuples of
operators $(T_1,\ldots,T_n)$ in the noncommutative domain
$\BB_f(\cH)$, subject to constraints such as
$$(q\circ
f)(T_1,\ldots,T_n)=0,\qquad q\in \cP,$$
 where $\cP$ is a set of
homogeneous noncommutative polynomials.  We show that if
$f=(f_1,\ldots, f_n)$ is an $n$-tuple of formal power series with
the radial  approximation property  and let $B=(B_1,\ldots, B_n)$ be
the universal model
  associated with   the WOT-closed two-sided ideal $J_{\cP\circ f}$
   generated by $q(f(M_Z))$, $q\in \cP$, in
$H^\infty(\BB_f)$, then
    the linear map $\Psi_{f,T,\cP}:\overline{\text{\rm
span}}\{{B}_\alpha { B}_\beta:\ \alpha,\beta\in \FF_n^+\}\to B(\cH)
$ defined by
$$\Psi_{f,T,\cP}({ B}_\alpha { B}_\beta):=T_\alpha T_\beta^*,\qquad \alpha,\beta\in \FF_n^+,
$$
is completely contractive. If  $\cH$ is a separable Hilbert space,
we prove that
  there exists a separable Hilbert space $\cK_\pi$ and a
$*$-representation $\pi:C^*(B_1,\ldots, B_n)\to B(\cK_\pi)$ which
annihilates the compact operators and
$$
\sum_{i=1}^n f_i(\pi(B_1),\ldots, \pi(B_n))f_i(\pi(B_1),\ldots,
\pi(B_n))^*= I_{ \cK_\pi},
$$
such that
\begin{enumerate}
\item[(a)]
$\cH$ can be identified with a $*$-cyclic co-invariant subspace of
$\tilde\cK:=(\cN_{f,J_{\cP\circ f}}\otimes
\overline{\Delta_{f,T}\cH})\oplus \cK_\pi$ under the operators
$$
V_i:=\left[\begin{matrix} B_i\otimes
I_{\overline{\Delta_{f,T}\cH}}&0\\0&\pi(B_i)
\end{matrix}\right],\quad i=1,\ldots,n;
$$
\item[(b)]
$ T_i^*=V_i^*|\cH,\  i=1,\ldots, n$;
\item[(c)] $V:=(V_1,\ldots,V_n)\in \BB_f(\widetilde \cK)$ and
$ (q\circ f)(V)=0$, \  $q\in \cP.$
\end{enumerate}

In Section 6, under the conditions  that $f=(f_1,\ldots, f_n)$ is an
$n$-tuple of power series with the model property and
  $J $ is a WOT-closed two-sided ideal of the
  Hardy algebra $H^\infty(\BB_f)$,  we provide  a Beurling \cite{Be}
type theorem characterizing the invariant subspaces under the
universal $n$-tuple  $(B_1,\ldots,B_n)$ associated with a
noncommutative variety $\cV_{f,J}(\cH)$, and a commutant lifting
theorem \cite{SzF-book} for pure $n$-tuples of operators in
$\cV_{f,J} (\cH)$.

 We
remark that all the results  of this paper   have commutative
versions which can be obtained when  $J$ is the $WOT$-closed
two-sided ideal generated by the commutators $
M_{Z_i}M_{Z_j}-M_{Z_j}M_{Z_i}$, $i,j\in \{1,\ldots, n\}$.  In this
case, if $T:=(T_1,\ldots, T_n)\in \BB_f(\cH)$ is  such that
$$
T_iT_j=T_jT_i, \qquad i,j=1,\ldots, n,
$$
then
 the  characteristic
function of $T$   can be identified with  the multiplier
$M_{\Theta_{f,J,T}}:\HH^2(g({\bf B}_n))\otimes \cD_{f,T^*}\to
\HH^2(g({\bf B}_n))\otimes \cD_{f,T} $ defined by    the operator-valued
analytic function
\begin{equation*}
\begin{split}
\Theta_{f,J,T}(z):=
  - f(T)+
 \Delta_{f,T}\left(I-\sum_{i=1}^n f_i(z) f_i(T)^*\right)^{-1}
\left[f_1(z)I_\cH,\ldots, f_n(z) I_\cH \right] \Delta_{f,T^*},\qquad
z\in g(\BB_n),
\end{split}
\end{equation*}
where
$H^2(g({\bf B}_n))$  is a reproducing kernel  Hilbert space of holomorphic functions on $g({\bf B}_n)$,
${\bf B}_n$ is the open unit ball of $\CC^n$,   and $g=(g_1,\ldots, g_n)$ is the inverse of $f$ with
respect to the composition.

It  would be interesting to see to what extent the results of this
paper and \cite{Po-multi} can be extended to the Muhly-Solel setting
of tensor algebras over $C^*$-correspondences (\cite{MuSo1},
\cite{MuSo2}, \cite{MuSo3}).

\bigskip

\section{  Hilbert spaces  of  formal power series and noncommutative domains }

In this section,  we   recall (see \cite{Po-multi}) some basic facts
regarding the Hilbert spaces  $\HH^2(f)$  and the  noncommutative
domains $\BB_f(\cH)$     associated with $n$-tuples  of formal power
series $f=(f_1,\ldots,f_n)$ with the {\it model property}.

Let  $\FF_n^+$ be the free semigroup with $n$ generators
$g_1,\ldots, g_n$ and the identity $g_0$.  The length of $\alpha\in
\FF_n^+$ is defined by $|\alpha|:=0$ if $\alpha=g_0$ and
$|\alpha|:=k$ if
 $\alpha=g_{i_1}\cdots g_{i_k}$, where $i_1,\ldots, i_k\in \{1,\ldots, n\}$.
 Let
$\CC[Z_1,\ldots, Z_n]$ be the algebra of noncommutative polynomials
 with complex coefficients and noncommuting indeterminates
$Z_1,\ldots, Z_n$. We  say that an $n$-tuple $p=(p_1,\ldots, p_n)$
of
   polynomials
 is invertible in $\CC[Z_1,\ldots, Z_n]^n$  with respect to the
composition if
  there exists an
 $n$-tuple $q=(q_1,\ldots, q_n)$ of
  polynomials such that
 \begin{equation*}
      p\circ q=q\circ p=id .
 \end{equation*}
In this case, we say that $p=(p_1,\ldots, p_n)$ has property
$(\cA)$. We introduce an inner product on $\CC[Z_1,\ldots,Z_n]$ by
setting $\left< p_\alpha, p_\beta\right>:=\delta_{\alpha \beta}$,
$\alpha,\beta\in \FF_n^+$, where $p_\alpha:=p_{i_1}\cdots p_{i_k}$
if $\alpha=g_{i_1}\cdots g_{i_k}\in \FF_n^+$, and $p_{g_0}:=1$. Let
$\HH^2(p)$ be the completion of the linear span  of the
noncommutative polynomials $ p_\alpha$, $\alpha\in \FF_n^+$, with
respect to this inner product.

  Denote by  $B(\cH)$   the algebra of all bounded linear operators on
an infinite dimensional  Hilbert space $\cH$ and let
$\Omega_0\subset B(\cH)^n$ be a set containing a ball $[B(\cH)^n]_r$
for some $r>0$, where
 $$
 [B(\cH)^n]_r:=\left\{ (X_1,\ldots, X_n)\in B(\cH)^n: \
 \left\|X_1X_1^*+\cdots + X_nX_n^*\right\|^{1/2}<r\right\}.
 $$
 We say
that $f:\Omega_0\to B(\cH)$ is a free holomorphic function on
$\Omega_0$ if there are some complex numbers $a_\alpha$, $\alpha\in
\FF_n^+$, such that
$$
f(X)=\sum_{k=0}^\infty \sum_{|\alpha|=k} a_\alpha X_\alpha, \qquad
X=(X_1,\ldots,X_n)\in \Omega_0,
$$
where  the convergence is in the operator norm topology. Here, we
denoted   $X_\alpha:=X_{i_1}\cdots X_{i_k}$ if $\alpha=g_{i_1}\cdots
g_{i_k}\in \FF_n^+$, and $X_{g_0}:=I_\cH$. One can show that  any
free holomorphic function on $\Omega_0$ has a unique representation.
 The algebra $H_{{\bf ball}}$~ of free holomorphic
functions on the open operatorial  $n$-ball of radius one is defined
 as the set of all formal power series $f=\sum_{\alpha\in
\FF_n^+}a_\alpha Z_\alpha$ with radius of convergence $r(f)\geq 1$,
i.e.,
 $\{a_\alpha\}_{\alpha\in \FF_n^+}$ are complex numbers  with
$r(f)^{-1}:=\limsup_{k\to\infty} \left(\sum_{|\alpha|=k}
|a_\alpha|^2\right)^{1/2k}\leq 1.$ In this case, the mapping
$$
[B(\cH)^n]_1\ni (X_1,\ldots, X_n)\mapsto f(X_1,\ldots,
X_n):=\sum_{k=0}^\infty \sum_{|\alpha|=k}
 a_\alpha X_\alpha\in
B(\cH)
$$
 is well-defined, where  the convergence is in the operator norm topology. Moreover, the
 series converges absolutely, i.e.,
  $\sum_{k=0}^\infty \left\|\sum_{|\alpha|=k} a_\alpha
  X_\alpha\right\|<\infty$
  and uniformly  on any ball $[B(\cH)^n]_\gamma$ with $0\leq
  \gamma<1$. More on free holomorphic functions on the unit ball
  $[B(\cH)^n]_1$ can be found in \cite{Po-holomorphic}, \cite{Po-automorphism},  and \cite{Po-holomorphic2}.

The  evaluation of $f=\sum_{\alpha\in \FF_n^+}a_\alpha Z_\alpha$  is
also well-defined  if there exists an $n$-tuple
$\rho=(\rho_1,\ldots, \rho_n)$ of strictly positive numbers such
that
$$\limsup_{k\to\infty} \left(\sum_{|\alpha|=k}
|a_\alpha|\rho_\alpha\right)^{1/k}\leq 1.
$$
 In this case, the series
$f(X_1,\ldots, X_n):=\sum_{k=0}^\infty \sum_{|\alpha|=k}a_\alpha
X_\alpha$ converges absolutely and uniformly on any noncommutative
polydisc

 $$P({\bf r}):=\left\{(X_1,\ldots, X_n)\in B(\cH)^n:\ \|X_j\|\leq r_j, j=1,\ldots,n\right\}$$
of multiradius ${\bf{r}}=(r_1,\ldots, r_n)$ with $r_j<\rho_j$, $j=1,\ldots, n$.

We  remark that, when $(X_1,\ldots, X_n)\in B(\cH)^n$ is a nilpotent
$n$-tuple of operators, i.e., there is $m\geq 1$ such that
$X_\alpha=0$ for all $\alpha\in \FF_n^+$ with $|\alpha|=m$, then
$f(X_1,\ldots,X_n)$ makes  sense for any formal power series $f$.

Let $g=\sum_{k=0}^\infty
\sum_{|\alpha|=k} a_\alpha Z_\alpha$ be a formal power series in
indeterminates  $Z_1,\ldots, Z_n$.
 Denote by $\cC_g(\cH)$ (resp.   $\cC_g^{SOT}(\cH)$) the set of all
 $Y:=(Y_1,\ldots, Y_n)\in B(\cH)^n$ such that the series
$$
g(Y_1,\ldots, Y_n):=\sum_{k=0}^\infty \sum_{|\alpha|=k} a_\alpha Y_\alpha
$$
 is norm (resp.   SOT) convergent. These sets are called  sets of norm (resp.   SOT)
 convergence for the power series $g$.
We   introduce the set $\cC_g^{rad}(\cH)$  of all $Y:=(Y_1,\ldots,
Y_n)\in B(\cH)^n$ such that there exists $\delta\in (0,1)$ with the
property that $rY\in \cC_g(\cH)$ for any $r\in (\delta, 1)$ and
$$
\widehat{g}(Y_1,\ldots, Y_n):=\text{\rm SOT-}\lim_{r\to 1}\sum_{k=0}^\infty \sum_{|\alpha|=k} a_\alpha r^{|\alpha|}Y_\alpha
$$
  exists in the strong operator topology.
Note that
$  \cC_g(\cH)\subseteq \cC_g^{SOT}$ and
 $\cC_g^{rad}(\cH)\subseteq \overline{\cC_g(\cH)}^{SOT}$.

Consider    an $n$-tuple of
 formal power series $f=(f_1,\ldots, f_n)$ in indeterminates  $Z_1,\ldots, Z_n$
  with    the Jacobian
$$
\det J_f(0):=\det \left[\left. \frac{\partial f_i} {\partial
Z_j}\right|_{Z=0}\right]_{i,j=1}^n\neq 0.
$$
   As shown in \cite{Po-multi},
the set
 $\{f_\alpha\}_{\alpha\in \FF_n^+}$    is  linearly independent
in ${\bf S}[Z_1,\ldots,Z_n]$, the algebra of all formal powers
series in $Z_1,\ldots,Z_n$. We introduce an inner product on the
linear span of $\{f_\alpha\}_{\alpha\in \FF_n^+}$ by setting
$$
\left< f_\alpha, f_\beta\right>:=\begin{cases} 1\  \ &\text{if } \alpha=\beta\\
0\  \ &\text{if }  \alpha\neq \beta
 \end{cases} \quad \text{ for } \quad \alpha,\beta\in \FF_n^+.
$$
 Let $\HH^2(f)$ be the
completion of the linear span  of   $  \{f_\alpha\}_{\alpha\in
\FF_n^+}$ with respect to this inner product. Assume now that
$f(0)=0$. As seen in \cite{Po-multi},  $f$ is not  a {\it right zero
divisor} with respect to the composition of power series, i.e.,
there is no non-zero formal power series $G\in
\mathbf{S}[Z_1,\ldots, Z_n]$ such that
 $G\circ f = 0$. Consequently,
the elements of  $\HH^2(f)$  can be seen as formal power series in
${\bf S}[Z_1,\ldots,Z_n]$  of  the form $\sum_{\alpha\in \FF_n^+}
a_\alpha f_\alpha$, where $\sum_{\alpha\in \FF_n^+}
|a_\alpha|^2<\infty$.

    Let $f=(f_1,\ldots, f_n)$  be  an $n$-tuple of
 formal power series  in $Z_1,\ldots, Z_n$ such that $f(0)=0$. We  say that $f$ has  property $(\mathcal{S})$
 if the following conditions
  hold.
\begin{enumerate}
\item[($\mathcal{S}_1$)] The $n$-tuple $f$ has nonzero radius of convergence  and
$ \det  J_f(0)\neq 0. $
\item[($\mathcal{S}_2$)]  The indeterminates   $Z_1,\ldots, Z_n$\ are in the Hilbert space
$\HH^2(f)$
 and  each left multiplication operator
  $M_{Z_i}:\HH^2(f)\to \HH^2(f)$  defined by
   $$
   M_{Z_i}\psi:=Z_i\psi, \qquad \psi\in \HH^2(f),
  $$
is a bounded multiplier of $\HH^2(f)$.
\item[($\mathcal{S}_3$)]
The left multiplication operators $M_{f_j}:\HH^2(f)\to \HH^2(f)$,
$M_{f_j}\psi=f_j \psi$,  satisfy the equations
\begin{equation*}
  M_{f_j}=f_j(M_{Z_1},\ldots, M_{Z_n}),\qquad j=1,\ldots,n,
\end{equation*}
where  $(M_{Z_1},\ldots, M_{Z_n})$ is either in the convergence set
$\cC_f^{SOT}(\HH^2(f))$ or $\cC_f^{rad}(\HH^2(f))$.
\end{enumerate}
Note that if $f$ is an $n$-tuple of noncommutative polynomials, then
the condition $(\cS_3)$ is always satisfied. We remark that, when
$(M_{Z_1},\ldots, M_{Z_n})$ is in the set $\cC_f^{rad}(\HH^2(f))$,
then  the condition $(\mathcal{S}_3)$ should be understood as
$$
M_{f_j}=\widehat{f}_j(M_{Z_1},\ldots, M_{Z_n}):=\text{\rm
SOT-}\lim_{r\to 1} f_j(rM_{Z_1},\ldots, rM_{Z_n}),\qquad
j=1,\ldots,n.
$$

Now, we introduce the class of $n$-tuples of  free holomorphic
function with property $(\mathcal{F})$. Let
$\varphi=(\varphi_1,\ldots, \varphi_n)$   be an $n$-tuple of  free
holomorphic functions on $[B(\cH)^n]_\gamma$, $\gamma>0$,  with
range in $[B(\cH)^n]_1$  and such that  $\varphi$ is
 not a   right zero divisor with respect to the  composition
  with free holomorphic functions on $[B(\cH)^n]_1$. Consider
the Hilbert space of free holomorphic functions
  $$
  \HH^2(\varphi):=\{G\circ \varphi: \ G\in  H_{\bf ball}^2\},
  $$
  with the inner product
\begin{equation*}
\label{inner} \left<F\circ \varphi, G\circ
\varphi\right>_{\HH^2(\varphi)}:=  \left<F, G\right>_{H_{\bf
ball}^2}.
\end{equation*}
  We recall that the noncommutative Hardy space
$H^2_{\bf ball}$ is the Hilbert space of all free holomorphic
functions on $ [B(\cH)^n]_1 $
 of the
form
$$f(X_1,\ldots, X_n)=\sum_{k=0}^\infty \sum_{|\alpha|=k}
a_\alpha  X_\alpha, \qquad   \sum_{\alpha\in \FF_n^+}|a_\alpha|^2<\infty,
$$ with the inner product
$ \left< f,g\right>:=\sum_{k=0}^\infty \sum_{|\alpha|=k}a_\alpha
{\overline b}_\alpha,$ where  $g=\sum_{k=0}^\infty \sum_{|\alpha|=k}
b_\alpha  X_\alpha$ is
 another free holomorphic function  in $H^2_{\bf ball}$.
We say that $\varphi$ has property $(\cF)$ if  the following
conditions hold.
\begin{enumerate}
\item[($\mathcal{F}_1$)]  The $n$-tuple  $\varphi=(\varphi_1,\ldots, \varphi_n)$
 has  the range in $[B(\cH)^n]_1$  and it is
 not a   right zero divisor with respect to the  composition
  with free holomorphic functions on $[B(\cH)^n]_1$.
\item[($\mathcal{F}_2$)]  The coordinate functions $X_1,\ldots, X_n$ on $[B(\cH)^n]_\gamma$
 are contained in $\HH^2(\varphi)$ and
  the left multiplication
 by $X_i$
is a bounded multiplier of $\HH^2(\varphi)$.

\item[($\mathcal{F}_3$)] For each $i=1,\ldots,n$,
the left multiplication operator $M_{\varphi_i}:\HH^2(\varphi)\to
\HH^2(\varphi)$ satisfies the equation
$$
M_{\varphi_i}=\varphi_i(M_{Z_1},\ldots, M_{Z_n}),
$$
 where  $(M_{Z_1},\ldots, M_{Z_n})$  is either in the convergence set
$\cC_\varphi^{SOT}(\HH^2(\varphi))$ or
$\cC_\varphi^{rad}(\HH^2(\varphi))$.
\end{enumerate}
We mention that if  $\varphi$ is an $n$-tuple of noncommutative
polynomials, then the condition $(\cF_3)$ is always satisfied.

An $n$-tuple
 $f=(f_1,\ldots, f_n)$ of formal power series is said to have the {\it model
 property} if it
   is either one of the following:
 \begin{enumerate}
 \item[(i)]   an $n$-tuple of polynomials with  property $(\cA)$;

 \item[(ii)] an $n$-tuple of  formal power series  with  $f(0)=0$ and    property
 $(\mathcal{S})$;

 \item[(iii)]  an $n$-tuple of  free holomorphic functions  with property $(\mathcal{F})$.
\end{enumerate}
  We denote by $\cM$ the set
of all $n$-tuples $f$ with the model property. For several examples
of formal power series with the model property we refer the reader
to \cite{Po-multi}.

Let $f=(f_1,\ldots, f_n)$ have the model property and let
$g=(g_1,\ldots, g_n)$ be the $n$-tuple of power series having the
representations
$$
g_i:=\sum_{k=0}^\infty\sum_{\alpha\in \FF_n^+, |\alpha|=k}
a_\alpha^{(i)}Z_\alpha,  \qquad i=1,\ldots,n,
$$
where the  sequence of complex numbers
$\{a_\alpha^{(i)}\}_{\alpha\in \FF_n^+}$ is uniquely defined by the condition
$g\circ f=id$. We say that an $n$-tuple of operators $X=(X_1,\ldots,
X_n)\in B(\cH)^n$ satisfies the equation   $g(f(X))=X$ if
 either
one of the following conditions holds:
\begin{enumerate}
     \item[(a)] $X\in \cC_f^{SOT}(\cH)$  and either
     \begin{equation*}
X_i=\sum_{k=0}^\infty\sum_{\alpha\in \FF_n^+, |\alpha|=k}
a_\alpha^{(i)}[f(X)]_\alpha,  \qquad i=1,\ldots,n,
\end{equation*}
  where the  convergence of the  series  is in the strong operator topology, or
  $$
 X_i= \text{\rm SOT-}\lim_{r\to 1} \sum_{k=0}^\infty\sum_{\alpha\in \FF_n^+, |\alpha|=k}
a_\alpha^{(i)}r^{|\alpha|}[{f}(X)]_\alpha ,\qquad i=1,\ldots,n;
 $$
 \item[(b)] $X\in \cC_f^{rad}(\cH)$  and either
 \begin{equation*}
X_i=\sum_{k=0}^\infty\sum_{\alpha\in \FF_n^+, |\alpha|=k}
a_\alpha^{(i)}[\widehat{f}(X)]_\alpha,  \qquad i=1,\ldots,n,
\end{equation*}
  where the  convergence of the  series  is in the strong operator topology, or
 $$
 X_i= \text{\rm SOT-}\lim_{r\to 1} \sum_{k=0}^\infty\sum_{\alpha\in \FF_n^+, |\alpha|=k}
a_\alpha^{(i)}r^{|\alpha|}[\widehat{f}(X)]_\alpha ,\qquad
i=1,\ldots,n.
 $$
  \end{enumerate}

  We define the noncommutative domain associated with
$f$  by setting
$$
\BB_f(\cH):=\{X=(X_1,\ldots, X_n)\in B(\cH)^n: \ g(f(X))=X \text{
and } \|f(X)\|\leq 1\},
$$
where $g=(g_1,\ldots, g_n)$ is the inverse   of $f$ with respect to
the composition of power series, and the evaluations are
well-defined as above. Note that  the condition $ g(f(X))=X$ is
 automatically satisfied when $f$ is an $n$-tuple of polynomials with  property
 $(\cA)$.

\bigskip

\section{Characteristic functions  and  models for $n$-tuples of operators in  $\BB_f^{cnc}(\cH)$}

In this section, we   present   models
  for  completely non-coisometric (c.n.c.) $n$-tuples of operators in   noncommutative
  domains
  $\BB_f(\cH)$, generated by   $n$-tuples of formal power series
  $f=(f_1,\ldots, f_n)$  of class $\cM^b$,
in which the characteristic function occurs explicitly. This  is
used to show that the characteristic function is a complete  unitary
invariant for the c.n.c. $n$-tuples of operators in $\BB_f(\cH)$. We
also show that
 any contractive multi-analytic operator with respect to
 $M_{Z_1},\ldots, M_{Z_n}$
   generates a c.n.c.
$n$-tuple of operators   ${\bf T}:=({\bf T}_1,\ldots, {\bf T}_n)\in
\BB_f( {\bf H})$.

First, we need a few definitions. Let $f=(f_1,\ldots, f_n)$ be   an
$n$-tuple of formal power series with the model property.
      We say that  $f$ has the radial
approximation property, and write  $f\in \cM_{rad}$,   if there is
$\delta\in (0,1)$ such that $(rf_1,\ldots, rf_n)$ has the model
property  for any $r\in (\delta, 1]$. Denote by  $\cM^{||}$  the set
of all  formal power series $f=(f_1,\ldots, f_n)$ having  the model
property  and such that the universal model $(M_{Z_1},\ldots,
M_{Z_n})$ associated with the noncommutative domain $\BB_f$ is in
the set of norm-convergence (or radial norm-convergence) of $f$. We
also introduce  the class $\cM_{rad}^{||}$ of all formal power
series $f=(f_1,\ldots, f_n)$ with the property that there is
$\delta\in (0,1)$ such that $rf\in \cM^{||}$ for any $r\in (\delta,
1]$. We recall that in all the examples presented in
\cite{Po-multi}, the corresponding  $n$-tuples $f=(f_1,\ldots, f_n)$
are in the class
 $\cM_{rad}^{||}$. Moreover, the  $n$-tuples   of polynomials
with property $(\cA)$ are also in  the class $\cM_{rad}^{||}$.

Let $f=(f_1,\ldots, f_n)$ be   an $n$-tuple of formal power series
with the model property  and assume that $f_i$ has the
representation $f_i(Z_1,\ldots,Z_n)=\sum_{\alpha\in \FF_n^+}
a_\alpha^{(i)} Z_\alpha$. We say that $f$ is in the class $\cM^b$ if
either one of the following conditions holds:
\begin{enumerate}
\item[(i)]   the $n$-tuple $(M_{Z_1},\ldots, M_{Z_n})$ is  in the convergence set
$\cC_f^{SOT}(\HH^2(f))$  and  $$\sup_{m\in
\NN}\left\|\sum_{|\alpha|\leq m} a_\alpha^{(i)}
M_{Z_\alpha}\right\|<\infty,\qquad i=1,\ldots,n;
$$
\item[(ii)]  the $n$-tuple  $(M_{Z_1},\ldots, M_{Z_n})$ is  in the convergence set
  $\cC_f^{rad}(\HH^2(f))$ and $$\sup_{r\in
  [0,1)}\left\|\sum_{k=0}^\infty \sum_{|\alpha|=k} a_\alpha^{(i)}
  r^{|\alpha|}
M_{Z_\alpha}\right\|<\infty, \qquad i=1,\ldots,n.
$$
\end{enumerate}
Note that $\cM^{||}_{rad}\subset\cM^{||}\subset \cM^{b}\subset \cM$.

 We recall that the noncommutative domain
associated with $f\in \cM$  is
$$
\BB_f(\cH):=\{X=(X_1,\ldots, X_n)\in B(\cH)^n: \ g(f(X))=X \text{
and } \|f(X)\|\leq 1\},
$$
where $g$ is the inverse power series of $f$ with respect to the
composition.  We say that $T:=(T_1,\ldots, T_n)\in B(\cH)^n$ is  a
pure $n$-tuple  of operators in  $\BB_f(\cH)$ if
$$
\text{\rm SOT-}\lim_{k\to \infty} \sum_{\alpha\in \FF_n,\,
|\alpha|=k} [f(T)]_\alpha [f(T)]_\alpha^*=0.
$$
The set of all pure elements of $\BB_f(\cH)$ is denoted by
$\BB^{pure}_f(\cH)$. An $n$-tuple $T \in \BB_f(\cH)$ is called {\it
completely  non-coisometric }(c.n.c)  if there is no vector $h\in
\cH$, $h\neq 0$, such that
$$
\left< \Phi_{f,T}^m(I)h,h\right>=\|h\|^2 \quad \text{ for any } \
m=1,2,\ldots,
$$
where the positive linear mapping $\Phi_{f,T}:B(\cH\to B(\cH)$ is
defined  by $ \Phi_{f,T}(Y):=\sum\limits_{i=1}^n f_i(T) Yf_i(T)^*. $
The set of all c.n.c. elements of $\BB_f(\cH)$ is denoted by
$\BB^{cnc}_f(\cH)$. Note that
$$
  \BB^{pure}_f(\cH)\subseteq \BB^{cnc}_f(\cH)\subseteq \BB_f(\cH).
$$

Let $H_n$ be an $n$-dimensional complex  Hilbert space with
orthonormal
      basis
      $e_1$, $e_2$, $\dots,e_n$, where $n\in\{1,2,\dots\}$.
       We consider the full Fock space  of $H_n$ defined by
      $$F^2(H_n):=\CC 1\oplus \bigoplus_{k\geq 1} H_n^{\otimes k},$$
      where   $H_n^{\otimes k}$ is the (Hilbert)
      tensor product of $k$ copies of $H_n$.
Define the left (resp.~right) creation
      operators  $S_i$ (resp.~$R_i$), $i=1,\ldots,n$, acting on $F^2(H_n)$  by
      setting
      $$
       S_i\varphi:=e_i\otimes\varphi, \qquad  \varphi\in F^2(H_n),
      $$
       (resp.~$
       R_i\varphi:=\varphi\otimes e_i$).
The noncommutative disc algebra $\cA_n$ (resp.~$\cR_n$) is the norm
closed algebra generated by the left (resp.~right) creation
operators and the identity. The   noncommutative analytic Toeplitz
algebra $F_n^\infty$ (resp.~$R_n^\infty$)
 is the the weakly
closed version of $\cA_n$ (resp.~$\cR_n$). These algebras were
introduced in \cite{Po-von} in connection with a noncommutative
version of the classical  von Neumann inequality \cite{von}.

 A free holomorphic function $g$ on $[B(\cH)^n]_1$ is
bounded if $
 \|g\|_\infty:=\sup  \|g(X)\|<\infty,
  $
where the supremum is taken over all $X\in [B(\cH)^n]_1$ and $\cH$
is an infinite dimensional Hilbert space. Let $H^\infty_{{\bf
ball}}$ be the set of all bounded free holomorphic functions and let
$A_{{\bf ball}}$ be the set of all elements $f\in H^\infty_{{\bf
ball}}$ such that the mapping
$$[B(\cH)^n]_1\ni (X_1,\ldots, X_n)\mapsto g(X_1,\ldots, X_n)\in B(\cH)$$
 has a continuous extension to the closed  ball $[B(\cH)^n]^-_1$.
We  showed in \cite{Po-holomorphic} that $H^\infty_{{\bf ball}}$ and
$A_{{\bf ball}}$ are Banach algebras under pointwise multiplication
and the norm $\|\cdot \|_\infty$. The noncommutative Hardy space
  $H_{\text{\bf ball}}^\infty $   can be identified to  the noncommutative analytic
Toeplitz algebra $F_n^\infty $ . More precisely, a bounded free
holomorphic function $g$ is uniquely determined by its {\it (model)
boundary function} $\widetilde g\in F_n^\infty$ defined by
$\widetilde g:=\text{\rm SOT-}\lim_{r\to 1} g(rS_1,\ldots, rS_n). $
Moreover, $g$ is  the noncommutative Poisson transform
\cite{Po-poisson} of $\widetilde g$ at $X \in [B(\cH)^n]_1$, i.e., $
g(X )=P_X[\widetilde g\otimes I].$ Similar results hold for bounded free
holomorphic functions on the noncommutative ball
$[B(\cH)^n]_\gamma$, $\gamma>0$.

The next result provides a characterization for  the $n$-tuples of
formal power series $f$ with property $(\cS)$ which are in $\cM^b$.
A similar result holds if $f$ has property  $(\cF)$.

\begin{lemma}
\label{charact-bound} Let  $f=(f_1,\ldots, f_n)$ be   an $n$-tuple
of
 formal power series with  $f(0)=0$  and $\det J_f(0)\neq 0$.  Assume that $f_i$ has the
representation $f_i(Z_1,\ldots,Z_n)=\sum_{\alpha\in \FF_n^+}
a_\alpha^{(i)} Z_\alpha$ and let
 $g=(g_1,\ldots,g_n)$ be
    the  inverse  of $f$   under the composition.  Then  $f$ is  in the class
    $\cM^{b}$ if and only if  each $g_i$ is a bounded free holomorphic function on $[B(\cH)^n]_1$ and
    either one of the following conditions holds:
  \begin{enumerate}
  \item[(i)]

 $
 S_i=\text{\rm SOT-}\lim_{m\to\infty}\sum_{|\alpha|\leq m} a_\alpha^{(i)}
\widetilde{g}_{\alpha}$
 and $$\sup_{m\in
\NN}\left\|\sum_{|\alpha|\leq m} a_\alpha^{(i)}
\widetilde{g}_{\alpha}\right\|<\infty,\qquad i=1,\ldots,n;$$
 \item[(ii)]
  $S_i=\text{\rm SOT-}\lim_{r\to 1}\sum_{k=0}^\infty \sum_{|\alpha|=k} a_\alpha^{(i)}
  r^{|\alpha|}
\widetilde{g}_{\alpha}$  and
 $$\sup_{r\in
  [0,1)}\left\|\sum_{k=0}^\infty \sum_{|\alpha|=k} a_\alpha^{(i)}
  r^{|\alpha|}
\widetilde{g}_{\alpha}\right\|<\infty, \qquad i=1,\ldots,n, $$ where
the series converges in the operator norm topology.
\end{enumerate}
\end{lemma}
\begin{proof}

Assume that $f$ has the property $(\cS)$ and let $g=(g_1,\ldots,
 g_n)$ be its inverse with respect to the composition.
Let $U:\HH^2(f)\to F^2(H_n)$ be the unitary operator defined by
$U(f_\alpha):=e_\alpha$, $\alpha\in \FF_n^+$. Note that
$Z_i=\sum_{\alpha\in \FF_n^+} b_\alpha^{(i)}
f_\alpha=U^{-1}(\varphi_i)$ for some coefficients $b_\alpha^{(i)}$
such that $\varphi:=\sum_{\alpha\in \FF_n^+}b_\alpha^{(i)}
e_\alpha\in F^2(H_n)$. Note
  that $M_{Z_i}$ is a bounded left multiplier of $\HH^2(f)$ if and only if $\varphi_i$ is a
  bounded left
   multiplier of $F^2(H_n)$.
 Moreover, $M_{Z_i}=U^{-1} \varphi_i(S_1,\ldots,S_n) U$, where
   $\varphi_i(S_1,\ldots,S_n)$ is   in  the noncommutative Hardy algebra $F_n^\infty$ and has the
    Fourier
    representation  $\sum_{\alpha\in \FF_n^+} b_\alpha^{(i)} S_\alpha$.
     According to Theorem 3.1 from \cite{Po-holomorphic}, we deduce that
      $g_i=\sum_{\alpha\in \FF_n^+} a_\alpha^{(i)} Z_\alpha$ is a bounded free
      holomorphic function on the unit ball $[B(\cH)^n]_1$ and  has its model boundary
       function $\widetilde{g_i}=\varphi_i(S_1,\ldots, S_n)$.
On the other hand, note that
  the left multiplication operator
$M_{f_j}:\HH^2(f)\to \HH^2(f)$ defined by
$$
M_{f_j}\left(\sum_{\alpha\in \FF_n^+} c_\alpha
f_\alpha\right)=\sum_{\alpha\in \FF_n^+}
 c_\alpha f_jf_\alpha,\qquad \sum_{\alpha\in \FF_n^+}
 |c_\alpha|^2<\infty,
$$
satisfies the equation
\begin{equation*}
  M_{f_j}=U^{-1} S_j U,\qquad j=1,\ldots,n,
\end{equation*}
where $S_1,\ldots, S_n$ are the left creation operators on
$F^2(H_n)$.   Since   $M_{Z_i}=U^{-1} \widetilde{g}_i U$, where
$\widetilde{g}_i$ is the model boundary function of $g_i\in
H^\infty_{\bf ball}$, it is easy to see that the relation $
M_{f_j}=f_j(M_{Z_1},\ldots, M_{Z_n})$ for  $ j=1,\ldots,n $   is
equivalent to the fact that
 the model boundary  function $\widetilde{g}=(\widetilde{g}_1,\ldots, \widetilde{g}_n)$
  satisfies either one of the following conditions:
     \begin{enumerate}
     \item[(a)] $\widetilde{g}$ is  in $\cC_f^{SOT}(\HH^2(f))$  and
 $
 S_i=f_i(\widetilde{g}_1,\ldots, \widetilde{g}_n)$, $ i=1,\ldots,n;
 $
 \item[(b)] $\widetilde{g}$ is  in $\cC_f^{rad}(\HH^2(f))$  and
 $
 S_i= \text{\rm SOT-}\lim_{r\to 1} f_j(r\widetilde{g}_1,\ldots, r\widetilde{g}_n)$ for
 $i=1,\ldots,n$.
\end{enumerate}
Now, one can easily see that $f\in \cM^{b}$ if and only if the
conditions in  the lemma hold.
\end{proof}

\begin{proposition} \label{ball} Let  $f=(f_1,\ldots, f_n)$ be   an $n$-tuple  of
 formal power series in the class $\cM^{b}$ and let  $g=(g_1,\ldots,g_n)$ be
    its inverse   under the composition.  Then   the set
    $\BB^{cnc}_f(\cH)$
  coincides with the image of all c.n.c.   row contractions under $g$.
  Moreover, $\BB_f^{pure}(\cH)=g([B(\cH)^n]_1^{pure})$.
  \end{proposition}
\begin{proof}
 Set $[B(\cH)^n]_1^{cnc}:=\{ X\in [B(\cH)^n]_1^-:\ X \text{ is  a c.n.c. row contraction}\}$
and note that $\BB^{cnc}_f(\cH)\subseteq \{g(Y): \ Y\in
[B(\cH)^n]_1^{cnc} \}$. To prove the reversed
 inclusion let $W=g(Y)$, where $Y\in [B(\cH)^n]_1^{cnc}$. Assume that $f_i$ has the
representation $f_i(Z_1,\ldots,Z_n)=\sum_{\alpha\in \FF_n^+}
a_\alpha^{(i)} Z_\alpha$.
  Since $f\in \cM^b$,
using  Lemma \ref{charact-bound}  and  the fact that the
noncommutative Poisson transform $P_Y$ is $SOT$-continuous on
bounded sets (since
  $Y$ is a c.n.c. row contraction),   we deduce that either
  \begin{equation*}
  \begin{split}
Y_i&=P_Y[S_i\otimes I]=\text{\rm
SOT-}\lim_{m\to\infty}\sum_{|\alpha|\leq m} a_\alpha^{(i)}
P_Y[\widetilde{g}_{\alpha}\otimes I]\\
&=\text{\rm SOT-}\lim_{m\to\infty}\sum_{|\alpha|\leq m}
a_\alpha^{(i)}  [g(Y)]_{\alpha}
  \end{split}
  \end{equation*}
or
\begin{equation*}
\begin{split}
Y_i&=P_Y[S_i\otimes I]=\text{\rm SOT-}\lim_{r\to
1}\sum_{k=0}^\infty\sum_{|\alpha|= k} a_\alpha^{(i)} r^{|\alpha|}
P_Y[\widetilde{g}_{\alpha}\otimes I]\\
&=\text{\rm SOT-}\lim_{r\to 1}\sum_{k=0}^\infty\sum_{|\alpha|= k}
a_\alpha^{(i)} r^{|\alpha|} [g(Y)]_{\alpha}.
  \end{split}
  \end{equation*}
Therefore, we obtain $Y_i=f_i(g_1(Y),\ldots, g_n(Y))$ for
$i=1,\ldots,n$. Consequently, we have
 $f(g(Y))=Y$ which  implies  that
 $f(W)=f(g(Y))=Y$ and $g(f(W))=g(Y)=W$, and  shows that $W\in \BB_{f}^{cnc}(\cH)$.
Therefore, $\BB_{f}^{cnc}(\cH)= g([B(\cH)^n]_1^{cnc})$, the function
$g$ is one-to-one on $[B(\cH)^n]_1^{cnc}$, and $f$ is its inverse on
$\BB_{f}^{cnc}(\cH)$.  Consequently, since $\BB_{f}^{pure}(\cH)\subset \BB_{f}^{cnc}(\cH)$, we deduce that $\BB_f^{pure}(\cH)=g([B(\cH)^n]_1^{pure})$.
The proof is complete.
\end{proof}

For simplicity, throughout this paper,
 $T:=[T_1,\ldots, T_n]$   denotes either the
 $n$-tuple $(T_1,\ldots, T_n)$  of bounded linear operators on a Hilbert space
  $\cH$ or the row operator matrix $[T_1 ~\cdots ~T_n]$  acting from $\cH^{(n)} $ to $\cH$,
  where $\cH^{(n)}:=\oplus_{i=1}^n \cH$  is the direct sum of $n$ copies of $\cH$.
Assume that $T:=[T_1,\ldots, T_n]$ is a  row contraction, i.e.,
$$
T_1T_1^*+\cdots +T_nT_n^*\leq I.
$$
 The defect operators of $T$ are
$$
\Delta_{T}:=\left(I_\cH-\sum_{i=1}^n T_i T_i^*\right)^{1/2}\in
B(\cH)\quad \text{ and } \quad
\Delta_{T^*}:=(I_{\cH^{(n)}}-T^*T)^{1/2}\in B(\cH^{(n)}),
$$
 and the defect spaces of $T$ are defined by
$$
\cD_T:=\overline{\Delta_{T}\cH}\quad  \text{ and }\quad
\cD_{T^*}:=\overline{\Delta_{T^*} \cH^{(n)}}.
$$
We recall   that the characteristic function of  a row contraction
$T:=[T_1,\ldots, T_n]$   is the multi-analytic operator
$\Theta_T: F^2(H_n)\otimes \cD_{T^*}\to F^2(H_n)\otimes \cD_T$ with
the formal Fourier representation
\begin{equation*}
  -I\otimes T+
\left(I\otimes \Delta_{T}\right)\left(I -\sum_{i=1}^n R_i\otimes
T_i^*\right)^{-1} \left[R_1\otimes I_\cH,\ldots, R_n\otimes I_\cH
\right] \left(I\otimes \Delta_{T^*}\right),
 \end{equation*}
where $R_1,\ldots, R_n$ are the right creation operators on the full
Fock space $F^2(H_n)$. The characteristic  function associated with
an arbitrary row contraction $T:=[T_1,\ldots, T_n]$, \ $T_i\in
B(\cH)$, was introduce in \cite{Po-charact} (see \cite{SzF-book} for
the classical case $n=1$) and it was proved to be  a complete
unitary invariant for completely non-coisometric (c.n.c.) row
contractions.

  Now, let $f=(f_1,\ldots, f_n)$ be an $n$-tuple of formal power series with the model property. The characteristic  function  of  an $n$-tuple $T=(T_1,\ldots, T_n)\in \BB_f(\cH)$
  was introduced in \cite{Po-multi} as
  the multi-analytic operator  with respect to $M_{Z_1},\ldots,
M_{Z_n}$,
$$
\Theta_{f,T}:\HH^2(f)\otimes \cD_{f,T^*}\to \HH^2(f)\otimes
\cD_{f,T},
$$
with the formal Fourier representation
\begin{equation*}
\begin{split}
  -I_{\HH^2(f)}\otimes f(T)+
\left(I_{\HH^2(f)}\otimes \Delta_{f,T}\right)&\left(I_{\HH^2(f)\otimes \cH}-\sum_{i=1}^n \Lambda_i\otimes f_i(T)^*\right)^{-1}\\
&\qquad \qquad \quad \left[\Lambda_1\otimes I_\cH,\ldots,
\Lambda_n\otimes I_\cH \right] \left(I_{\HH^2(f)}\otimes
\Delta_{f,T^*}\right),
\end{split}
\end{equation*}
where $\Lambda_1,\ldots, \Lambda_n$ are the right multiplication
operators  by the power series $f_i$  on the Hardy space $\HH^2(f)$
and the defect operators  associated with  $T:=(T_1,\ldots, T_n)\in
\BB_f(\cH)$ are
\begin{equation*}
\Delta_{f,T}:=\left( I_\cH-\sum_{i=1}^n
f_i(T)f_i(T)^*\right)^{1/2}\in B(\cH) \quad \text{ and }\quad
\Delta_{f,T^*}:=(I-f(T)^*f(T))^{1/2}\in B(\cH^{(n)}),
\end{equation*}
while the defect spaces are $\cD_{f,T}:=\overline{\Delta_{f,T}\cH}$
and $\cD_{f,T^*}:=\overline{\Delta_{f,T^*}\cH^{(n)}}$. We recall
that a  bounded operator  $\Phi: \HH^2(f)\otimes \cK_1\to
\HH^2(f)\otimes \cK_2$  is   multi-analytic   with respect to
$M_{Z_1},\ldots, M_{Z_n}$ if $\Phi(M_{Z_i}\otimes
I_{\cK_1})=(M_{Z_i}\otimes I_{\cK_2})\Phi$ for any $i=1,\ldots,n$.

 In what follows, we   present a model
  for the $n$-tuples of operators in $\BB^{cnc}_f(\cH)$
in which the characteristic function occurs explicitly.
\begin{theorem}\label{funct-model1} Let  $f=(f_1,\ldots, f_n)$ be   an $n$-tuple  of
 formal power series in the class $\cM^b$ and  let $(M_{Z_1},\ldots, M_{Z_n})$  be the universal
model associated with  the noncommutative domain $\BB_f$.  Every
$n$-tuple of operators $T:=(T_1,\ldots, T_n)$ in $\BB_f^{cnc}(\cH)$
is unitarily equivalent to an
 $n$-tuple ${\bf T}:=({\bf T}_1,\ldots, {\bf T}_n)$ in $\BB_f^{cnc}({\bf H})$
 on the Hilbert space
$$
{\bf H}:=[ (\HH^2(f)\otimes \cD_{f,T
})\oplus\overline{\Delta_{\Theta_{f,T}} (\HH^2(f)\otimes
\cD_{f,T^*})}]\ominus
\{\Theta_{f,T} x\oplus\Delta_{\Theta_{f,T}} x:\ \ x\in
\HH^2(f)\otimes \cD_{f,T^*}\},
$$
where $~\Delta_{\Theta_{f,T}}:=(I-\Theta_{f,T}^*\Theta_{f,T})^{1/2}$
and
  the operator ${\bf T}_i$  is defined by
$$
{\bf T}_{i}^*[x\oplus\Delta_{\Theta_{f,T}}   y]:= (M_{Z_i}^*\otimes
I_{\cD_{f,T}}) x\oplus D_i^*(\Delta_{\Theta_{f,T}}y), \qquad
i=1,\ldots,n,
$$
for   $x\in \HH^2(f)\otimes \cD_{f,T} $,  $y\in \HH^2(f)\otimes
\cD_{f,T^*}$,   where
$D_i(\Delta_{\Theta_{f,T}}y):=\Delta_{\Theta_{f,T}} (M_{Z_i} \otimes
I_{\cD_{f,T^*}}) y$.

 Moreover, $~T$ is a pure  $n$-tuple of operators in $\BB_f(\cH)$ if and only  if
 the characteristic function  $~\Theta_{f,T}~$ is an isometry.
  In this  case the model reduces to
$$
{\bf H}=\left(\HH^2(f)\otimes \cD_{f,T}\right)\ominus
\Theta_{f,T}(\HH^2(f)\otimes \cD_{f,T^*}),\qquad{\bf T}_{i}^*
x=(M_{Z_i}^*\otimes I_{\cD_{f,T}}) x,\qquad
x\in \bf H.
$$
\end{theorem}
\begin{proof} If  $T:=(T_1,\ldots, T_n)$ is in $\BB_f(\cH)^{cnc}$, then
the $n$-tuple $f(T):=(f_1(T),\ldots, f_n(T))$ is a  c.n.c row
contraction. According to
 \cite{Po-charact},  $f(T)$ is unitarily equivalent to a row
 contraction  ${\bf A}:=({\bf A}_1,\ldots, {\bf A}_n)$ on  the Hilbert
 space
 $$
\widetilde{\cH}:=[ (F^2(H_n)\otimes
\cD_{f(T)})\oplus\overline{\Delta_{\Theta_{f(T)}} (F^2(H_n)\otimes
\cD_{f(T)^*})}]\ominus
\{\Theta_{f(T)} z\oplus\Delta_{\Theta_{f(T)}} z:\ \ z\in
F^2(H_n)\otimes \cD_{f(T)^*}\},
$$
where $\Theta_{f(T)}$ is the characteristic function of the row
contraction  $f(T):=(f_1(T),\ldots, f_n(T))$, the defect operator
$~\Delta_{\Theta_{f(T)}}:=(I-\Theta_{f(T)}^*\Theta_{f(T)})^{1/2}$,
and
  the operator ${\bf A}_i$  is defined on $\widetilde \cH$
  by setting
\begin{equation}
\label{A} {\bf A}_{i}^*[\omega\oplus\Delta_{\Theta_{f(T)}}   z]:=
(S_i^*\otimes I_{\cD_{f(T)}}) \omega\oplus
C_i^*(\Delta_{\Theta_{f(T)}}z), \qquad i=1,\ldots,n,
\end{equation}
for   $\omega\in F^2(H_n)\otimes \cD_{f(T)} $,  $z\in
F^2(H_n)\otimes \cD_{f(T)^*}$,   where $C_i$ is  defined on
$\overline{\Delta_{\Theta_{f(T)}} (F^2(H_n)\otimes \cD_{f(T)^*})}$
by
$$
C_i(\Delta_{\Theta_{f(T)}}z):=\Delta_{\Theta_{f(T)}} (S_i \otimes
I_{\cD_{f(T)^*}}) z, \qquad i =1,\ldots,n,
$$
 and $S_1,\ldots, S_n$
are the left creation operators on the full Fock space $F^2(H_n)$.
Since $f(T)$ and ${\bf A}$ are completely non-coisometric row
contractions and $g$, the inverse of $f$ with respect to the
composition, is a bounded free holomorphic function on the unit ball
$[B(\cH)^n]_1$, then,  using the functional calculus for c.n.c. row
contractions (see \cite{Po-funct}), it makes sense to talk about
$g({\bf A}):=(g_1({\bf A}),\ldots, g_n({\bf A}))$  and
$g(f(T)):=(g_1(f(T)),\ldots, g_n(f(T)))$. Consequently, since
$g(f(T))=T$ and  $f(T)$ is unitarily equivalent to ${\bf A}$, we
deduce that $T=(T_1,\ldots, T_n)$ is unitarily equivalent to
$\TT:=(\TT_1,\ldots, \TT_n)$, where $\TT_i=g_i({\bf A})$. Since
$f\in \cM^b$, we use  Proposition \ref{ball} to conclude that
$\TT\in \BB_f^{cnc}(\widetilde \cH)$.

Consider
 the canonical unitary operator $U:\HH^2(f)\to F^2(H_n)$ defined by
 $Uf_\alpha=e_\alpha$, $\alpha\in \FF_n^+$ and
note that
\begin{equation}
\label{2charact}
\Theta_{f,T}=(U^*\otimes I_{\cD_{f,T}})\Theta_{f(T)}(U\otimes
I_{\cD_{f,T^*}}),
\end{equation}
where $\Theta_{f(T)}$ is the characteristic function of the row
contraction
 $f(T)=[f_1(T),\ldots, f_n(T)]$. Hence, we deduce that
 \begin{equation}\label{DeTe}
 \Delta_{\Theta_{f,T}}=(U^*\otimes
 I_{\cD_{f,T^*}})\Delta_{\Theta_{f(T)}}(U\otimes I_{\cD_{f,T^*}}).
\end{equation}
Define the subspaces
$$
{\bf G}:=\{\Theta_{f,T} x\oplus\Delta_{\Theta_{f,T}} x:\ \ x\in
\HH^2(f)\otimes \cD_{f,T^*}\}
$$
and
$$
{\widetilde G}:=\{\Theta_{f(T)} z\oplus\Delta_{\Theta_{f(T)}} z:\ \
z\in F^2(H_n)\otimes \cD_{f(T)^*}\},
$$
and the unitary operator $\Gamma$ acting  from  the Hilbert space
$(\HH^2(f)\otimes \cD_{f,T })\oplus\overline{\Delta_{\Theta_{f,T}}
(\HH^2(f)\otimes \cD_{f,T^*})}$ to $(F^2(H_n)\otimes \cD_{f(T)
})\oplus\overline{\Delta_{\Theta_{f(T)}} (F^2(H_n)\otimes
\cD_{f(T)^*})}$  and defined by
$$
\Gamma:=(U\otimes I_{\cD_{f,T}})\oplus (U\otimes I_{\cD_{f,T^*}}).
$$
Since $\cD_{f,T}=\cD_{f(T)}$ and $\cD_{f,T^*}=\cD_{f(T)^*}$, it is
easy to see that  that $\Gamma({\bf G})={\widetilde \cG}$ and
$\Gamma({\bf H})={\widetilde \cH}$. Therefore, $\Gamma|_{\bf H}:{\bf
H}\to \widetilde \cH$ is a unitary operator.

We introduce the operators ${\bf B}_i:=(\Gamma|_{\bf H})^{-1} \TT_i
(\Gamma|_{\bf H})$, $i=1,\ldots, n$. Since  the $n$-tuple
$(\TT_1,\ldots, \TT_n)$ is in $\BB_f^{cnc}(\widetilde \cH)$, we
deduce that ${\bf B}:=({\bf B}_1,\ldots, {\bf B}_n)$  is in
$\BB_f^{cnc}({\bf H})$.

Now, we show that the operators ${\bf T}_i$, $i=1,\ldots,n$, defined
in the theorem are well-defined and bounded on the Hilbert space
${\bf H}$. Note that since $\Theta_{f,T}$ is a multi-analytic
operator with respect to $M_{Z_1},\ldots, M_{Z_n}$, we have
\begin{equation*}
\begin{split}
\left\|D_i(\Delta_{\Theta_{f,T}}y)\right\|&=\left\|\Delta_{\Theta_{f,T}}(M_{Z_i}\otimes
I_{\cD_{f,T^*}})y\right\| \\
&=\left<M_{Z_i}^* M_{Z_i}\otimes I_{\cD_{f,T^*}}-\Theta_{f,T}^*
M_{Z_i}^*M_{Z_i}\Theta_{f,T})y,y\right>\\
&=\|Z_i\|^2_{\HH^2(f)} \left<
(I-\Theta_{f,T}^*\Theta_{f,T})y,y\right>=\|Z_i\|^2_{\HH^2(f)}\left\|\Delta_{\Theta_{f,T}}y\right\|^2
\end{split}
\end{equation*}
for any $y\in \HH^2(f)\otimes \cD_{f,T^*}$.  Consequently, $D_i$
extends to a unique bounded operator on the Hilbert space
$\overline{\Delta_{\Theta_{f,T}}
(\HH^2(f)\otimes \cD_{f,T^*})}$.
Note also that due to the fact that the subspace ${\bf G}$ is
invariant under $(M_{Z_i}\otimes I_{\cD_{f,T^*}})\oplus D_i$, we
have $\left[(M_{Z_i}^*\otimes I_{\cD_{f,T^*}})\oplus
D_i^*\right]({\bf H})\subset {\bf H}$, which proves our assertion.

Our next step is to show   that ${\bf B}_i={\bf T}_i$ for
$i=1,\ldots, n$. First, note that  due to relation \eqref{A} and the
functional calculus for c.n.c. row contractions, we have
\begin{equation*}\begin{split}
&\left(\omega\oplus\Delta_{\Theta_{f(T)}} z,
\TT_i(\omega'\oplus\Delta_{\Theta_{f(T)}}   z')\right>
\\
&\quad =\left(\omega\oplus\Delta_{\Theta_{f(T)}} z, g_i({\bf
A}_1,\ldots,
{\bf A}_n)(\omega'\oplus\Delta_{\Theta_{f(T)}}   z')\right>\\
 &\quad
=\lim_{r\to 1} \left<\omega\oplus\Delta_{\Theta_{f(T)}} z,\
g_i\left(r[(S_1\otimes I_{\cD_{f(T)}})\oplus
C_1],\ldots,r[(S_n\otimes I_{\cD_{f(T)}})\oplus C_n)]\right)
(\omega' \oplus
 \Delta_{\Theta_{f(T)}}z')\right>\\
 &
 \quad =\lim_{r\to 1} \left<\omega\oplus\Delta_{\Theta_{f(T)}} z,\
\left[(g_i(rS_1, \ldots, rS_n)\otimes I_{\cD_{f(T)}})\oplus
g_i(rC_1,\ldots, rC_n) \right] (\omega' \oplus
 \Delta_{\Theta_{f(T)}}z')\right>\\
 &\quad =
\lim_{r\to 1} \left<\omega\oplus\Delta_{\Theta_{f(T)}} z,\
[g_i(rS_1, \ldots, rS_n)\otimes I_{\cD_{f(T)}}]\omega'\oplus
[\Delta_{\Theta_{f(T)}}(g_i(rS_1,\ldots, rS_n)\otimes
I_{\cD_{f(T)^*}}) ]
  z'\right>
 \end{split}
\end{equation*}
for any $\omega, \omega'\in F^2(H_n)\otimes \cD_{f(T)}$ and $z,
z'\in F^2(H_n)\otimes \cD_{f(T)^*}$.

\smallskip
Now,  using relation \eqref{DeTe},  for any $x, x'\in
\HH^2(f)\otimes \cD_{f,T}$ and $y,y'\in \HH^2(f)\otimes
\cD_{f,T^*}$, we have
\begin{equation*}
\begin{split}
&\left<x\oplus\Delta_{\Theta_{f,T}}y, {\bf B}_i
(x'\oplus\Delta_{\Theta_{f,T}}y')\right> \\
&\quad = \left<x\oplus\Delta_{\Theta_{f,T}}y, (\Gamma|_{\bf H})^{-1}
\TT_i
(\Gamma|_{\bf H})(x'\oplus\Delta_{\Theta_{f,T}}y')\right>\\
&\quad =\left<[(U\otimes
I_{\cD_{f,T}})x\oplus\Delta_{\Theta_{f,T}}(U\otimes
I_{\cD_{f,T}})y], \TT_i[(U\otimes
I_{\cD_{f,T}})x'\oplus\Delta_{\Theta_{f,T}}(U\otimes
I_{\cD_{f,T}})y']\right>.
\end{split}
\end{equation*}
Setting $\omega=(U\otimes I_{\cD_{f,T}})x$, $z=(U\otimes
I_{\cD_{f,T}})y$, $\omega'=(U\otimes I_{\cD_{f,T}})x'$,
$z'=(U\otimes I_{\cD_{f,T}})y'$, and combining  the results above,
we obtain
\begin{equation*}
\begin{split}
&\left<x\oplus\Delta_{\Theta_{f,T}}y, {\bf B}_i
(x'\oplus\Delta_{\Theta_{f,T}}y')\right> \\
&\quad = \lim_{r\to 1}\left<[(U\otimes
I_{\cD_{f,T}})x\oplus\Delta_{\Theta_{f,T}}(U\otimes
I_{\cD_{f,T}})y],\right.\\
&\qquad \left.[g_i(rS_1, \ldots, rS_n)\otimes
I_{\cD_{f(T)}}](U\otimes I_{\cD_{f,T}})x'\oplus
[\Delta_{\Theta_{f(T)}}(g_i(rS_1,\ldots, rS_n)\otimes
I_{\cD_{f(T)^*}}) ](U^*\otimes I_{\cD_{f,T}})y'\right>\\
&\quad =\lim_{r\to 1}\left<x\oplus\Delta_{\Theta_{f,T}}y,
[(U^*g_i(rS_1,\ldots, rS_n)U)\otimes I_{\cD_{f,T}}]x'\oplus
[\Delta_{\Theta_{f(T)}}(U^*g_i(rS_1,\ldots, rS_n)U\otimes
I_{\cD_{f(T)^*}})] y'\right>\\
&\quad =\lim_{r\to 1}\left<x\oplus\Delta_{\Theta_{f,T}}y,
[g_i(rM_{f_1},\ldots, rM_{f_n})\otimes I_{\cD_{f,T}}]x'\oplus
[\Delta_{\Theta_{f(T)}}(g_i(rM_{f_1},\ldots, rM_{f_n})\otimes
I_{\cD_{f(T)^*}})] y'.
 \right>
\end{split}
\end{equation*}
On the other hand, since $f$ has the model property, we have
$$
M_{Z_i}=g_i(M_{f_1},\ldots, M_{f_n})=\text{\rm SOT-}\lim_{r\to
1}g_i(rM_{f_1},\ldots, rM_{f_n}).
$$
Now, we can  deduce that
\begin{equation*}
\begin{split}
&\left<x\oplus\Delta_{\Theta_{f,T}}y, {\bf B}_i
(x'\oplus\Delta_{\Theta_{f,T}}y')\right> \\
&\quad =\left<x\oplus\Delta_{\Theta_{f,T}}y, (M_{Z_i}\otimes
I_{\cD_{f,T}})x'\oplus \Delta_{\Theta_{f(T)}} (M_{Z_i}\otimes
I_{\cD_{f,T}})y'\right>\\
&\quad =\left<x\oplus\Delta_{\Theta_{f,T}}y, {\bf
T}_i(x\oplus\Delta_{\Theta_{f,T}}y')\right>
\end{split}
\end{equation*}
for any $x, x'\in \HH^2(f)\otimes \cD_{f,T}$ and $y,y'\in
\HH^2(f)\otimes \cD_{f,T^*}$. Hence, we obtain ${\bf B}_i={\bf T}_i$
for any $i=1,\ldots,n$, which completes the first part of the proof.

Since the $n$-tuples of operators $T:=(T_1,\ldots, T_n)$ and $\TT:=(\TT_1,\ldots, \TT_n)=(g_1({\bf A}),\ldots, g_n({\bf A}))$ are unitarily equivalent, we deduce that
$T\in \BB_f^{pure}(\cH)$ if and only if $\TT\in \BB_f^{pure}(\widetilde\cH)$. On the other hand, due to Proposition \ref{ball},  $\TT\in \BB_f^{pure}(\widetilde\cH)$ if and only if ${\bf A}\in [B(\cH)^n]_1^{pure}$. Since the row contraction $f(T):=(f_1(T),\ldots, f_n(T))$  is unitarily equivalent to ${\bf A}$, Theorem 4.1 from \cite{Po-charact} shows that ${\bf A}$ is pure if and only if the characteristic function $\Theta_{f(T)}$ is an isometry which, due to relation \eqref{2charact}, is equivalent to  $\Theta_{f,T}$ being an isometry. This completes the proof.
\end{proof}

Let  $\Phi: \HH^2(f)\otimes \cK_1\to \HH^2(f)\otimes \cK_2$ and
$\Phi':\HH^2(f)\otimes\cK_1'\to \HH^2(f)\otimes  \cK_2'$ be two
multi-analytic operators with respect to $M_{Z_1},\ldots, M_{Z_n}$. We say
that $\Phi$ and $\Phi'$ coincide if there are two unitary operators
$\tau_j\in B(\cK_j, \cK_j')$, $j=1,2$,  such that
$$
\Phi'(I_{\HH^2(f)}\otimes \tau_1)=(I_{\HH^2(f)}\otimes \tau_2) \Phi.
$$
The next result shows that the   characteristic function is a
complete unitary invariant for  the $n$-tuples of operators in the c.n.c. part of the
 noncommutative domain  $\BB_f(\cH)$.

\begin{theorem}\label{u-inv}
 Let  $f=(f_1,\ldots, f_n)$ be   an $n$-tuple  of
 formal power series in the class $\cM^b$ and let $T:=(T_1,\ldots, T_n)\in \BB_f^{cnc}(\cH)$
   and
$T':=(T_1',\ldots, T_n')\in \BB_f^{cnc}(\cH')$. Then $T$ and $T'$
are unitarily equivalent if and only if their characteristic
functions $\Theta_{f,T}$  and $\Theta_{f,T'}$ coincide.
\end{theorem}

\begin{proof}

 Assume that $T$ and $T'$ are unitarily equivalent and let $W:\cH\to
\cH'$ be a unitary operator such that $T_i=W^*T_i'W$ for any
$i=1,\ldots, n$.  Since $T\in \cC_f^{SOT}(\cH)$ or $T\in
\cC_f^{rad}(\cH)$ and similar relations hold for $T'$, it is  easy
to see that
$$
W\Delta_{f,T}=\Delta_{f,T'}W \quad \text{ and }\quad (\oplus_{i=1}^n
W)\Delta_{f,T^*}=\Delta_{f,T'^*}(\oplus_{i=1}^n W).
$$
Define the unitary operators $\tau$ and $\tau'$ by setting
$$\tau:=W|_{\cD_{f,T}}:\cD_{f,T}\to \cD_{f,T'} \quad \text{ and }\quad
\tau':=(\oplus_{i=1}^n W)|_{\cD_{f,T^*}}:\cD_{f,T^*}\to
\cD_{f,T'^*}.
$$
Using the definition of the  characteristic function,
we deduce that that
$$
(I_{\HH^2(f)}\otimes
\tau)\Theta_{f,T}=\Theta_{f,T'}(I_{\HH^2(f)}\otimes \tau').
$$

Conversely, assume that the   characteristic functions  of $T$ and
$T'$ coincide.  Then  there exist unitary operators
$\tau:\cD_{f,T}\to \cD_{f,T'}$ and $\tau_*:\cD_{f,T^*}\to
\cD_{f,{T'}^*}$ such that
\begin{equation}\label{com}
(I_{\HH^2(f)}\otimes
\tau)\Theta_{f,T}=\Theta_{f,T'}(I_{\HH^2(f)}\otimes \tau_*).
\end{equation}
Hence, we obtain
\begin{equation}
\label{DETh}
\Delta_{\Theta_{f,T}}=\left(I_{\HH^2(f)}\otimes \tau_*\right)^*
\Delta_{\Theta_{f,T'}}\left(I_{\HH^2(f)}\otimes \tau_*\right)
\end{equation}
and
$$
\left(I_{\HH^2(f)}\otimes
\tau_*\right)\overline{\Delta_{f,T}(\HH^2(f)\otimes \cD_{f,T^*})}=
\overline{\Delta_{f,T'}(\HH^2(f)\otimes \cD_{f,{T'}^*})}.
$$
Consider the Hilbert spaces
$$
{\bf K}_{f,T}:=[ (\HH^2(f)\otimes \cD_{f,T
})\oplus\overline{\Delta_{\Theta_{f,T}} (\HH^2(f)\otimes
\cD_{f,T^*})}],
$$
$${\bf G}_{f,T}:= \{\Theta_{f,T} x\oplus\Delta_{\Theta_{f,T}} x:\ \
x\in \HH^2(f)\otimes \cD_{f,T^*}\},
$$
and  ${\bf H}_{f,T}:={\bf K}_{f,T}\ominus {\bf G}_{f,T}$.
 We define
the unitary operator $\Gamma: {\bf K}_{f,T}\to {\bf K}_{f,T'}$ by
setting
$$\Gamma:=(I_{\HH^2(f)}\otimes \tau)\oplus (I_{\HH^2(f)}\otimes \tau_*).
$$
Due to relations \eqref{com} and  \eqref{DETh}, we have $\Gamma({\bf G}_{f,T})={\bf G}_{f,T'}$ and
$\Gamma({\bf H}_{f,T})={\bf H}_{f,T'}$. Therefore, the operator
$\Gamma|_{{\bf H}_{f,T}}:{\bf H}_{f,T}\to {\bf H}_{f, T'}$ is
unitary. Now, let ${\bf T}:=[{\bf T}_1,\ldots {\bf T}_n]$ and ${\bf
T}':=[{\bf T}_1',\ldots {\bf T}_n']$ be the models provided by
Theorem \ref{funct-model1}  for the $n$-tuples $T$ and $T'$,
respectively.

We recall that  the operator $D_i$  is defined by
$D_i(\Delta_{\Theta_{f,T}}y):=\Delta_{\Theta_{f,T}} (M_{Z_i} \otimes
I_{\cD_{f,T^*}}) y$ \, for  all  $y\in \HH^2(f)\otimes \cD_{f,T^*}$. Using
relation \eqref{DETh}, we deduce that
\begin{equation*}
\begin{split}
D_i'((I\otimes \tau_*)\Delta_{\Theta_{f,T}}y)&=
D_i'(\Delta_{\Theta_{f,T'}}(I\otimes \tau_*)y)\\
&=\Delta_{\Theta_{f,T'}}(M_{Z_i}\otimes I_{\cD_{f,{T'}^*}})(I\otimes
\tau_*)y)\\
&=(I\otimes \tau_*)\Delta_{\Theta_{f,T}}(M_{Z_i}\otimes
I_{\cD_{f,{T}^*}})y\\
&=(I\otimes \tau_*)D_i(\Delta_{\Theta_{f,T}}y).
\end{split}
\end{equation*}
Consequently, we obtain that ${D_i'}^*(I\otimes \tau_*)=(I\otimes
\tau_*)D_i^*$. On the other hand, we have
\begin{equation*}
 (M_{Z_i}^*\otimes
I_{\cD_{f,T'}})(I_{\HH^2(f)}\otimes \tau)= (I_{\HH^2(f)}\otimes
\tau)(M_{Z_i}^*\otimes I_{\cD_{f,T}}).
\end{equation*}
Combining these relations and using the definition of the unitary operator $\Gamma$,  we deduce  that
\begin{equation*}
\begin{split}
{{\bf T}_i'}^* \Gamma(x+\Delta_{\Theta_{f,T}}y)&={{\bf
T}_i'}^*\left((I_{\HH^2(f)}\otimes \tau)x\oplus (I_{\HH^2(f)}\otimes
\tau_*)\Delta_{\Theta_{f,T}}y\right)\\
&= {{\bf T}_i'}^*\left((I_{\HH^2(f)}\otimes \tau)x\oplus
\Delta_{\Theta_{f,T'}}(I_{\HH^2(f)}\otimes
\tau_*)y\right)\\
&=(M_{Z_i}^*\otimes I_{\cD_{f,T'}})(I_{\HH^2(f)}\otimes \tau)x\oplus
{D_i'}^*\left(\Delta_{\Theta_{f,T'}}(I_{\HH^2(f)}\otimes
\tau_*)y\right)\\
&= (M_{Z_i}^*\otimes I_{\cD_{f,T'}})(I_{\HH^2(f)}\otimes
\tau)x\oplus {D_i'}^*\left((I_{\HH^2(f)}\otimes
\tau_*)\Delta_{\Theta_{f,T}}y\right)\\
&= (I_{\HH^2(f)}\otimes \tau)(M_{Z_i}^*\otimes I_{\cD_{f,T}})x\oplus
(I_{\HH^2(f)}\otimes
\tau_*){D_i'}^*\left(\Delta_{\Theta_{f,T}}y\right)\\
&=\Gamma{\bf T}_i^*(x\oplus \Delta_{\Theta_{f,T}}y)
\end{split}
\end{equation*}
for any $x\oplus \Delta_{\Theta_{f,T}}y\in {\bf H}_{f,T}$ and
$i=1,\ldots, n$. Consequently,   we obtain ${{\bf T}_i'}^*
(\Gamma|_{{\bf H}_{f,T}})=(\Gamma|_{{\bf H}_{f,T}}){\bf T}_i^*$ for
  $i=1,\ldots,n$.
 Now, using Theorem \ref{funct-model1}, we conclude that $T$ and $T'$ are unitarily equivalent.
  The proof is complete.
\end{proof}

 In what follows we prove that any contractive multi-analytic operator
$\Theta:\HH^2(f)\otimes \cE_*\to \HH^2(f)\otimes
{\cE}\quad(\cE,\cE_*~$ are Hilbert spaces)  with respect to $M_{Z_1},\ldots, M_{Z_n}$ generates a c.n.c.
$n$-tuple of operators
${\bf T}:=({\bf T}_1,\ldots, {\bf T}_n)\in \BB_f({\bf H})$.
We mention that $~\Theta~$ is  called purely contractive if  $\|P_{\cE} \Theta (1\otimes
x)\|<\|x\|$ for any  $x\in \cE_*$.
\begin{theorem}\label{funct-model2} Let  $f=(f_1,\ldots, f_n)$ be   an $n$-tuple  of
 formal power series in the class $\cM^b$ and  let $(M_{Z_1},\ldots, M_{Z_n})$  be the universal
model associated with  the noncommutative domain $\BB_f$.
 Let $~\Theta:\HH^2(f)\otimes \cE_*\to \HH^2(f)\otimes \cE$ be a contractive multi-analytic operator
  and  set
 $~\Delta_\Theta:=(I-\Theta^*\Theta)^{1/2}~$.
Then the  $n$-tuple  ${\bf T}:=({\bf T}_1,\ldots, {\bf T}_n)$
defined on the Hilbert space
$$
{\bf H}:=[ (\HH^2(f)\otimes {\cE})\oplus\overline{\Delta_\Theta
(\HH^2(f)\otimes \cE_*)}]\ominus\{\Theta y\oplus\Delta_\Theta
y:\quad
y\in \HH^2(f)\otimes \cE_*\}
$$
by
$$
{\bf T}_i^*(x\oplus \Delta_\Theta y):=(M_{Z_i}^*\otimes I_{\cE_*}) x
\oplus D_i^*(\Delta_\Theta y),\qquad  i=1,\ldots,n,
$$
where each operator $~D_i~$ is defined by
$$D_i(\Delta_\Theta y):=\Delta_\Theta (M_{Z_i}\otimes I_\cE) y, \quad y\in \HH^2(f)\otimes \cE_*,
$$
is in  $\BB_f^{cnc}({\bf H})$.

Moreover, if  $~\Theta~$ is purely
contractive and
\begin{equation*}
\overline{\Delta_\Theta(\HH^2(f)\otimes
\cE_*)}=\overline{\Delta_\Theta(\vee_{i=1}^n M_{f_i}(\HH^2(f)\otimes
\cE_*))},
\end{equation*}
then $~\Theta~$ coincides with the characteristic function of the
$n$-tuple  ${\bf T}:=({\bf T}_1,\ldots, {\bf T}_n)$.
\end{theorem}

\begin{proof}  Since  $f=(f_1,\ldots, f_n)$  has  the model
property, we have $M_{f_j}=f_j(M_{Z_1},\ldots, M_{Z_n})$ where the
$n$-tuple $(M_{Z_1},\ldots, M_{Z_n})$ is either in the convergence
set $\cC_f^{SOT}(\HH^2(f))$ or $\cC_f^{rad}(\HH^2(f))$, and
$g_j(M_{f_1},\ldots, M_{f_n})=M_{Z_j}$ where $(M_{f_1},\ldots,
M_{f_n})$ is  in the convergence set $\cC_g^{rad}(\HH^2(f))$.
Consequently,  we have $\Theta (M_{Z_i}\otimes
I_{\cE_*})=(M_{Z_i}\otimes I_\cE)$, $i=1,\ldots,n$,  if and only if
$\Theta (M_{f_i}\otimes I_{\cE_*})=(M_{f_i}\otimes I_\cE)$,
$i=1,\ldots,n$. Setting $\Psi:=(U\otimes I_\cE)\Theta (U^*\otimes
I_{\cE_*})$, where the canonical unitary operator $U:\HH^2(f)\to
F^2(H_n)$  is defined by
 $Uf_\alpha=e_\alpha$, $\alpha\in \FF_n^+$, we deduce that $\Psi$ is
 a multi-analytic operator on the Fock space $F^2(H_n)$, i.e.,
 $\Psi(S_i\otimes I_{\cE_*})=(S_i\otimes I_\cE)$ for $i=1,\ldots,
 n$.

Define the $n$-tuple ${\bf A}:=({\bf A}_1,\ldots, {\bf A}_n)$ on
the Hilbert
 space
 $$
\widetilde{\cH}:=[ (F^2(H_n)\otimes
\cE)\oplus\overline{\Delta_{\Psi} (F^2(H_n)\otimes
 {\cE_*})}]\ominus
\{\Psi z\oplus\Delta_{\Psi} z:\ \ z\in F^2(H_n)\otimes
 {\cE_*}\},
$$
where  the defect operator $~\Delta_{\Psi}:=(I-\Psi^*\Psi)^{1/2}$
and
  the operator ${\bf A}_i$, $i=1,\ldots, n$, is defined on $\widetilde \cH$
  by setting
\begin{equation*}
 {\bf A}_{i}^*[\omega\oplus\Delta_{\Psi}   z]:=
(S_i^*\otimes I_{\cE}) \omega\oplus C_i^*(\Delta_{\Psi}z), \qquad
i=1,\ldots,n,
\end{equation*}
for   $\omega\in F^2(H_n)\otimes  {\cE} $,  $z\in F^2(H_n)\otimes
 {\cE_*}$,   where $C_i$ is  defined on $\overline{\Delta_{\Psi}
(F^2(H_n)\otimes  {\cE_*})}$ by
$$
C_i(\Delta_{\Psi}z):=\Delta_{\Psi} (S_i \otimes I_{\cE_*}) z, \qquad
i =1,\ldots,n,
$$
 and $S_1,\ldots, S_n$
are the left creation operators on the full Fock space $F^2(H_n)$.

Consider the Hilbert spaces
$$
{\widetilde \cK} :=  (F^2(H_n)\otimes
\cE)\oplus\overline{\Delta_{\Psi} (F^2(H_n)\otimes
 {\cE_*})}
$$
and
$${\widetilde \cG}:= \{\Psi z\oplus\Delta_{\Psi} z:\ \ z\in F^2(H_n)\otimes
 {\cE_*}\}.
$$
Since $\Psi$ is
 a multi-analytic operator on the Fock space $F^2(H_n)$, it is easy to
  see the $[C_1,\ldots,C_n]$ is a row isometry and
$\widetilde \cG$ is invariant under each operator ${\bf
W}_i:=S_i\oplus C_i$, $i=1,\ldots, n$,   acting on  $\widetilde
\cK$. Therefore, ${\bf A}_i^*={\bf W}_i^*|_{\widetilde\cH}$,
$i=1,\ldots, n$. Note also that ${\bf A}=[{\bf A}_1,\ldots, {\bf
A}_n]$ is a c.n.c. row contraction. Indeed, let  $\omega\oplus
\Delta_\Psi z\in \widetilde \cH$ be such that
$\sum_{|\alpha|=k}\|{\bf A}_\alpha^*(\omega\oplus \Delta_\Psi
z)\|^2=\|\omega\oplus \Delta_\Psi z\|^2$ for any $k\in \NN$. Taking
into account that
$\lim_{k\to\infty}\sum_{|\alpha|=k}\|S_\alpha^*\omega\|^2=0$ and
$\sum_{|\alpha|=k}\|C_\alpha^*\Delta_\Psi z\|^2\leq \|\Delta_\Psi
z\|^2$, we deduce that $\omega=0$. On the other hand, since $0
\oplus \Delta_\Psi z\in \widetilde \cH$, we must have $\left<0
\oplus \Delta_\Psi z,\Psi u\oplus\Delta_{\Psi} u\right>=0 $ for any
$ u\in F^2(H_n)\otimes
 {\cE_*}$, which implies $\Delta_{\Psi} z=0$ and proves our
 assertion.

 Since   ${\bf A}$ is a completely non-coisometric row
contractions and $g$, the inverse of $f$ with respect to the
composition, is a bounded free holomorphic function on the unit ball
$[B(\cH)^n]_1$,   it makes sense to talk about $g({\bf
A}):=(g_1({\bf A}),\ldots, g_n({\bf A}))$   using the functional
calculus for c.n.c. row contractions (see \cite{Po-funct}).
  Since $f\in \cM^b$, setting  $\TT:=(\TT_1,\ldots, \TT_n)$, where $\TT_i=g_i({\bf
A})$,   and using  Proposition \ref{ball}, we deduce that  $\TT\in
\BB_f^{cnc}(\widetilde \cH)$. Consider
 the unitary
operator $\Gamma$ acting  from  the Hilbert space $(\HH^2(f)\otimes
\cE)\oplus\overline{\Delta_{\Theta} (\HH^2(f)\otimes \cE_*)}$ to
$(F^2(H_n)\otimes \cE)\oplus\overline{\Delta_{\Psi} (F^2(H_n)\otimes
\cE_*)}$  and defined by
$$
\Gamma:=(U\otimes I_{\cE})\oplus (U\otimes I_{\cE_*}).
$$
As in the proof of Theorem \ref{funct-model1}, one can show that
${\bf T}_i$ is a bounded operator on ${\bf H}$ and  ${\bf
T}_i=(\Gamma|_{\bf H})^{-1} \TT_i (\Gamma|_{\bf H})$, $i=1,\ldots,
n$. Consequently, we have ${\bf T}\in \BB_f^{cnc}({\bf H})$, which
proves the first part of the theorem.

To prove the second part of the theorem, we assume that
$~\Theta~$ is purely contractive, i.e., $\|P_{\cE} \Theta (1\otimes
x)\|<\|x\|$ for any  $x\in \cE_*$, and
\begin{equation*}
\overline{\Delta_\Theta(\HH^2(f)\otimes
\cE_*)}=\overline{\Delta_\Theta(\vee_{i=1}^n M_{f_i}(\HH^2(f)\otimes
\cE_*))}.
\end{equation*}
These conditions imply that $\Psi$ is purely contractive and
\begin{equation*}
\overline{\Delta_\Psi(F^2(H_n)\otimes
\cE_*)}=\overline{\Delta_\Psi[(F^2(H_n)\otimes \cE_*)\ominus
\cE_*]}.
\end{equation*}
According to Theorem 4.1 from \cite{Po-charact},  the multi-analytic
operator $\Psi$ coincides with the characteristic function
$\Theta_{\bf A}$ of the row contraction ${\bf A}=[{\bf A}_1,\ldots,
{\bf A}_n]$.

Note that the characteristic function of  $\TT=g(\bf A)\in \BB_f^{cnc}(\widetilde \cH)$ is
$$\Theta_{f,\TT}=(U^*\otimes I) \Theta_{f(\TT)}(U\otimes I)=(U^*\otimes I) \Theta_{{\bf A}}(U\otimes I).
$$
Since ${\bf
T}_i=(\Gamma|_{\bf H})^{-1} \TT_i (\Gamma|_{\bf H})$, $i=1,\ldots,
n$, we also have that the characteristic functions $\Theta_{f,{\bf T}}$  and  $\Theta_{f,\TT}$ coincide. Combining these results with the fact that
$(U^*\otimes
I_{\cE_*})\Psi(U\otimes I_\cE)=\Theta $, we conclude that $~\Theta~$ coincides with $\Theta_{f,{\bf T}}$.
 The proof is complete.
\end{proof}

\bigskip

\section{  Dilation theory on noncommutative domains  }

In this section, we study the $*$-representations of the
$C^*$-algebra $C^*(M_{Z_1},\ldots, M_{Z_n})$ and obtain a Wold type
decomposition for the nondegenerate $*$-representations. Under
natural conditions on the $n$-tuple $f=(f_1,\ldots, f_n)$ of formal
power series,  we show that any $n$-tuple $T=(T_1,\ldots,T_n)$ of
operators is
 in the noncommutative domain
 $\BB_f(\cH)$,  has  a minimal dilation which is unique up to an isomorphism.
 We also provide a commutant lifting theorem for $\BB_f(\cH)$.

\begin{proposition}\label{compact} Let $f=(f_1,\ldots, f_n)$ be  an  $n$-tuple of formal power
series in  the class $\cM^{||}$ and let $(M_{Z_1},\ldots, M_{Z_n})$
be the universal model associated with $\BB_f$.
   Then all compact
 operators in $B(\HH^2(f))$ are contained in $C^*(M_{Z_1},\ldots,
 M_{Z_n})$.
 \end{proposition}
\begin{proof} Since $f\in \cM^{||}$, the  universal $n$-tuple $(M_{Z_1},\ldots,
 M_{Z_n})$
 is in  the set  of  norm-convergence (or radial norm-convergence)  for
 $f$ and, consequently,  the operator
$f_i(M_{Z})$ is in
 $ \overline{\text{\rm span}} \{M_{Z_\alpha} M_{Z_\beta}^*:\
\alpha,\beta\in \FF_n^+\}$.  Taking into account that
$f=(f_1,\ldots, f_n)$ is   an  $n$-tuple of formal power series
 with the model
 property,  we have $M_{f_i}=f_i(M_{Z_1},\ldots, M_{Z_n})$. On the other hand,  the
orthogonal projection of $\HH^2(f)$ onto the constant power series
satisfies the equation $P_\CC=I-\sum_{i=1}^n f_i(M_{Z})f_i(M_{Z})^*
$. Therefore,  $P_\CC$ is also in the above-mentioned span. Let
$q(M_Z):=\sum_{|\alpha|\leq m}a_\alpha [f(M_Z)]_\alpha$ and let
 $\xi:=\sum_{\beta\in \FF_n^+} b_\beta f_\beta\in \HH^2(f)$. Note
 \begin{equation*}
 \begin{split}
 P_\CC q(M_Z)^*\xi&=P_\CC \sum_{|\alpha|\leq m}\overline{a}_\alpha M_{f_\alpha}^*\xi
 =\sum_{|\alpha|\leq m} \overline{a}_\alpha b_\alpha\\
 &=\left<\xi, \sum_{|\alpha|\leq m} a_\alpha f_\alpha\right>
 =\left<\xi, q(M_Z)1\right>.
 \end{split}
 \end{equation*}
Consequently, if $r(M_Z):=\sum_{|\gamma|\leq p}c_\gamma [f(M_Z)]_\gamma$, then
\begin{equation}
\label{rP}
r(M_Z)P_\CC q(M_Z)^*\xi=\left<\xi,q(M_Z)1\right>r(M_Z)1,
\end{equation}
which shows that  $r(M_Z)P_\CC q(M_Z)^*$ is a rank one operator in
$B(\HH^2(f))$. Taking into account that  the   vectors of the form $
\sum_{|\alpha|\leq m} a_\alpha[f(M_Z)]_\alpha 1$, where $\ m\in
\NN$, $a_\alpha\in \CC$, are dense in $\HH^2(f)$, and using relation
\eqref{rP}, we deduce that all compact operators in $B(\HH^2(f))$
are in $
 \overline{\text{\rm span}} \{M_{Z_\alpha} M_{Z_\beta}^*:\ \alpha,\beta\in
 \FF_n^+\}.
 $
The proof is complete.
\end{proof}

 The next result, is a Wold type decomposition for
nondegenerate $*$-representations of the $C^*$-algebra
$C^*(M_{Z_1},\ldots, M_{Z_n})$.

\begin{theorem}\label{wold} Let $f=(f_1,\ldots, f_n)$ be  an  $n$-tuple of formal power series
in the class $\cM^{||}$ and let $(M_{Z_1},\ldots, M_{Z_n})$ be the
universal model associated with $\BB_f$.  If
$\pi:C^*(M_{Z_1},\ldots, M_{Z_n})\to B(\cK)$ is a nondegenerate
$*$-representation  of \ $C^*(M_{Z_1},\ldots, M_{Z_n})$ on a
separable Hilbert space  $\cK$, then $\pi$ decomposes into a direct
sum
$$
\pi=\pi_0\oplus \pi_1 \  \text{ on  } \ \cK=\cK_0\oplus \cK_1,
$$
where $\pi_0$, $\pi_1$  are disjoint representations of
$C^*(M_{Z_1},\ldots, M_{Z_n})$ on the Hilbert spaces
$$
\cK_0:=\overline{\text{\rm span}}\left\{\pi(M_{Z_\alpha})
\left(I-\sum_{i=1}^n f_i(\pi(M_{Z_1}),\ldots,
\pi(M_{Z_n}))f_i(\pi(M_{Z_1}),\ldots, \pi(M_{Z_n}))^*\right) \cK:\
\alpha\in \FF_n^+\right\}
$$
 and $\cK_1:=\cK_0^\perp$,
 respectively, such that, up to an isomorphism,
\begin{equation}\label{shi}
\cK_0\simeq \HH^2(f)\otimes \cG, \quad  \pi_0(X)=X\otimes I_\cG,
\quad X\in C^*(M_{Z_1},\ldots, M_{Z_n}),
\end{equation}
 for some Hilbert space $\cG$ with
$$
\dim \cG=\dim \left[\text{\rm range}\,\left(I-\sum_{i=1}^n
f_i(\pi(M_{Z_1}),\ldots, \pi(M_{Z_n}))f_i(\pi(M_{Z_1}),\ldots,
\pi(M_{Z_n}))^*\right)\right],
$$
 and $\pi_1$ is a $*$-representation  which annihilates the compact operators   and
$$
\sum_{i=1}^n f_i(\pi_1(M_{Z_1}),\ldots,
\pi_1(M_{Z_n}))f_i(\pi_1(M_{Z_1}),\ldots,
\pi_1(M_{Z_n}))^*=I_{\cK_1}.
$$
Moreover, if $\pi'$ is another nondegenerate  $*$-representation of
$C^*(M_{Z_1},\ldots, M_{Z_n})$ on a separable  Hilbert space
$\cK'$, then $\pi$ is unitarily equivalent to $\pi'$ if and only if
$\dim\cG=\dim\cG'$ and $\pi_1$ is unitarily equivalent to $\pi_1'$.
\end{theorem}

\begin{proof}
Since  $f=(f_1,\ldots, f_n)$ is  an  $n$-tuple of formal power
series in the class $\cM^{||}$, Proposition \ref{compact} implies
that all the compact operators   in $B(\HH^2(f))$ are contained in
$C^*(M_{Z_1},\ldots, M_{Z_n})$. The standard theory of
representations of   $C^*$-algebras shows that the representation
$\pi$ decomposes into a direct sum $\pi=\pi_0\oplus \pi_1$ on  $
\cK=\cK_0\oplus \cK_1$, where
$$\cK_0:=\overline{\text{\rm span}}\{\pi(X)\cK:\ X \ \text{ is compact operator    in }
B(\HH^2(f))\}
 \quad \text{ and  }\quad  \cK_1:=\cK_0^\perp,
$$
and the the representations $\pi_j:C^*(M_{Z_1},\ldots, M_{Z_n})\to
\cK_j$ are defined by $\pi_j(X):=\pi(X)|_{\cK_j}$, \ $j=0,1$. The
disjoint representations $\pi_0$, $\pi_1$  are
 such that
 $\pi_1$ annihilates  the compact operators in $B(\HH^2(f))$, and  $\pi_0$
 is uniquely determined by the action of $\pi$ on the ideal  of compact operators
  in $B(\HH^2(f))$.
Taking into account that  every representation of  the compact
operators on $\HH^2(f)$ is equivalent to a multiple of the identity
representation, we deduce  relation \eqref{shi}. Using the  proof of
Theorem \ref{compact}, we deduce that
\begin{equation*}\begin{split}
\cK_0&:=\overline{\text{\rm span}}\{\pi(X)\cK:\ X \ \text{ is
compact operator    in }
B(\HH^2(f))\}\\
&=\overline{\text{\rm span}}\{\pi(M_{Z_\alpha} P_\CC  M_{Z_\beta}^*)\cK:\ \alpha, \beta\in \FF_n^+\}\\
&= \overline{\text{\rm span}}\left\{\pi(M_{Z_\alpha})
\left(I-\sum_{i=1}^n f_i(\pi(M_{Z_1}),\ldots,
\pi(M_{Z_n}))f_i(\pi(M_{Z_1}),\ldots, \pi(M_{Z_n}))^*\right) \cK:\
\alpha\in \FF_n^+\right\}.
\end{split}
\end{equation*}
Now, since $P_\CC =I-\sum_{i=1}^n f_i(M_{Z_1},\ldots,
M_{Z_n})f_i(M_{Z_1},\ldots, M_{Z_n})^*$ is a rank one projection in
the $C^*$-algebra  $C^*(M_{Z_1},\ldots, M_{Z_n})$ , we have
 $$ \sum_{i=1}^n
f_i(\pi_1(M_{Z_1}),\ldots, \pi_1(M_{Z_n}))f_i(\pi_1(M_{Z_1}),\ldots,
\pi_1(M_{Z_n}))^*=I_{\cK_1}
$$ and
$$
\dim \cG=\dim \left[\text{\rm range}\,\pi(P_\CC )\right] =\dim
\left[\text{\rm range}\,\left(I-\sum_{i=1}^n
f_i(\pi(M_{Z_1}),\ldots, \pi(M_{Z_n}))f_i(\pi(M_{Z_1}),\ldots,
\pi(M_{Z_n}))^*\right)\right].
$$
To prove the last part of the theorem, we recall that, according to
the standard theory of representations of $C^*$-algebras, $\pi$ and
$\pi'$ are unitarily equivalent if and only if
 $\pi_0$ and $\pi_0'$ (resp.~$\pi_1$ and $\pi_1'$) are unitarily equivalent.
On the other hand, we proved in \cite{Po-multi} that   the
$C^*$-algebra $C^*(M_{Z_1},\ldots,
 M_{Z_n})$  is irreducible  and, consequently,  the $n$-tuples $(M_{Z_1}\otimes I_{\cK},\ldots, M_{Z_n}\otimes I_{\cK})$ is unitarily  equivalent to
 $(M_{Z_1}\otimes I_{\cK'},\ldots, M_{Z_n}\otimes I_{\cK'})$  if and only if
 $\dim \cK=\dim \cK'$.
Hence, we conclude that $\dim\cG=\dim\cG'$ and
 complete the proof.
\end{proof}

 We
  introduce  the class $\cM_{rad}^{b}$ of all formal power
series $f=(f_1,\ldots, f_n)$ with the property that there is
$\delta\in (0,1)$ such that $rf\in \cM^{b}$ for any $r\in (\delta,
1]$. We remark that in all the examples presented in
\cite{Po-multi}, the corresponding  $n$-tuples $f=(f_1,\ldots, f_n)$
are in the class
 $\cM_{rad}^{||}\subset \cM_{rad}^b$. Moreover, the   $n$-tuple   of polynomials
with property $(\cA)$ are also in  the class $\cM_{rad}^{b}$. The
 {\it coisometric} part of $\BB_f(\cH)$  is defined as the set
$$
\BB^c_f(\cH):=\{X=(X_1,\ldots, X_n)\in \BB_f(\cH):\ \sum_{i=1}^n
f_i(X)f_i(X)^*=I\}.
$$

\begin{proposition} \label{Mb}  Let $f=(f_1,\ldots, f_n)$ be  an  $n$-tuple of formal power
series with the model property and let $g=(g_1,\ldots,g_n)$ be its
inverse with respect to the composition. If $f$ satisfies either one
of the following conditions:
\begin{enumerate}
\item[(i)]
  $
f\in \cM_{rad}^{b}$;
 \item[(ii)]  $f\in \cM_{rad}\cap \cM^{||}$;
 \item[(iii)] $f\in \cM^{||}$ and $g\in \cA_n$,
\end{enumerate}
 then
 $\BB_f(\cH)=g\left([B(\cH)^n]_1^-\right).
 $
Moreover, in this case, $g:[B(\cH)^n]_1^-\to \BB_f(\cH)$ is a
bijection with inverse
 $f:\BB_f(\cH)\to [B(\cH)^n]_1^-$. In particular,
 $g([B(\cH)^n]_1^c)=\BB_f^c(\cH)$.
\end{proposition}
\begin{proof} Assume that condition  $(i)$ holds.
Since  $\BB_f(\cH)\subseteq g\left([B(\cH)^n]_1^-\right)$, it
remains to prove the reverse inclusion. Let    $Y:=g(X)$ with
$X=(X_1,\ldots,X_n)\in [B(\cH)^n]_1^-$. Assume that
$f_i:=\sum_{\alpha\in \FF_n^+} c_\alpha^{(i)} Z_\alpha$,
$i=1,\ldots,n$. Since $f=(f_1,\ldots, f_n)$
 has the
 radial approximation property, $g_i:=\sum_{\alpha\in \FF_n^+} a_\alpha^{(i)} Z_\alpha$
 is a free holomorphic function on $[B(\cH)^n]_\gamma$ for some $\gamma>1$. Moreover,
 there is \ $\delta\in (0,1)$ with the property   that   for any $r\in (\delta,1]$,
 the series
  $g_i(\frac{1}{r}S):=\sum_{k=0}^\infty\sum_{|\alpha|=k} \frac{a_\alpha^{(i)}}{r^{|\alpha|}} S_\alpha$
   is convergent in the operator norm topology and represents   an element
    in  the noncommutative disc algebra $\cA_n$,
and
  $$
  \frac{1}{r} S_j=f_j\left(g_1\left(\frac{1}{r} S\right),\ldots,
  g_n\left(\frac{1}{r} S\right)\right),\qquad j\in\{1,\ldots, n\}, \ r\in (\delta,
  1],
  $$
where  $g(\frac{1}{r} S)$ is in the SOT-convergence (or radial
SOT-convergence) of $f$ and $S=(S_1,\ldots, S_n)$ is the $n$-tuple
of left creation operators on the Fock space $F^2(H_n)$. Since $f\in
\cM^b_{rad}$, we
  deduce that one of the following conditions holds:
  \begin{enumerate}
  \item[(a)]
$
 \frac{1}{r}S_i=\text{\rm SOT-}\lim_{m\to\infty}\sum_{|\alpha|\leq m} c_\alpha^{(i)}
 {g}_{\alpha}\left(\frac{1}{r} S\right) $
 and $$\sup_{m\in
\NN}\left\|\sum_{|\alpha|\leq m} c_\alpha^{(i)}
 {g}_{\alpha} \left(\frac{1}{r} S\right)\right\|<\infty,\qquad i=1,\ldots,n;$$
 \item[(b)]
  $\frac{1}{r}S_i=\text{\rm SOT-}\lim_{\gamma\to 1}\sum_{k=0}^\infty \sum_{|\alpha|=k} c_\alpha^{(i)}
  \gamma^{|\alpha|}
 {g}_{\alpha}\left(\frac{1}{r} S\right)$  and
 $$\sup_{\gamma\in
  [0,1)}\left\|\sum_{k=0}^\infty \sum_{|\alpha|=k} c_\alpha^{(i)}
  \gamma^{|\alpha|}
 {g}_{\alpha}\left(\frac{1}{r} S\right)\right\|<\infty, \qquad i=1,\ldots,n. $$
\end{enumerate}
As  in the proof of Proposition \ref{ball}, using    the fact that
the noncommutative Poisson transform $P_{rX}$, $r\in (\delta,1)$, is
$SOT$-continuous on bounded sets, we deduce that
  that $X_j=f_j(g(X))$ for $j=1,\ldots,n$. This also shows that
$g$ is one-to-one on $[B(\cH)^n]_1^-$. On the other hand, the
relation above implies $Y=g(X)=g(f(g(X)))=g(f(Y))$ and $\|f(Y)\|\leq
1$, which shows that $Y\in \BB_f(\cH)$. Therefore,
$\BB_f(\cH)=g\left([B(\cH)^n]_1^-\right)$ and $f$ is one-to-one on
$\BB_f(\cH)$. Hence, we also deduce that
$g([B(\cH)^n]_1^c)=\BB_f^c(\cH)$. Similarly, one can prove this
proposition when condition $(ii)$ or $(iii)$ holds. The proof is
complete.
\end{proof}

\begin{theorem}\label{rep1} Let  $f=(f_1,\ldots, f_n)$ be    an $n$-tuple of formal power
series with the model property and let $(M_{Z_1},\ldots, M_{Z_n})$
be the universal model associated with $\BB_f$. If  $f\in \cM^{||}$
and $\pi$ is a $*$-representation of $C^*(M_{Z_1},\ldots, M_{Z_n})$,
then
$$
[f_1(\pi(M_{Z_1}),\ldots, \pi(M_{Z_n})), \ldots,
f_n(\pi(M_{Z_1}),\ldots, \pi(M_{Z_n}))]
$$
is a row isometry.

Conversely, if $f\in \cM_{rad}^b$  or $f\in \cM_{rad}\cap \cM^{||}$, and $[W_1,\ldots, W_n]\in B(\cK)^n$
is a row isometry, then  there is a unique
 $*$-representation $\pi:C^*(M_{Z_1},\ldots,
M_{Z_n})\to B(\cK)$ such that $ \pi(M_{Z_i})=g_i(W_1,\ldots, W_n)$,
$ i=1,\ldots,n$, where $g=(g_1,\ldots,g_n)$ is the inverse of $f$
with respect to the composition.
\end{theorem}

\begin{proof}
Let   $f_i$ have the representation $f_i=\sum_{\alpha\in \FF_n^+}
c_\alpha^{(i)} Z_\alpha$. Assuming that  $f\in \cM^{||}$, we deduce
that $ f_i(M_{Z_1},\ldots, M_{Z_n})=\sum_{k=o}^\infty
\sum_{|\alpha|=k} c_\alpha^{(i)} M_{Z_\alpha}$ or $
f_i(M_{Z_1},\ldots, M_{Z_n})=\lim_{r\to 1} \sum_{k=0}^\infty
\sum_{|\alpha|=k} r^{|\alpha|} c_\alpha^{(i)} M_{Z_\alpha}$ where
the convergence is in the operator norm topology. In either case, if
$\pi:C^*(M_{Z_1},\ldots, M_{Z_n})\to B(\cK)$  is a
$*$-representation, we have
$$
 \pi(f_i(M_{Z_1},\ldots, M_{Z_n}))=f_i(\pi(M_{Z_1}),\ldots,
\pi(M_{Z_n})),\qquad i=1,\ldots,n,
$$
and, taking into account that $f_i(M_{Z_1},\ldots,
M_{Z_n})=M_{f_i}$, we obtain
\begin{equation*}
 f_i(\pi(M_{Z_1}),\ldots,
\pi(M_{Z_n}))^*f_i(\pi(M_{Z_1}),\ldots, \pi(M_{Z_n})) =\pi(M_{f_i}^*
M_{f_i})=\delta_{ij}I_\cK
\end{equation*}
for any $i,j=1,\ldots,n$. Therefore, $ [f_1(\pi(M_{Z_1}),\ldots,
\pi(M_{Z_n})), \ldots, f_n(\pi(M_{Z_1}),\ldots, \pi(M_{Z_n}))] $ is
a row isometry.

Conversely, assume that $f\in \cM_{rad}^b$  or $f\in \cM_{rad}\cap \cM^{||}$ and $[W_1,\ldots, W_n]\in
B(\cK)^n$ is a row isometry. Let $g=(g_1,\ldots, g_n)$ be the
inverse of $f$ with respect to the composition and let
$g_i:=\sum_{\alpha\in \FF_n^+} a_\alpha^{(i)} Z_\alpha$. Since $f$ has the radial approximation property, we deduce that $g$  is a free holomorphic function on
$[B(\cH)^n]_\gamma$ for some $\gamma >1$. Consequently,
$g_i(W_1,\ldots, W_n)=\sum_{k=0}^\infty \sum_{|\alpha|=k}
a_\alpha^{(i)} W_\alpha$, where the convergence is in the operator
norm topology. Applying Proposition \ref{Mb}, we deduce that
$(g_1(W),\ldots, g_n(W))\in \BB_f(\cK)$ and $f_i(g_1(W),\ldots,
g_n(W))=W_i$ for $i=1,\ldots, n$.

According to \cite{Po-multi}, since $f=(f_1,\ldots, f_n)$  has the
radial approximation property
   and $(g_1(W),\ldots, g_n(W))\in \BB_f(\cK)$,    there is
    a unique unital completely contractive linear map
$$
\pi: C^*(M_{Z_1},\ldots, M_{Z_n})\to B(\cK)
$$
such that
  \begin{equation}
 \label{pig}
\pi(M_{Z_\alpha} M_{Z_\beta}^*)=g_\alpha (W)g_\beta(W)^*, \qquad
\alpha,\beta\in \FF_n^+.
\end{equation}
On the other hand, since $[W_1,\ldots, W_n]\in B(\cK)^n$ is a row
isometry, we deduce that
\begin{equation*}
\begin{split}
\left< g_i(W)^* g_j(W)x,y\right>&=
\left<\sum_{k=0}^\infty\sum_{|\beta|=k} a_\beta^{(j)} W_\beta
x,\sum_{k=0}^\infty\sum_{|\alpha|=k} a_\alpha^{(i)} W_\alpha
y\right>\\
&=  \lim_{m\to\infty} \left<\sum_{|\alpha|\leq
m}\sum_{k=0}^\infty\sum_{|\beta|=k} \overline{a_\alpha^{(i)}}
a_\beta^{(j)} W_\alpha^*W_\beta x,
y\right>\\
&=  \lim_{m\to\infty} \left<\sum_{|\alpha|\leq
m}\sum_{k=0}^\infty\sum_{|\beta|=k} \overline{a_\alpha^{(i)}}
a_\beta^{(j)}  \delta_{\alpha \beta} x,
y\right>\\
&= \lim_{m\to\infty}\sum_{|\alpha|\leq m}\overline{a_\alpha^{(i)}}
a_\alpha^{(j)} \left<x,y\right>\\
&= \sum_{k=0}^\infty \sum_{|\alpha|=k}\overline{a_\alpha^{(i)}}
a_\alpha^{(j)} \left<x,y\right>\\
&=\left< Z_j,Z_i\right>_{\HH^2(f)}\left<x, y\right>
\end{split}
\end{equation*}
for any   $x,y\in \cK$. Hence, and using the fact  that
$$
M_{Z_i}^* M_{Z_j}=\left< Z_j,Z_i\right>_{\HH^2(f)}
I_{\HH^2(f)},\qquad i,j\in \{1,\ldots,n\},
$$
we deduce that $\pi(M_{Z_i}^* M_{Z_j})=\pi(M_{Z_i})^*\pi(M_{Z_j})$.
Therefore,  taking into account  relation \eqref{pig} and the fact
that $C^*(M_{Z_1},\ldots,
 M_{Z_n})$   coincides with
 $
 \overline{\text{\rm span}} \{M_{Z_\alpha} M_{Z_\beta}^*:\ \alpha,\beta\in
 \FF_n^+\},
 $
we conclude that $\pi$ is a $*$-representation of
$C^*(M_{Z_1},\ldots,
 M_{Z_n})$. The
proof is complete.
\end{proof}

\begin{corollary}\label{rep2}
Let $f=(f_1,\ldots, f_n)$ be    an $n$-tuple of formal power series
in the set $\cM_{rad}\cap \cM^{||}$. Then  any  $*$-representation
$\pi:C^*(M_{Z_1},\ldots, M_{Z_n})\to B(\cK)$ is generated  by a row
isometry $[W_1,\ldots, W_n]$, $W_i\in B(\cK)$, such that
$$
\pi(M_{Z_i})=g_i(W_1,\ldots, W_n),\qquad i=1,\ldots,n.
$$
\end{corollary}

We remark that, in the particular case  when $f\in \cM_{rad}\cap
\cM^{||}$, one can use Corollary \ref{rep2} and the Wold
decomposition for isometries with orthogonal ranges
\cite{Po-isometric}, to provide another proof of Theorem \ref{wold}.

Let $T:=(T_1,\ldots, T_n)\in \BB_f(\cH)$. We say that an $n$-tuple
$V:=(V_1,\ldots, V_n)$ of operators on a Hilbert space $\cK\supseteq
\cH$ is a minimal dilation of $T$ if the following properties are
satisfied:
\begin{enumerate}
\item[(i)] $(V_1,\ldots, V_n)\in \BB_f(\cK)$;
\item[(ii)] there  is a $*$-representation
$\pi:C^*(M_{Z_1},\ldots, M_{Z_n})\to B(\cK)$ such that
$\pi(M_{Z_i})=V_i$, $i=1,\ldots,n$;
\item[(iii)]$V_i^*|_\cH=T_i^*$, $i=1,\ldots,n$;
\item[(iv)] $\cK=\bigvee_{\alpha\in \FF_n^+} V_\alpha \cH$.
\end{enumerate}
Without the condition (iv), the $n$-tuple $V$ is called dilation of
$T$. We remark that if $f\in \cM_{rad}\cap \cM^{||}$, then the
condition (i) is a consequence of (ii).

\begin{theorem}\label{dilation} Let $f=(f_1,\ldots, f_n)$ be    an $n$-tuple of formal
 power series   in the set $\cM_{rad}^b$ or $\cM_{rad}\cap \cM^{||}$
 and let $g=(g_1,\ldots, g_n)$ be its
 inverse with respect to the composition. If $T=(T_1,\ldots,T_n)$ is
 in the noncommutative domain
 $\BB_f(\cH)$, then  it has a minimal dilation which is unique up to an isomorphism. Moreover,
  its minimal dilation coincides with  $(g_1(W),\ldots, g_n(W))$,
 where $W=(W_1,\ldots, W_n)$ is
 the minimal isometric dilation  of the row contraction
 $(f_1(T ), \ldots, f_n(T ))$
 on a Hilbert space $\cK\supseteq \cH$.
\end{theorem}
\begin{proof}
Since $(f_1(T ), \ldots, f_n(T ))$ is a row contraction, according
to \cite{Po-isometric}, there is a minimal isometric dilation
$W=(W_1,\ldots, W_n)\in B(\cK)^n$ with $\cK\supseteq \cH$.
Therefore, we have $W_i^* W_j=\delta_{ij} I_\cK$,
$W_i^*|_\cH=f_i(T)^*$ for $i=1,\ldots,n$, and
$\cK=\bigvee_{\alpha\in \FF_n^+} W_\alpha \cH$. Applying Proposition
\ref{Mb}, we deduce that $(g_1(W),\ldots, g_n(W))\in \BB_f(\cK)$ and
$f_i(g_1(W),\ldots, g_n(W))=W_i$ for $i=1,\ldots, n$. Since $f\in
\cM_{rad}$, for each $i=1,\ldots,n$, we have that
$g_i=\sum_{\alpha\in \FF_n^+} a_\alpha^{(i)} Z_\alpha$ is a free
holomorphic function on $[B(\cH)^n]_\gamma$ for  some $\gamma
>1$. Consequently, $g_i(W)=\sum_{k=0}^\infty \sum_{|\alpha|=k}
a_\alpha^{(i)} W_\alpha$ and $g_i(f(T))=\sum_{k=0}^\infty
\sum_{|\alpha|=k} a_\alpha^{(i)} [f(T)]_\alpha$  where the
convergence is in the operator norm. Hence, and using the fact that
$W_i^*|_\cH=f_i(T)^*$, we obtain
$$
g_i(W)^*|_\cH=g_i(f(T))^*|_\cH,\qquad i=1,\ldots,n.
$$
Now, using  relation $T_i=g_i(f(T))$,   we deduce that
$g_i(W)^*|_\cH=T_i^*|_\cH$ for $i=1,\ldots,n$. Note also that
$\bigvee_{\alpha\in \FF_n^+} [g(W)]_\alpha \cH\subseteq
\bigvee_{\alpha\in \FF_n^+} W_\alpha \cH=\cK$. To prove the reverse
inclusion, one can use the relation $f_i(g_1(W),\ldots, g_n(W))=W_i$
for $i=1,\ldots, n$, and the fact that $(g_1(W),\ldots, g_n(W))$ is
either in the convergence set $\cC_f^{SOT}(\cK)$ or
$\cC_f^{rad}(\cK)$. The fact that there  is a $*$-representation
$\pi:C^*(M_{Z_1},\ldots, M_{Z_n})\to B(\cK)$ such that
$\pi(M_{Z_i})=g_i(W)$ for any $i=1,\ldots,n$, follows from Theorem
\ref{rep1}.

 To prove the uniqueness,  let $V=(V_1,\ldots, V_n)$ and $V'=(V_1',\ldots, V_n')$ be
 two minimal dilations of $T=(T_1,\ldots, T_n)\in \BB_f(\cH)$ on the Hilbert space
  $\cK\supseteq \cH$ and $\cK'\supseteq \cH$, respectively.
 Let $\alpha:=g_{i_1}\cdots g_{i_k}\in \FF_n^+$ and $\beta:=g_{j_1}\cdots g_{j_p}\in \FF_n^+$.
  Note that  $V_i^*V_j=\pi(M_{Z_i}^*)\pi(M_{Z_j})=\left<Z_j,
  Z_i\right>_{\HH^2(f)}I_\cK$ for
  $i,j=1,\ldots,n$. Consequently, if $k>p$ and $h^{(\beta)}, k^{(\alpha)}\in \cH$, then

 \begin{equation*}
 \begin{split}
 \left<V_\alpha^* V_\beta h^{(\beta)}, k^{(\alpha)}\right>
 &=\left< \left<Z_{j_1}, Z_{i_1}\right>\cdots \left<Z_{j_p}, Z_{i_p}\right>
 V_{i_k}^*\cdots V_{i_{p+1}}^*h^{(\beta)}, k^{(\alpha)}\right>\\
 &=\left<Z_{j_1}, Z_{i_1}\right>\cdots \left<Z_{j_p}, Z_{i_p}\right>
 \left<T_{i_k}^*\cdots T_{i_{p+1}}^*h^{(\beta)}, k^{(\alpha)}\right>.
 \end{split}
 \end{equation*}
When $p>k$, we obtain
 $$
 \left<V_\alpha^* V_\beta h^{(\beta)}, k^{(\alpha)}\right>
 =
 \left<Z_{j_1}, Z_{i_1}\right>\cdots \left<Z_{j_k}, Z_{i_k}\right>
 \left<h^{(\beta)}, T_{i_{k+1}}\cdots T_{i_{p}}k^{(\alpha)}\right>,
 $$
 and, if $k=p$,
we have
  $
 \left<V_\alpha^* V_\beta h^{(\beta)}, k^{(\alpha)}\right>
 =
 \left<Z_{j_1}, Z_{i_1}\right>\cdots \left<Z_{j_k}, Z_{i_k}\right>
 \left<h^{(\beta)}, k^{(\alpha)}\right>.
 $
 Similar relations hold for the minimal dilation  $V'=(V_1',\ldots, V_n')$.
 Hence, and taking into account that the dilations are minimal, i.e.,
 $\cK=\bigvee_{\alpha\in \FF_n^+} V_\alpha \cH$ and $\cK'=\bigvee_{\alpha\in \FF_n^+} V_\alpha' \cH$, one can easily see that there is a unitary operator
 $U:\cK\to \cK'$ such that
 $U\left(\sum_{|\alpha|\leq m}V_\alpha h^{(\alpha)}\right)=
 \sum_{|\alpha|\leq m}V_\alpha' h^{(\alpha)}
 $
 for any $h^{(\alpha)}\in \cH$, $|\alpha|\leq m$, and $m\in \NN$. Consequently, we  deduce that $UV_i=V_i'U$ for any $i=1,\ldots,n$.
 The proof is complete.
\end{proof}

Using Theorem \ref{dilation} and  Theorem \ref{rep1}, we deduce the
following result.

\begin{corollary}\label{conseq}  Let $f=(f_1,\ldots, f_n)$ be an  $n$-tuple of formal power
series  with the model property and let $T=(T_1,\ldots,T_n)\in
\BB_f(\cH)$.
\begin{enumerate}
 \item[(i)]
  If $f\in \cM_{rad}^b$   or $ f\in \cM_{rad}\cap
\cM^{||}$,   then $(g_1(W),\ldots, g_n(W))$ is a dilation of
  $T=(T_1,\ldots, T_n)$, where   $W=(W_1,\ldots,
  W_n)$ is an isometric dilation of  the row contraction $(f_1(T),\ldots,
  f_n(T))$.
\item[(ii)]
If   $ f\in  \cM^{||}$
  and
  $V=(V_1,\ldots, V_n)\in \BB_f(\cK)$ is a dilation of $T$, then
  $(f_1(V),\ldots, f_n(V))$ is  an isometric dilation of $(f_1(T),\ldots,
  f_n(T))$.

  \end{enumerate}
\end{corollary}

We remark that under the conditions and notations  of Theorem
\ref{funct-model2} (see also the  proof), if  $f\in \cM_{rad}^b$,
$~\Theta~$ is purely contractive, and
\begin{equation*}
\overline{\Delta_\Theta(\HH^2(f)\otimes
\cE_*)}=\overline{\Delta_\Theta(\vee_{i=1}^n M_{f_i}(\HH^2(f)\otimes
\cE_*))},
\end{equation*}
then,   considering $~\bf H~$ as a subspace of
$$
{\bf K}:= (\HH^2(f)\otimes {\cE})\oplus \overline{\Delta_\Theta
(\HH^2(f)\otimes \cE_*))},
$$
one can prove that the sequence of operators ${\bf V}:=({\bf
V}_1,\ldots, {\bf V}_n)$
 defined on
$~\bf K~$ by
$$
{\bf V}_i:=(M_{Z_i}\otimes I_{\cE})\oplus D_i,\qquad i=1,\ldots,n,
$$
is the minimal   dilation of ${\bf T}:=({\bf T}_1,\ldots, {\bf
T}_n)\in \BB_f^{cnc}({\bf H})$.
Indeed, according to Theorem 4.1 from \cite{Po-charact},  the
multi-analytic operator $\Psi$ coincides with the characteristic
function $\Theta_{\bf A}$ of the row contraction ${\bf A}=[{\bf
A}_1,\ldots, {\bf A}_n]$. Moreover, the $n$-tuple ${\bf W}:=[{\bf
W}_1,\ldots, {\bf W}_n]$  defined on $\widetilde \cK$ is the minimal
isometric dilation of ${\bf A}$. Consequently,
  ${\bf W}:=[{\bf W}_1,\ldots, {\bf W}_n]$  is a row isometry,  ${\bf
W}_i^*|_{\widetilde \cH}={\bf A}_i^*$ for  $i=1,\ldots, n$, and
$\widetilde \cK=\bigvee_{\alpha\in \FF_n^+} {\bf W}_\alpha
\widetilde \cH$. On the other hand, since $f=(f_1,\ldots, f_n)$ has
the radial approximation property, we have $g_i({\bf W}_1,\ldots,
{\bf W}_n)=g_i(S_1,\ldots, S_n)\oplus g_i(C_1,\ldots, C_n)$, where
the convergence defining these operators is in the operator norm
topology. Now, it easy to see that $g_i({\bf W}_1,\ldots, {\bf
W}_n)^*|_{\widetilde \cH}=g_i({\bf A}_1,\ldots, {\bf
A}_n)^*=\TT_i^*$ and $\widetilde \cK=\bigvee_{\alpha\in \FF_n^+}
g_\alpha({\bf W}_1,\ldots, {\bf W}_n) \widetilde \cH$. Note also
that ${\bf V}_i=\Gamma^{-1} g_i({\bf W}_1,\ldots, {\bf W}_n)
\Gamma$, $i=1,\ldots,n$. Taking into account that $\Gamma=(U\otimes
\cE)\oplus (U\otimes I_{\cE_*})$ is a unitary operator with the
property that $\Gamma({\bf H})=\widetilde \cH$, we deduce that ${\bf
V_i}^*|_{\bf H}={\bf T}_i$, $i=1,\ldots, n$ and ${\bf
K}=\bigvee_{\alpha\in \FF_n^+} {\bf V}_\alpha {\bf H}$. Using
Theorem \ref{dilation}, we deduce that $({\bf V}_1,\ldots, {\bf
V}_n)$ is the minimal dilation of $({\bf T}_1,\ldots, {\bf T}_n)$,
which proves our assertion.

 In what follows, we provide a commutant lifting theorem  for the noncommutative domains $\BB_f$.

\begin{theorem}\label{CLT} Let $f=(f_1,\ldots, f_n)$ be    an $n$-tuple of formal power series
in the set $\cM_{rad}\cap \cM^{||}$. Let $T=(T_1,\ldots, T_n)\in
\BB_f(\cH)$ and  $T'=(T_1',\ldots, T_n')\in \BB_f(\cH')$, and let
$V=(V_1,\ldots, V_n)\in \BB_f(\cK)$ and  $V'=(V_1',\ldots, V_n')\in
\BB_f(\cK')$ be  dilations  of $T$ and $T'$, respectively. If
$X:\cH\to \cH'$  is bounded operator  satisfying  the intertwining
relations $XT_i=T_i'X$ for any  $i=1,\ldots n$, then there exists a
bounded operator $Y:\cK\to \cK'$ with the following properties:
\begin{enumerate}
\item[(i)] $YV_i=V_i'Y$ for any  $i=1,\ldots n$;
\item[(ii)] $Y^*|_\cH=X^*$ and $\|X\|=\|Y\|$.
\end{enumerate}

\end{theorem}
\begin{proof} According to Corollary \ref{conseq}, since $f\in \cM_{rad}\cap
\cM^{||}$, there are isometric dilations of any $n$-tuple
$T=(T_1,\ldots, T_n)\in \BB_f(\cH)$. Let $V=(V_1,\ldots,V_n)\in
\BB_f(\cK)$ be a dilation of $T$. Then $V_i^*|_\cH=T_i^*$,
$i=1,\ldots,n$, and there is a $*$-representation
$\pi:C^*(M_{Z_1},\ldots, M_{Z_n})\to B(\cK)$ such that
$\pi(M_{Z_i})=V_i$, $i=1,\ldots,n$. Let $f_i$ have the
representation  $f_i=\sum_{\alpha\in \FF_n^+} c_\alpha^{(i)}
Z_\alpha$. Taking into account that $f\in \cM^{||}$, we deduce that
$ f_i(M_{Z_1},\ldots, M_{Z_n})=\sum_{k=o}^\infty \sum_{|\alpha|=k}
c_\alpha^{(i)} M_{Z_\alpha}$ or $ f_i(M_{Z_1},\ldots,
M_{Z_n})=\lim_{r\to 1} \sum_{k=0}^\infty \sum_{|\alpha|=k}
r^{|\alpha|} c_\alpha^{(i)} M_{Z_\alpha}$, where the convergence is
in the operator norm topology. Consequently, we have
$$
f_i(V_1,\ldots, V_n)=f_i(\pi(M_{Z_1}),\ldots,
\pi(M_{Z_n}))=\pi(f_i(M_{Z_1},\ldots, M_{Z_n}),\qquad i=1,\ldots,n,
$$
and
\begin{equation*}
\begin{split}
f_i(V_1,\ldots, V_n)^*f_i(V_1,\ldots, V_n)&=\pi(f_i(M_{Z_1},\ldots,
M_{Z_n}^*f_i(M_{Z_1},\ldots, M_{Z_n}))\\
&=\pi(M_{f_i}^* M_{f_i})=\delta_{ij}I_\cK
\end{split}
\end{equation*}
for any $i,j=1,\ldots,n$. Since $V_i^*|_\cH=T_i^*$, we also have
$f_i(V_1,\ldots, V_n)^*|_\cH= f_i(T_1,\ldots,T_n)$ for any
$i=1,\ldots,n$. This shows that $(f_1(V),\ldots, f_n(V))$ is an
isometric dilation of  $(f_1(T),\ldots, f_n(T))$. A similar results
holds if $V'=(V_1',\ldots, V_n')\in \BB_f(\cK')$ is a dilation  of
$T'=(T_i',\ldots, T_n')\in \BB_f(\cH')$. Now, we assume that
$X:\cH\to \cH'$  is a bounded operator such that $XT_i=T_i'X$ for
$i=1,\ldots,n$. Hence, we deduce that $Xf_i(T)=f_i(T')X$,
$i=1,\ldots,n$. Applying the noncommutative commutant lifting
theorem for row contractions \cite{Po-isometric}, we find an
operator $Y:\cK\to \cK'$ such that $Yf_i(V)=f_i(V')Y$ for
$i=1,\ldots,n$, $Y^*|_{\cH'}=X^*$,  and $\|Y\|=\|X\|$. Consequently,
we have $Yg_i(f(V))=g_i(f(V'))Y$, $i=1,\ldots,n$. Since
$g_i(f(V))=V_i$ and $g_i(f(V'))=V'$, we conclude that  $YV_i=V_i'Y$
for any $i=1,\ldots,n$. The proof is complete.
\end{proof}

\bigskip

\section{Noncommutative varieties, constrained characteristic functions, and operator models}

We present operator models, in terms of constrained characteristic
functions, for
  $n$-tuples of operators in    noncommutative
varieties $\cV_{f,J}^{cnc}(\cH)$  associated with WOT-closed
two-sided ideals $J$ of the Hardy algebra $H^\infty(\BB_f)$. This is
used to show that the constrained characteristic function
$\Theta_{f,T,J}$ is a complete unitary invariant for
$\cV_{f,J}^{cnc}(\cH)$.

Let $f=(f_1,\ldots, f_n)$ be an $n$-tuple of formal power  series
with the model property. The noncommutative Hardy algebra
$H^\infty(\BB_f)$ is  the WOT-closure of all
  noncommutative polynomials in $M_{Z_1},\ldots, M_{Z_n}$ and the
  identity.
According to  \cite{Po-multi},    $J$  is a WOT-closed two-sided
ideal of $H^\infty(\BB_f)$ if and only if there is a WOT-closed
two-sided ideal $\cI$ of the noncommutative analytic Toeplitz
algebra $F_n^\infty$ such that
$$
J=\{\varphi(f(M_Z)): \  \varphi\in \cI\}.
$$
We mention that if $\varphi(S_1,\ldots,S_n)\in F_n^\infty$ has the
Fourier representation $\varphi(S_1,\ldots,S_n)=\sum_{\alpha\in
\FF_n^+} c_\alpha S_\alpha$, then $$\varphi(f(M_Z)):=\text{\rm
SOT-}\lim_{r\to 1}\sum_{k=0}^\infty \sum_{|\alpha|=k} c_\alpha
r^{|\alpha|} [f(M_Z)]_\alpha
$$
exists.
 Denote by $H^\infty(\cV_{f,J})$   the WOT-closed algebra
generated
 by the operators
 $B_i:=P_{\cN_{f,J}} M_{Z_i} |_{\cN_{f,J}}$,  for  $i=1,\ldots, n$, and the identity, where
 $$
 \cN_{f,J}:= \HH^2(f)\ominus \cM_{f,J}\quad \text{ and }\quad
 \cM_{f,J}:=\overline{ J \HH^2(f)}.
 $$
We recall that  the map
$$\Gamma:H^\infty(\BB_f)/J\to
B(\cN_{f,J}) \quad \text{ defined by } \quad
\Gamma(\varphi+J)=P_{\cN_{f,J}} \varphi|_{\cN_{f,J}}
$$
is a   completely isometric representation. Since the set of all
polynomials in $M_{Z_1},\ldots, M_{Z_n}$  and the identity is
WOT-dense in $H^\infty(\BB_f)$,  we can conclude that
$P_{\cN_{f,J}}H^\infty(\BB_f)|_{\cN_{f,J}}$ is  a WOT-closed
subalgebra of $B(\cN_{f,J})$ and, moreover,
$H^\infty(\cV_{f,J})=P_{\cN_{f,J}}H^\infty(\BB_f)|_{\cN_{f,J}}$.

Given an $n$-tuple  $T=(T_1,\ldots,T_n)\in \BB_f(\cH)$,
  we consider the defect operator
$$
\Delta_{f,T} :=\left(I-\sum_{j=1}^nf_j(T)f_j(T)^*\right)^{1/2}
$$
and the defect space $\cD_{f,T}:=\overline{\Delta_{f,T} \cH}$.
Define the noncommutative Poisson kernel $K_{f,T}:\cH\to
\HH^2(f)\otimes \cD_{f,T}$ by setting
\begin{equation}
\label{KfT} K_{f,T}h:=\sum_{\alpha\in \FF_n^+} f_\alpha\otimes
\Delta_{f,T}  [f (T)]_\alpha^*h,\qquad h\in \cH.
\end{equation}
Let $J\neq H^\infty(\BB_f)$ be a  WOT-closed two-sided ideal of
$H^\infty(\BB_f)$. The {\it constrained Poisson kernel} associated
with $f$, $T$, and $J$ is the operator $K_{f, T,J}:\cH\to
\cN_{f,J}\otimes \cD_{f,T} $ defined by
$$
K_{f, T,J}:=(P_{\cN_{f,J}}\otimes  I_{ \cD_{f,T}}) K_{f,T}.
$$
We   remark that if $\psi=\varphi(f(M_Z))$ for some
$\varphi(S_1,\ldots,S_n)=\sum_{\alpha\in \FF_n^+} c_\alpha S_\alpha$
in the noncommutative analytic Toeplitz algebra $F_n^\infty$, and
$T:=(T_1,\ldots, T_n)$ is a c.n.c. $n$-tuple in $\BB_f(\cH)$, then
$f(T)=(f_1(T),\ldots, f_n(T))$ is a c.n.c. row contraction and, due
to the $F_n^\infty$-functional calculus  \cite{Po-funct}, the limit
$$
\psi(T_1,\ldots, T_n):= \text{\rm SOT-}\lim_{r\to
1}\sum_{k=0}^\infty \sum_{|\alpha|=k} c_\alpha r^{|\alpha|}
[f(T)]_\alpha
$$
exists. Therefore, we can talk about an $H^\infty(\BB_f)$-functional
calculus for the $n$-tuples of operators in $\BB_f^{cnc}(\cH)$. We
introduce the noncommutative variety $\cV_{f,J}^{cnc}(\cH)\subset
\BB_f(\cH)$ defined by
$$
\cV_{f,J}^{cnc}(\cH):=\left\{ (T_1,\ldots, T_n)\in
\BB_f^{cnc}(\cH):\ \psi(T_1,\ldots,T_n)=0 \ \text{ for any } \
\psi\in J\right\}.
$$
Note that $\psi(B_1,\ldots, B_n)=0$ for any $\psi\in J$. The $n$-tuple
 $B=(B_1,\ldots, B_n)\in \cV_{f,J}^{cnc}(\cN_{f,J})$ will play the role of universal
  model for the noncommutative variety $\cV_{f,J}^{cnc}$.

\begin{proposition}\label{KJT}  Let $f=(f_1,\ldots, f_n)$ be an  $n$-tuple of formal power
series  with the model property and  let $J\neq H^\infty(\BB_f)$ be
a WOT-closed two-sided  ideal of $H^\infty(\BB_f)$. If
$T:=(T_1,\ldots, T_n)$ is an $n$-tuple in $\cV_{f,J}^{cnc}(\cH)$,
then
$$
K_{f,T,J} T_i^*=(B_i^*\otimes I_{ \cD_{f,T}})K_{f, T,J}, \qquad
i=1,\ldots, n,
$$
 and
\begin{equation*}
K_{f,T,J}^* K_{f,T,J}=I_\cH-\text{\rm SOT-}\lim_{q\to \infty}
\sum_{\alpha\in \FF_n,\, |\alpha|=q} [f (T)]_\alpha [f
(T)]_\alpha^*,
\end{equation*}
where $K_{f,T,J}$ is the constrained Poisson kernel associated with
$f$,  $T$, and $J$.
\end{proposition}
\begin{proof}
The fact  that

\begin{equation}
\label{Poisson-inter} K_{f,T}T_i^*=(M_{Z_i}^*\otimes
I_{\cD_{f,T}})K_{f,T}, \qquad i=1,\ldots,n,
\end{equation}
was proved in Theorem 4.1 from \cite{Po-multi}.
 Now,  we show that
$$
K_{f,T,J} T_i^*=(B_i^*\otimes I_{ \cD_{f,T}})K_{f, T,J}, \qquad
i=1,\ldots, n,
$$
where $K_{f,T,J}$ is the constrained Poisson kernel associated with
$f$,  $T$, and $J$.
 Note that, due to  relation \eqref{Poisson-inter},
  we have
 \begin{equation}\label{ker-pol}
 K_{f,T}^*(p(M_{Z_1},\ldots, M_{Z_n})\otimes I_{\cD_{f,T}}) =p(T_1,\ldots, T_n)K_{f,T}^*
 \end{equation}
 for any   polynomial  $p$ in $M_{Z_1},\ldots, M_{Z_n}$.
 If $\varphi(S_1,\ldots, S_n):= \sum\limits_{k=0}^\infty
 \sum\limits_{|\alpha|=k} a_\alpha S_\alpha$
 is
 in $F_n^\infty$, then, for any $0<r<1$,
 $\varphi_r(S_1,\ldots, S_n):= \sum\limits_{k=0}^\infty \sum\limits_{|\alpha|=k}
 r^{|\alpha|}a_\alpha S_\alpha$  is in the noncommutative disc algebra $\cA_n$.
 Consequently,
  $$
 \lim_{m\to \infty} \sum_{k=0}^m \sum_{|\alpha|=k}
  r^{|\alpha|} a_\alpha [f(M_Z)]_\alpha = \varphi_r(f(M_Z))
 $$
  in the norm topology and,  using relation
   \eqref{ker-pol} we obtain
  \begin{equation*}
 K_{f,T}^*[ \varphi_r(f(M_Z))\otimes I_{  \cD_{f,T}}] =\varphi_r(f(T))K_{f,T}^*
 \end{equation*}
for any $\varphi(S_1,\ldots, S_n)\in F_n^\infty$ and $0<r<1$. Since
$T\in \BB_f^{cnc}(\cH)$   and $M_Z:=(M_{Z_1},\ldots, M_{Z_n})\in
\BB^{pure}(\HH^2(f))$, we can
 use the $F_n^\infty$-functional calculus for row contractions.
  We recall that the map $A\mapsto A\otimes I$ is SOT-continuous on bounded sets of $B(F^2(H_n))$ and,
   due to the noncommutative von Neumann inequality \cite{Po-von},
    we have $\|\varphi_r(f(M_Z))\|\leq \|\varphi(S_1,\ldots, S_n)\|$. Therefore,  we can
 take  $r\to 1$ in
 the equality  above and
  obtain
 \begin{equation*}
K_{f,T}^*( \varphi(f(M_Z))\otimes I_{  \cD_{f,T}})
=\varphi(f(T))K_{f,T}^*
 \end{equation*}
 for any $\varphi(S_1,\ldots, S_n)\in F_n^\infty$. Consequently, we
 have
\begin{equation}
\label{kft} K_{f,T}^*( \psi\otimes I_{  \cD_{f,T}}) =
\psi(T)K_{f,T}^*, \qquad \psi \in H^\infty(\BB_f),
\end{equation}
which implies
\begin{equation*}
\left<(\psi ^*\otimes I_{  \cD_{f,T}})K_{f,T}h,1\otimes k\right>=
 \left<K_{f,T}\psi(T )^*h,1\otimes k\right>
 \end{equation*}
 for any $\psi\in H^\infty(\BB_f)$,\  $h\in \cH$, and $k\in  \cD_{f,T} $.
 Hence, and using the fact that
 $\psi(T)=0$ for $\psi\in J$, we obtain
 $\left<K_{f,T}h, \psi(1)\otimes k\right>=0$ for any $h\in \cH$ and $k\in  \cD_{f,T} $.
 Taking into account the definition of $\cM_{f,J}$, we deduce that
 $K_{f,T}(\cH)\subseteq \cN_{f,J}\otimes \cD_{f,T}$. This shows that the constrained Poisson
 kernel $K_{f,T,J}$
  satisfies the relation
 \begin{equation}
 \label{KTJ}
 K_{f,T,J}h=\left(P_{ \cN_{f,J}}\otimes I_{\cD_{f,T}}\right) K_{f,T}h=K_{f,T} h, \qquad h\in \cH.
 \end{equation}
 Since $J$ is a left ideal of $H^\infty(\BB_f)$, $\cN_{f,J}$ is an invariant subspace
 under each operator $M_{Z_1}^*,\ldots, M_{Z_n}^*$ and therefore
  $B_\alpha=P_{ \cN_{f,J}}M_{Z_\alpha}|\cN_{f,J}$ for
 $\alpha\in \FF_n^+$.
 Since $(B_1,\ldots, B_n)\in \BB^{pure}(\cN_{f,J})$, we can use
the $H^\infty(\BB_f)$-functional calculus to deduce that

 \begin{equation}
 \label{fnn}
 \chi(B_1,\ldots, B_n)=P_{ \cN_{f,J}}\chi(M_{Z_1},\ldots,
 M_{Z_n})|_{\cN_{f,J}}
 \end{equation}
 for any $\chi\in H^\infty(\BB_f)$.
Taking into account relations \eqref{kft},
 \eqref{KTJ}, and \eqref{fnn},
we obtain
\begin{equation*}
 \begin{split}
K_{f,T,J}\chi(T_1,\ldots, T_n)^*&=\left(P_{\cN_{f,J}}\otimes
I_{\cD_{f,T}}\right) [\chi(M_{Z_1},\ldots, M_{Z_n})^*\otimes
I_{\cD_{f,T}}]\left(P_{\cN_{f,J}}\otimes
 I_{\cD_{f,T}}\right)K_{f,T}\\
&=
 \left[\left(P_{ \cN_{f,J}}\chi(M_{Z_1},\ldots, M_{Z_n})|
 \cN_{f,J}\right)^*\otimes I_{\cD_{f,T}}\right] K_{f,T, J}\\
&= \left[\chi(B_1,\ldots, B_n)^*\otimes
I_{\cD_{f,T}}\right]K_{f,T,J}.
 \end{split}
 \end{equation*}
Therefore, we have
\begin{equation*}
 K_{f,T,J}\chi(T_1,\ldots, T_n)^*=\left[\chi(B_1,\ldots, B_n)^*\otimes
 I_{\cD_{f,T}}\right]K_{f,T}
 \end{equation*}
 for any $\chi(B_1,\ldots, B_n)\in H^\infty(\cV_{f,J})$. In
 particular, we have $
K_{f,T,J} T_i^*=(B_i^*\otimes I_{ \cD_{f,T}})K_{f, T,J}$ for
$i=1,\ldots,n$, which proves the first part of the proposition.

Due to relation \eqref{KTJ} and the definition of the Poisson kernel, we have
\begin{equation*}
\begin{split}
\left<K_{f,T,J}^*K_{f,T,J}h,h\right>&=\|K_{f,T}h\|^2=\lim_{q\to\infty}
\left\|\sum_{\alpha\in \FF_n^+, |\alpha|\leq q} f_\alpha\otimes
\Delta_{f,T} [f (T)]_\alpha^*h\right\|^2_{\HH^2(f)\otimes \cH}\\
&=\lim_{q\to\infty}\sum_{\alpha\in \FF_n^+, |\alpha|\leq q} \left<[f (T)]_\alpha\Delta_{f,T}^2 [f (T)]_\alpha^*h, h\right>\\
&=\|h\|-\lim_{q\to\infty}\left<\left(\sum_{\alpha\in \FF_n^+,
 |\alpha|=q} [f (T)]_\alpha [f (T)]_\alpha^*\right)h, h\right>
\end{split}
\end{equation*}
for any $h\in \cH$.
Since $[f_1(T),\ldots, f_n(T)]$ is a row contraction, the latter limit exists.
The proof is complete.
\end{proof}

\begin{corollary}\label{vN-variety} If $T=(T_1,\ldots, T_n)$ is  a pure $n$-tuple of operators in
$\cV_{f,J}^{cnc}(\cH)$, then
$$
 T_\alpha T_\beta^*=K_{f,T,J}^*[ B_\alpha B_\beta^* )\otimes
 I]K_{f,T,J},\qquad \alpha,\beta\in \FF_n^+,
$$
and
$$
 \left\|\sum_{i=1}^m q_i(T_1,\ldots, T_n)q_i(T_1,\ldots,
T_n)^*\right\| \leq \left\|\sum_{i=1}^m q_i(B_1,\ldots,
B_n)q_i(B_1,\ldots, B_n)\right\|
$$
for any     $q_i\in \CC[Z_1,\ldots,Z_n]$ and $m\in\NN$. Moreover,
$$
\|\chi(T_1,\ldots, T_n)\|\leq \|\chi(B_1,\ldots, B_n)\|
$$
for any $\chi\in H^\infty(\BB_f)$.
\end{corollary}

Let $J$ be a WOT-closed two-sided ideal of the Hardy  algebra
$H^\infty(\BB_f)$.   We define the {\it constrained characteristic
function} associated with an   $n$-tuple $T:=(T_1,\ldots, T_n)\in
\cV_{f,J}^{cnc}(\cH)$
  to be the multi-analytic operator with
respect to $B_1,\ldots, B_n$,
$$
\Theta_{f,T,J}(W_1,\ldots, W_n):\cN_{f,J}\otimes \cD_{f,T^*}\to
\cN_{f,J}\otimes \cD_{f,T},
$$
defined by the formal Fourier representation
$$ -I_{\cN_{f,J}}\otimes f(T)+
\left(I_{\cN_{f,J}}\otimes
\Delta_{f,T}\right)\left(I_{{\cN_{f,J}}\otimes \cH}-
\sum_{i=1}^n W_i\otimes f_i(T)^*\right)^{-1}\\
\left[W_1\otimes I_\cH,\ldots, W_n\otimes I_\cH \right]
\left(I_{\cN_{f,J}}\otimes \Delta_{f,T^*}\right),
$$
where $W_i:=P_{\cN_{f,J}}\Lambda_i|_{\cN_{f,J}}$, $i=1,\ldots,n$,
and $\Lambda_1,\ldots,\Lambda_n$ are the right multiplication
operators by the power series $f_i$ on the Hardy space $\HH^2(f)$.

\begin{theorem}\label{model}  Let $f=(f_1,\ldots, f_n)$ be an  $n$-tuple of formal power
series  with the model property and  let $J\neq H^\infty(\BB_f)$ be
a WOT-closed two-sided  ideal of $H^\infty(\BB_f)$.  If
$T:=(T_1,\ldots, T_n)$ is  an $n$-tuple of operators in  the
noncommutative variety $\cV_{f,J}^{cnc}(\cH)$,
  then $T $ is unitarily equivalent to the
  $n$-tuple ${\bf T}:=({\bf T}_1,\ldots, {\bf T}_n)$ in
$\cV_{f,J}^{cnc}({\bf H})$  where:
\begin{enumerate}
\item[(i)]  the Hilbert space ${\bf H}$ is defined by
\begin{equation*}
{\bf H}:=\left[\left(\cN_{f,J}\otimes \cD_{f,T}\right)\oplus
\overline{\Delta_{\Theta_{f,T,J}}(\cN_{f,J}\otimes \cD_{f,T^*})}\right]
\ominus\left\{\Theta_{f,T, J}f\oplus \Delta_{\Theta_{f,T,J}}f:\ f\in
\cN_{f,J}\otimes \cD_{f,T^*}\right\},
\end{equation*}
where $\Delta_{\Theta_{f,T,J}}:= \left(I-\Theta_{f,T,J}^*
\Theta_{f,T,J}\right)^{1/2}$;
 \item[(ii)] each operator ${\bf T}_i$,
\ $i=1,\ldots, n$,  is uniquely defined by the relation
$$
\left( P_{\cN_{f,J}\otimes \cD_{f,T}}|_{{\bf H}}\right) {\bf
T}_i^*x= (B_i^*\otimes I_{\cD_{f,T}})\left( P_{\cN_{f,J}\otimes
\cD_{f,T}}|_{{\bf H}}\right)x, \qquad x\in {\bf H},
$$
where $P_{\cN_{f,J}\otimes \cD_{f,T}}|_{{\bf H}}$ is a one-to-one
operator, $ P_{\cN_{f,J}\otimes \cD_{f,T}}$ is the orthogonal
projection of the Hilbert space $\left(\cN_{f,J}\otimes
\cD_{f,T}\right)\oplus \overline{\Delta_{\Theta_{f,T,J}}(\cN_{f,J}\otimes
\cD_{f,T^*})}$ onto the subspace $\cN_{f,J}\otimes \cD_{f,T}$, and
$B_i:=P_{\cN_{f,J}} M_{Z_i} |_{\cN_{f,J}}$  for  any  $i=1,\ldots,
n$.
\end{enumerate}

Moreover,  $T$ is  in  $\cV_{f,J}^{pure}(\cH)$ if and only if the
constrained characteristic function  $\Theta_{f,T, J}$ is a partial
isometry.
  In this case,   $T$ is unitarly equivalent to the $n$-tuple
\begin{equation*}
\left(P_ {{\bf H}} (B_1\otimes I_{\cD_{f,T}})|_{\bf H},\ldots,
P_{{\bf H}} (B_n\otimes I_{\cD_{f,T}})|_{\bf H}\right),
\end{equation*}
where $P_{{\bf H}}$ is the orthogonal projection of
$\cN_{f,J}\otimes \cD_{f,T}$ onto the Hilbert space
$${\bf H}=\left(\cN_{f,J}\otimes \cD_{f,T}\right)\ominus
\Theta_{f,T, J}(\cN_{f,J}\otimes \cD_{f,T^*}).
$$
 \end{theorem}

\begin{proof}
Taking into account that $\cN_{f,J}$ is  a co-invariant subspace
under $\Lambda_1,\ldots, \Lambda_n$, we can see that
\begin{equation*}
\begin{split}
\Theta_{T}(\Lambda_1,\ldots, \Lambda_n)^*(\cN_{f,J}\otimes\cD_{f,T})
&\subseteq \cN_{f,J}\otimes \cD_{f, T^*}\ \text{  and  }\\
P_{\cN_{f,J}\otimes \cD_{f,T}}\Theta_{f,T}(\Lambda_1,\ldots,
\Lambda_n)|\cN_{f,J}\otimes \cD_{f, T^*}&=\Theta_{f,T,
J}(W_1,\ldots, W_n).
\end{split}
\end{equation*}
Hence, using relation   $
K_{f,T}K_{f,T}^*+\Theta_{f,T}\Theta_{f,T}^*= I_{\HH^2(f)\otimes
{\cD_{f,T}}}$ (see \cite{Po-multi}),   the fact that $\text{\rm
range}\, K_{f,T}\subseteq \cN_{f,J}\otimes \cD_{f,T}$  and
$W_i^*=\Lambda_i^*|_{\cN_{f,J}}$, \ $i=1,\ldots, n$,    we deduce
that

\begin{equation}\label{J-fa}
I_{\cN_{f,J}\otimes \cD_{f,T}}-\Theta_{f,T,
J}\Theta_{f,T,J}^*=K_{f,T,J}K_{f,T,J}^*,
\end{equation}
where $\Theta_{f,T,J}$ is  the  constrained characteristic function
of $f$, $T$ and $J$, and  $K_{f,T,J}$ is  the corresponding
constrained Poisson kernel.

Now, we introduce  the Hilbert space
$$
{\bf K}:=\left(\cN_{f,J}\otimes \cD_{f,T}\right)\oplus
\overline{\Delta_{\Theta_{f,T,J}}(\cN_{f,J}\otimes \cD_{f,T^*})}
$$
and  define the operator $\Phi: \cN_{f,J}\otimes \cD_{f, T^*}\to
{\bf K}$ by setting
\begin{equation}
\label{PhiTheta} \Phi x:=\Theta_{f,T,J} x\oplus
\Delta_{\Theta_{f,T,J}} x,\qquad x\in \cN_{f,J}\otimes \cD_{f,T^*}.
\end{equation}
It is easy to see that $\Phi$ is an isometry and
\begin{equation}\label{fi}
\Phi^*(y\oplus 0)=\Theta_{f,T,J}^*y, \qquad y\in \cN_{f,J}\otimes
\cD_{f,T}.
\end{equation}
Hence, letting  $P_{{\bf H}}$  be the orthogonal projection of ${\bf
K}$ onto the subspace ${\bf H}$, we have
\begin{equation*}
\|y\|^2= \|P_{{\bf H}}(y\oplus 0)\|^2+\|\Phi \Phi^*(y\oplus 0)\|
=\|P_{{\bf H}}(y\oplus 0)\|^2+\|\Theta_{f,T,J}^*y\|^2
\end{equation*}
for any $y\in \cN_{f,J}\otimes \cD_{f,T}$. Note  also that relation
\eqref{J-fa} implies
$$
\|K_{f,T,J}^* y\|^2+ \|\Theta_{f,T,J}^*y\|^2=\|y\|^2, \qquad y\in
\cN_{f,J}\otimes \cD_{f,T}.
$$
Consequently,  we deduce that
\begin{equation}\label{K*P}
\|K_{f,T,J} y\|=\|P_{{\bf H}}(y\oplus 0)\|,  \qquad y\in
\cN_{f,J}\otimes \cD_{f,T}.
\end{equation}
Now, we prove that
  $K_{f,T,J}$ is a one-to-one operator.   Indeed, due to Proposition \ref{KJT}, for
any $h\in \cH$, we have
$$
 \| K_{f,T,J}h\|^2=\|h\|^2- \lim_{q\to \infty}
\sum_{\alpha\in \FF_n,\, |\alpha|=q} \| [f (T)]_\alpha^*h\|^2.
$$
Consequently, if $K_{f,T,J} h=0$, then $\|h\|^2= \lim_{q\to \infty}
\sum_{\alpha\in \FF_n,\, |\alpha|=q} \| [f (T)]_\alpha^*h\|^2$.
Since $[f_1(T),\ldots, f_n(T)]$ is a row contraction, the latter
relation implies $\|h\|^2=  \sum_{\alpha\in \FF_n,\, |\alpha|=q} \|
[f (T)]_\alpha^*h\|^2$ for any $q\in \NN$. Since  $T\in
\BB_f^{cnc}(\cH)$, we deduce that $h=0$, which proves that
$K_{f,T,J}$ is a one-to-one operator and, therefore,  the range of
$K_{f,T,J}^*$ is dense in $\cH$.

Next, we show that
\begin{equation}
\label{H}
 {\bf H}=\left\{P_{{\bf H}}(y\oplus 0):\ y\in
\cN_{f,J}\otimes \cD_{f,T}\right\}^{-}.
\end{equation}
Let $x\in {\bf H}$ and assume that $x\perp P_{{\bf H}}(y\oplus 0)$
for any $y\in \cN_{f,J}\otimes \cD_{f,T}$. Using the definition of
${\bf H}$ and the fact that ${\bf K}$  is the closed span of  all
the vectors $y\oplus 0$  for $y\in \cN_{f,J}\otimes \cD_T$
 and $\Theta_{J,T} x\oplus \Delta_{\Theta_{f,T,J}}
x$ for $x\in \cN_{f,J}\otimes \cD_{f,T^*}$, we deduce that $x=0$.
Therefore, relation \eqref{H} holds.

Note that,  due to relation \eqref{K*P} and \eqref{H},  there is a
unique unitary operator $\Gamma:\cH\to {\bf H}$ such that
\begin{equation}\label{Ga}
\Gamma(K_{f,T,J}^* y)=P_{{\bf H}}(y\oplus 0), \qquad  y\in
\cN_{f,J}\otimes \cD_{f,T}.
\end{equation}
Using  relations \eqref{J-fa} and  \eqref{fi},  and the fact that
$\Phi$ is an isometry defined by \eqref{PhiTheta}, we have
\begin{equation*}
\begin{split}
P_{\cN_{f,J}\otimes \cD_{f,T}} \Gamma K_{f,T,J}^* y&=
P_{\cN_{f,J}\otimes \cD_{f,T}} P_{{\bf H}}(y\oplus 0) =
y-P_{\cN_{f,J}\otimes \cD_{f,T}} \Phi \Phi^*(y\oplus 0)\\
&=y-\Theta_{f,T,J} \Theta_{f,T,J}^* y =K_{f,T,J} K_{f,T,J}^*y
\end{split}
\end{equation*}
for any $y\in \cN_{f,J}\otimes \cD_{f,T}$. Hence,  and using the
fact that the range of $K_{f,T,J}^*$ is dense in $\cH$, we obtain
relation
\begin{equation}
\label{PGK} P_{\cN_{f,J}\otimes \cD_{f,T}} \Gamma=K_{f,T,J}.
\end{equation}
Now, we define ${\bf T}_i:{\bf H}\to {\bf H}$ be the transform of
$T_i$ under the unitary operator $\Gamma:\cH\to {\bf H}$ defined  by
\eqref{Ga}. More precisely, we set ${\bf T}_i:=\Gamma T_i\Gamma^*$,\
$i=1,\ldots, n$.   Since $K_{f,T,J}$ is one-to-one, relation
\eqref{PGK} implies that
\begin{equation*}
  P_{\cN_{f,J}\otimes \cD_{f,T}} |_{{\bf H}}= K_{f,T,J}
\Gamma^*
\end{equation*}
is a one-to-one operator acting from ${\bf H}$ to $\cN_{f,J}\otimes
\cD_{f,T}$. Due to relation \eqref{PGK} and Proposition \ref{KJT},
we obtain
\begin{equation*}
\begin{split}
\left(P_{\cN_{f,J}\otimes \cD_{f,T}} |_{{\bf H}}\right) {\bf
T}_i^*\Gamma h&= \left(P_{\cN_{f,J}\otimes \cD_{f,T}} |_{{\bf
H}}\right) \Gamma T_i^* h =K_{f,T,J} T_i^*h\\
&= \left( B_i^*\otimes I_{\cD_{f,T}}\right) K_{f,T,J}h = \left(
B_i^*\otimes I_{\cD_{f,T}}\right) \left(P_{\cN_{f,J}\otimes
\cD_{f,T}} |_{{\bf H}}\right)\Gamma h
\end{split}
\end{equation*}
for any $h\in \cH$. Consequently,

\begin{equation}
\label{def} \left( P_{\cN_{f,J}\otimes \cD_{f,T}}|_{{\bf H}}\right)
{\bf T}_i^*x= (B_i^*\otimes I_{\cD_{f,T}})\left( P_{\cN_{f,J}\otimes
\cD_{f,T}}|_{\bf H}\right)x, \qquad x\in {\bf H}.
\end{equation}
Due to the fact that  the operator $P_{\cN_{f,J}\otimes
\cD_{f,T}}|_{{\bf H}}$ is one-to-one, the relation \eqref{def}
uniquely determines the operators ${\bf T}_i^*$ for $i=1,\ldots, n$.

Now, we assume that $T:=(T_1,\ldots, T_n)\in \cV_{f,J}^{pure}(\cH)$.
 According to Proposition \ref{KJT}, the constrained Poisson
 kernel $K_{f,T,J}:\cH\to \cN_{f,J}\otimes \cD_{f,T}$
 is an isometry and, therefore,
  $K_{f,T,J}K_{f,T,J}^*$ is the orthogonal projection of
$\cN_{f,J}\otimes \cD_{f,T}$ onto $K_{f,T,J}\cH$. Relation
\eqref{J-fa} implies that $K_{f,T,J}K_{f,T,J}^*$ and
$\Theta_{f,T,J}\Theta_{f,T,J}^*$ are mutually orthogonal projections
such that
$$
K_{f,T,J}K_{f,T,J}^*+\Theta_{f,T,J}\Theta_{f,T,J}^*=
I_{\cN_{f,J}\otimes \cD_{f,T}}.
$$
This shows that  $\Theta_{f,T,J}$ is a partial isometry and
$\Theta_{f,T,J}^* \Theta_{f,T,J}$ is a projection. Consequently,
$\Delta_{\Theta_{f,T,J}}$ is the projection  on the orthogonal
complement of  the range of $ \Theta_{f,T,J}^*$.

We remark that a vector $u\oplus \Delta_{\Theta_{f,T,J}}v\in {\bf
K}$ is in ${\bf H}$ if and only if
$$
\left< u\oplus \Delta_{\Theta_{f,T,J}}v,\Theta_{f,T,J}x\oplus
\Delta_{\Theta_{f,T,J}}x\right>=0\quad \text{ for any } \ x\in
\cN_{f,J}\otimes \cD_{f,T^*},
$$
which is equivalent to
\begin{equation}
\label{TH*} \Theta_{f,T,J}^*u+\Delta^2_{\Theta_{f,T,J}}v=0.
\end{equation}
Since  $\Theta_{f,T,J}^*u\perp \Delta^2_{\Theta_{f,T,J}}v$, relation
\eqref{TH*} holds if and only if $\Theta_{f,T,J}^*u=0$ and
$\Delta_{\Theta_{f,T,J}}v=0$. Consequently, we have
$${\bf H}=\left(\cN_{f,J}\otimes \cD_{f,T}\right)\ominus
\Theta_{f,T,J}(\cN_{f,J}\otimes \cD_{f,T^*}).
$$
Note that $P_{\cN_{f,J}\otimes \cD_{f,T}}|_{{\bf H}}$ is the
restriction operator and relation \eqref{def} implies $ {\bf T}_i=
P_ {{\bf H}} (B_i\otimes I_{\cD_{f,T}})|_{\bf H}$ for  $i=1,\ldots,
n$.

 Conversely,  if $\Theta_{f,T,J}$ is   is a partial
isometry, relation \eqref{J-fa} implies that $K_{f,T,J}$ is a
partial isometry. On the other hand, since $T$ is c.n.c.,
Proposition \ref{KJT} implies
$$
\text{\rm SOT-}\,\lim_{k\to\infty}\sum_{|\alpha|=k} [f(T)]_\alpha
[f(T)]_\alpha^*=0,
$$
which proves that $T\in \cV_{f,J}^{pure}(\cH)$.
    This completes the proof.
\end{proof}

 We remark that,
if $T:=(T_1,\ldots, T_n)\in \cV_{f,J}^{cnc}(\cH)$, then
$\Theta_{f,T,J}$   has dense range if and only if there is no
element $h\in \cH$, $h\neq 0$, such that
\begin{equation*}
 \lim_{k\to\infty}\sum_{|\alpha|=k} [f(T)]_\alpha
 [f(T)]_\alpha^*h=0.
\end{equation*}
Indeed, due to Proposition \ref{KJT}, the condition above is
equivalent to $\ker \left( I-K_{f,T,J}^* K_{f,T,J}\right)=\{0\}$.
Using relation \eqref{J-fa} , we deduce that the latter equality is
equivalent to
$$
\ker \Theta_{f,T,J} \Theta_{f,T,J}^*=\ker \left( I-K_{f,T,J}
K_{f,T,J}^*\right)=\{0\},
$$
which implies that  $\Theta_{f,T,J}$ has dense range.

\begin{proposition}  Let $f=(f_1,\ldots, f_n)$ be an  $n$-tuple of formal power
series  with the model property and  let $J\neq H^\infty(\BB_f)$ be
a WOT-closed two-sided  ideal of $H^\infty(\BB_f)$ such that $1\in
\cN_{f,J}$. Then  $T:=(T_1,\ldots, T_n)\in \cV_{f,J}^{cnc}(\cH)$ is
unitarily equivalent to  the universal $n$-tuple  $(B_1\otimes
I_\cK,\ldots, B_n\otimes I_\cK)$ for some Hilbert space $\cK$ if and
only if  \ $\Theta_{f,T,J}=0$.
\end{proposition}
\begin{proof} First, we assume that
  $T=(B_1\otimes I_\cK,\ldots, B_n\otimes I_\cK)$ and prove that
$K_{f,T,J}F=F$ for $F\in \cN_{f,J}\otimes \cK$. Since $1\in
\cN_{f,J}$, a straightforward calculation  shows that
$\Delta_{f,T}=P_{  \cK}|_{\cN_{f,J}\otimes \cK}$ as an operator
acting on $\cN_{f,J}\otimes \cK$. Indeed, note that
\begin{equation*}
\begin{split}
\Delta^2_{f,T}&=I_{\cN_{f,J}\otimes \cK}-\sum_{i=1}^n
f_i(B)f_i(B)^*\otimes I_\cK\\
&=P_{\cN_{f,J}\otimes \cK}\left(I_{\HH^2(f)\otimes \cK}-\sum_{i=1}^n
f_i(M_Z)f_i(M_Z)^*\otimes I_\cK\right)|_{\cN_{f,J}\otimes \cK}\\
&=P_{\cN_{f,J}\otimes \cK}\left(I_{\HH^2(f)\otimes \cK}-\sum_{i=1}^n
M_{f_i}M_{f_i} ^*\otimes I_\cK\right)|_{\cN_{f,J}\otimes \cK}\\
&=P_\cK|_{\cN_{f,J}\otimes \cK}.
\end{split}
\end{equation*}
Here we used the natural identification of $1\otimes \cK$ with
$\cK$. Since $\cD_{f,T}=\cK$, using the definition of the
constrained Poisson kernel $K_{f,T,J}$,   for any
$F=\sum\limits_{\beta\in \FF_n^+} f_\beta\otimes k_\beta$ in
$\cN_{f,J}\otimes \cK\subseteq \HH^2(f)\otimes \cK$, we have

\begin{equation*}
\begin{split}
K_{f,T,J}F&=\sum\limits_{\alpha\in \FF_n^+}
P_{\cN_{f,J}}f_\alpha\otimes P_{\cK}([f(B)]_\alpha^*\otimes I_\cK)F=
 \sum\limits_{\alpha\in \FF_n^+} P_{\cN_{f,J}}f_\alpha\otimes
 P_{ \cK}([f(M_Z)]_\alpha^*\otimes I_\cK)F\\
 &=
 \sum\limits_{\alpha\in \FF_n^+} P_{\cN_{f,J}}f_\alpha\otimes
 P_{ \cK}(M_{f_\alpha}^*\otimes I_\cK)F
= \sum\limits_{\alpha\in \FF_n^+} P_{\cN_J}f_\alpha\otimes
k_\alpha\\
&=P_{\cN_{f,J}\otimes \cK} F=F.
 \end{split}
 \end{equation*}
Due to relation \eqref{J-fa} we have  $\Theta_{f,T,J}=0$.
Conversely, if $\Theta_{f,T,J}=0$, then Theorem \ref{model} shows
that $T$ is unitarily equivalent to   $(B_1\otimes
I_{\cD_{f,T}},\ldots, B_n\otimes I_{\cD_{f,T}})$. This completes the
proof.
\end{proof}

Let  $\Phi: \cN_{f,J}\otimes \cK_1\to \cN_{f,J}\otimes \cK_2$ and
$\Phi':\cN_{f,J}\otimes\cK_1'\to \cN_{f,J}\otimes  \cK_2'$ be two
multi-analytic operators with respect to $B_1,\ldots, B_n$, i.e.,
  $\Phi(B_i\otimes
I_{\cK_1})=(B_i\otimes I_{\cK_2})\Phi$  and $\Phi'(B_i\otimes
I_{\cK_1'})=(B_i\otimes I_{\cK_2'})\Phi'$ for any $i=1,\ldots,n$. We
say that $\Phi$ and $\Phi'$ coincide if there are two unitary
operators $\tau_j\in B(\cK_j, \cK_j')$, $j=1,2$,  such that
$$
\Phi'(I_{\cN_{f,J}}\otimes \tau_1)=(I_{\cN_{f,J}}\otimes \tau_2)
\Phi.
$$

The next result shows that the constrained characteristic function
is a complete unitary invariant for the   $n$-tuples of operators in
the noncommutative variety $\cV_{f,J}^{cnc}(\cH)$.

\begin{theorem}\label{u-inv2} Let $f=(f_1,\ldots, f_n)$ be an  $n$-tuple of formal power
series  with the model property and  let $J\neq H^\infty(\BB_f)$ be
a WOT-closed two-sided  ideal of $H^\infty(\BB_f)$.
 If  $T:=(T_1,\ldots, T_n)\in \cV_{f,J}^{cnc}(\cH) $  and
 $T':=(T_1',\ldots, T_n')\in  \cV_{f,J}^{cnc}(\cH')$,  then $T$ and $T'$ are
unitarily equivalent if and only if their  constrained
characteristic functions $\Theta_{f,T,J}$  and $\Theta_{f,T',J}$
coincide.
\end{theorem}
\begin{proof}  Let $W:\cH\to
\cH'$ be a unitary operator such that $T_i=W^*T_i'W$ for any
$i=1,\ldots, n$.  Since $T\in \cC_f^{SOT}(\cH)$ or $T\in
\cC_f^{rad}(\cH)$ and similar relations hold  for $T'$, it is easy
to see that
$$
W\Delta_{f,T}=\Delta_{f,T'}W \quad \text{ and }\quad (\oplus_{i=1}^n
W)\Delta_{f,T^*}=\Delta_{f,T'^*}(\oplus_{i=1}^n W).
$$
We introduce  the unitary operators $\tau$ and $\tau'$ by setting
$$\tau:=W|_{\cD_{f,T}}:\cD_{f,T}\to \cD_{f,T'} \quad \text{ and }\quad
\tau':=(\oplus_{i=1}^n W)|_{\cD_{f,T^*}}:\cD_{f,T^*}\to
\cD_{f,T'^*}.
$$
It is easy to see that
$$
(I_{\cN_{f,J}}\otimes \tau)\Theta_{f,T,J}=\Theta_{f,T',J}(I_{\cN_{f,J}}\otimes
\tau').
$$

Conversely, assume that the constrained characteristic functions  of
$T$ and $T'$ coincide, i.e.,  there exist unitary operators
$\tau:\cD_{f,T}\to \cD_{f,T'}$ and $\tau_*:\cD_{f,T^*}\to
\cD_{f,{T'}^*}$ such that
\begin{equation*}
(I_{\cN_{f,J}}\otimes \tau)\Theta_{f,T,J}=\Theta_{f,T',J}(I_{\cN_{f,J}}\otimes
\tau_*).
\end{equation*}
As consequences, we obtain
$$
\Delta_{\Theta_{J,T}}=\left(I_{\cN_{f,J}}\otimes \tau_*\right)^*
\Delta_{\Theta_{f,T',J}}\left(I_{\cN_{f,J}}\otimes \tau_*\right)
$$
and
$$
\left(I_{\cN_{f,J}}\otimes
\tau_*\right)\overline{\Delta_{\Theta_{J,T}}(\cN_{f,J}\otimes \cD_{f,T^*})}=
\overline{\Delta_{\Theta_{J,T'}}(\cN_{f,J}\otimes \cD_{f,{T'}^*})}.
$$
 Now, we define the unitary operator $U:{\bf K}\to {\bf K}'$
by setting
$$U:=(I_{\cN_{f,J}}\otimes \tau)\oplus (I_{\cN_{f,J}}\otimes
\tau_*),
$$
where the Hilbert spaces ${\bf K} $ and ${\bf K}'$  were defined in
Theorem \ref{model}. We remark that the operator
$\Phi:\cN_{f,J}\otimes \cD_{f,T^*}\to {\bf K}$, defined in the proof
of the same theorem and the corresponding $\Phi'$ satisfy the
relations
\begin{equation}
\label{Uni1} U \Phi\left(I_{\cN_{f,J}}\otimes
\tau_*\right)^*=\Phi'\quad \text{ and } \quad
  \left(I_{\cN_{f,J}}\otimes \tau\right) P_{\cN_{f,J}\otimes
\cD_{f,T}}^{\bf K} U^*=P_{\cN_{f,J}\otimes \cD_{f,T'}}^{\bf K'},
\end{equation}
where $P_{\cN_{f,J}\otimes \cD_{f,T}}^{\bf K}$ is the orthogonal
projection of ${\bf K}$ onto $\cN_{f,J}\otimes \cD_{f,T}$. Note that
relation \eqref{Uni1} implies
\begin{equation*}
\begin{split}
U{\bf H}&=U {\bf K}\ominus U\Phi(\cN_{f,J}\otimes \cD_{f,T^*})\\
&={\bf K}'\ominus \Phi'(I_{\cN_{f,J}}\otimes \tau_*)(\cN_{f,J}\otimes \cD_{f,T^*})\\
&={\bf K}'\ominus \Phi' (\cN_{f,J}\otimes \cD_{f,{T'}^*})={\bf H}'.
\end{split}
\end{equation*}
Consequently, the operator $U|_{\bf H}:{\bf H}\to {\bf H}'$ is
unitary. Note also that
\begin{equation}
\label{intertw2} (B_i^*\otimes I_{\cD_{f,T'}})(I_{\cN_{f,J}}\otimes
\tau)= (I_{\cN_{f,J}}\otimes \tau)(B_i^*\otimes I_{\cD_{f,T}}).
\end{equation}
Let ${\bf T}:=({\bf T}_1,\ldots {\bf T}_n)$ and ${\bf T}':=({\bf
T}_1',\ldots {\bf T}_n')$ be the models provided by Theorem
\ref{model}  for the $n$-tuples $T$ and $T'$, respectively. Using
the relation \eqref{def}  for $T'$ and $T$,  and  relations
\eqref{Uni1}, \eqref{intertw2},  we obtain
\begin{equation*}
\begin{split}
P_{\cN_{f,J}\otimes \cD_{f,T'}}^{\bf K'}{{\bf T}_i'}^*Ux
&=  (B_i^*\otimes I_{\cD_{T'}}) P_{\cN_{f,J}\otimes \cD_{f,T}}^{\bf K}Ux\\
&=(B_i^*\otimes I_{\cD_{f,T'}})(I_{\cN_{f,J}}\otimes \tau) P_{\cN_{f,J}\otimes \cD_{f,T}}^{\bf K}x\\
&=(I_{\cN_{f,J}}\otimes \tau)(B_i^*\otimes I_{\cD_{f,T}})P_{\cN_{f,J}\otimes \cD_{f,T}}^{\bf K}x\\
&=(I_{\cN_{f,J}}\otimes \tau) P_{\cN_{f,J}\otimes \cD_{f,T}}^{\bf K}
{\bf T}_i^*x\\
&= P_{\cN_{f,J}\otimes \cD_{f,T'}}^{\bf K'}U {\bf T}_i^*x
\end{split}
\end{equation*}
for any $x\in {\bf H}$ and $i=1,\ldots, n$. Since
$P_{\cN_{f,J}\otimes \cD_{f,T'}}^{\bf K'}$ is a one-to-one operator
(see the proof of  Theorem \ref{model}), we deduce that
$$
\left(U|_{\bf H}\right) {\bf T}_i^*={{\bf T}_i'}^*\left(U|_{\bf
H}\right),\qquad i=1,\ldots,n.
$$
 According to Theorem \ref{model}, the $n$-tuples  $T$ and $T'$ are unitarily equivalent.
  The proof is complete.
\end{proof}

\bigskip

\section{ Dilation theory on noncommutative varieties }

In this section, we develop a dilation theory for $n$-tuples of
operators  in the noncommutative variety $$ \{(T_1,\ldots,T_n)\in
\BB_f(\cH): \   (q\circ f)(T_1,\ldots,T_n)=0,\  q\in \cP\},
$$
 where $\cP$ is a set of
homogeneous noncommutative polynomials.

Let $f=(f_1,\ldots, f_n)$ be an $n$-tuple of formal power  series
with the model property   and let $J$ be a WOT-closed two-sided
ideal of $H^\infty(\BB_f)$. We recall that the universal model
 $B=(B_1,\ldots, B_n) $  for the noncommutative variety
 $\cV_{f,J}^{cnc}$ is defined by
$B_i:=P_{\cN_{f,J}} M_{Z_i} |_{\cN_{f,J}}$,  for  $i=1,\ldots, n$,
  where
 $$
 \cN_{f,J}:= \HH^2(f)\ominus \cM_{f,J}\quad \text{ and }\quad
 \cM_{f,J}:=\overline{ J \HH^2(f)}.
 $$

\begin{theorem}\label{compact2}  Let $f=(f_1,\ldots, f_n)$ be  an  $n$-tuple of formal power
series with the model property and let $J$ be a WOT-closed two-sided
ideal of $H^\infty(\BB_f)$ such that $1\in\cN_{f,J}$. Then the
$C^*$-algebra  $C^*(B_1,\ldots, B_n)$ is irreducible.

If $f\in \cM^{||}$, then all the compact operators in $B(\cN_{f,J})$
are contained in the operator space
$$\overline{\text{\rm span}}\{B_\alpha B_\beta^*:\ \alpha,\beta\in \FF_n^+\}.
$$

\end{theorem}
\begin{proof}
To prove the first part of the theorem, let $\cM\subseteq \cN_{f,J}$
be a nonzero subspace which is jointly reducing for $B_1,\ldots,
B_n$, and let $y=\sum_{\alpha\in \FF_n^+} a_\alpha f_\alpha$ be a
nonzero power series in  $\cM$. Then there is $\beta\in \FF_n^+$
such that $a_\beta\neq 0$. Since $f=(f_1,\ldots, f_n)$ is  an
$n$-tuple of formal power series
 with the model
 property, we have $M_{f_i}=f_i(M_{Z})$,
where  $M_Z:=(M_{Z_1},\ldots, M_{Z_n})$ is either in the convergence
set $\cC_f^{SOT}(\HH^2(f))$ or $\cC_f^{rad}(\HH^2(f))$.
Consequently,  since $1\in \cN_{f,J}$, we obtain
$$
a_\beta=P_\CC^{\cN_{f,J}}
[f(B)]_\beta^*y=\left(I_{\cN_{f,J}}-\sum_{i=1}^n
f_i(B)f_i(B)^*\right)[f(B)]_\beta^* y,
$$
where $B=(B_1,\ldots,B_n)$. Taking into account that $\cM$ is
reducing for $B_1,\ldots, B_n$ and $a_\beta\neq 0$, we deduce that
$1\in \cM$. Using again that $\cM$ is invariant under  $B_1,\ldots,
B_n$, we obtain $P_{\cN_{f,J}}\CC[Z_1,\ldots, Z_n]\subseteq \cM$.
Since $\CC[Z_1,\ldots, Z_n]$ is dense in $\HH^2(f)$, we conclude
that $\cM=\cN_{f,J}$, which shows that $C^*(B_1,\ldots,
 B_n)$  is irreducible.

Now, we prove the second part of the theorem. Since   $\cN_{f,J}$ is
an invariant subspace under each operator $M_{Z_i}^*$, \
$i=1,\ldots,n$, and $(M_{Z_1},\ldots,
 M_{Z_n})$
 is in  the set  of  norm-convergence (or radial norm-convergence)  for the $n$-tuple
 $f$, the operator
$f_i(B)$ is in
 $ \overline{\text{\rm span}} \{B_\alpha B_\beta^*:\
\alpha,\beta\in \FF_n^+\}$.  Taking into account that
$f=(f_1,\ldots, f_n)$ is   an  $n$-tuple of formal power series
 with the model
 property,  we have $M_{f_i}=f_i(M_{Z_1},\ldots, M_{Z_n})$. On the other hand, since $1\in \cN_{f,J}$ the
orthogonal projection of $\cN_{f,J}$ onto the constant power series
satisfies the equation
$$P_\CC^{\cN_{f,J}}=I_{\cN_{f,J}}-\sum_{i=1}^n f_i(B)f_i(B)^*.
 $$
 Therefore,  $P_\CC^{\cN_{f,J}}$ is also in  $ \overline{\text{\rm span}} \{B_\alpha B_\beta^*:\
\alpha,\beta\in \FF_n^+\}$. Let
$q(B):=\sum_{|\alpha|\leq m}a_\alpha [f(B)]_\alpha$ and let
 $\xi:=\sum_{\beta\in \FF_n^+} b_\beta f_\beta\in {\cN_{f,J}}$. Note
 \begin{equation*}
 \begin{split}
 P_\CC^{\cN_{f,J}} q(B)^*\xi&=P_\CC \sum_{|\alpha|\leq m}\overline{a}_\alpha M_{f_\alpha}^*\xi
 =\sum_{|\alpha|\leq m} \overline{a}_\alpha b_\alpha\\
 &=\left<\xi, \sum_{|\alpha|\leq m} a_\alpha f_\alpha\right>
 =\left<\xi, q(B)1\right>.
 \end{split}
 \end{equation*}
Consequently, if $r(B):=\sum_{|\gamma|\leq p}c_\gamma [f(B)]_\gamma$, then
\begin{equation}
\label{rP2} r(B)P_\CC q(B)^*\xi=\left<\xi,q(B)1\right>r(B)1,
\end{equation}
which shows that  $r(B)P_\CC^{\cN_{f,J}} q(B)^*$ is a rank one
operator acting on  $\cN_{f,J}$. Since the set of all vectors of the form $
\sum_{|\alpha|\leq m} a_\alpha[f(B)]_\alpha 1$, where $\ m\in \NN$,
$a_\alpha\in \CC$, is dense in $\cN_{f,J}$, and using relation
\eqref{rP2}, we deduce that all compact operators in $B(\cN_{f,J})$
are in $
 \overline{\text{\rm span}} \{B_\alpha B_\beta^*:\ \alpha,\beta\in
 \FF_n^+\}.
 $
This completes the proof.
\end{proof}

\begin{proposition}\label{eq-mult}  Under the hypotheses of
Theorem \ref{compact2}, if $\cH$, $\cK$ are Hilbert spaces, then the
$n$-tuples $(B_1\otimes I_\cH,\ldots, B_n\otimes I_\cH)$
 and $(B_1\otimes I_\cK,\ldots, B_n\otimes I_\cK)$are unitarily equivalent if and only if
  their multiplicities are equal, i.e., $\dim \cH=\dim \cK$.
\end{proposition}
\begin{proof} Let   $U:\cN_{f,J}\otimes \cH\to \cN_{f,J}\otimes \cK$ be a unitary  operator such that
$
U(B_i\otimes I_\cH)=(B_i\otimes I_{\cK})U$ for $ i=1,\ldots, n.
$
Since $U$ is unitary , we deduce that
$
U(B_i^*\otimes I_\cH)=(B_i^*\otimes I_{\cK})U$, $i=1,\ldots, n.
$
Since, according to Theorem \ref{compact}, the $C^*$-algebra $C^*(B_1,\ldots, B_n)$ is
 irreducible, we infer  that
$U=I_{\cN_{f,J}}\otimes W$ for some unitary operator $W\in
B(\cH,\cK)$. Therefore, $\dim \cH=\dim \cK$. The converse is clear.
 \end{proof}

\begin{theorem}\label{ccmap} Let $f=(f_1,\ldots, f_n)$ be an $n$-tuple of formal power
series  with the radial  approximation property,
  let $\cP\subset \CC[Z_1,\ldots,Z_n]$ be a set of homogeneous
  polynomials, and let $B=(B_1,\ldots, B_n)$ be the universal model
  associated with $f$ and     the WOT-closed two-sided ideal  $J_{\cP\circ f}$
   generated by $q(f(M_Z))$, $q\in\cP$, in the Hardy algebra
$H^\infty(\BB_f)$.
  If  the $n$-tuple  $T=(T_1,\ldots, T_n)\in \BB_f(\cH)$ has the property that
$$
(q\circ f)(T)=0,\qquad q\in\cP,
$$
then the linear map $\Psi_{f,T,\cP}:\overline{\text{\rm
span}}\{{B}_\alpha { B}_\beta:\ \alpha,\beta\in \FF_n^+\}\to B(\cH)
$ defined by
$$\Psi_{f,T,\cP}({ B}_\alpha { B}_\beta):=T_\alpha T_\beta^*,\qquad \alpha,\beta\in \FF_n^+,
$$
is completely contractive.
\end{theorem}
\begin{proof}
Let $g=(g_1,\ldots,g_n)$ be the inverse of $f=(f_1,\ldots, f_n)$
with respect to the composition, and assume that
$g_i:=\sum_{\alpha\in\FF_n^+} a_\alpha^{(i)} Z_\alpha$,
$i=1,\ldots,n$. Since $f$ has the model property, the left
multiplication $M_{Z_i}:\HH^2(f)\to \HH^2(f)$  defined by
   $$
   M_{Z_i}\psi:=Z_i\psi, \qquad \psi\in \HH^2(f),
  $$
  is a bounded left multiplier of $\HH^2(f)$ and
$M_{Z_i}=U^{-1} \varphi_i(S_1,\ldots,S_n) U$, where
   $\varphi_i(S_1,\ldots,S_n)$ is   in  the noncommutative Hardy algebra $F_n^\infty$ and has the
    Fourier
    representation  $\sum_{\alpha\in \FF_n^+} a_\alpha^{(i)}
    S_\alpha$,
    and
     $U:\HH^2(f)\to F^2(H_n)$ is the unitary operator defined by
$U(f_\alpha):=e_\alpha$, $\alpha\in \FF_n^+$.

Since $f$ has the radial approximation property, there is $\delta\in
(0,1)$ such that $rf:=(rf_1,\ldots, rf_n)$ has the model property
for any $r\in (\delta, 1]$.  We remark the the Hilbert space
$\HH^2(rf)$ is in fact $\HH^2(f)$ with the inner product defined by
$\left<f_\alpha,
f_\beta\right>_{\HH^2(rf)}:=\frac{1}{r^{|\alpha|+|\beta|}}
\delta_{\alpha\beta}$, $\alpha,\beta\in \FF_n^+$. Denote
$(g_i)_{1/r}:=\sum_{\alpha \in \FF_n^+} a_\alpha^{(i)}
\frac{1}{r^{|\alpha|}} Z_\alpha$ and note that $Z_i=(g_i)_{1/r}\circ
(rf)$ for  $i=1,\ldots,n$. Since $rf$ has the model  property for $r\in
(\delta, 1]$,  we deduce that the multiplication
$M_{Z_i}^{(r)}:\HH^2(rf)\to \HH^2(rf)$  defined by
   $$
   M_{Z_i}^{(r)}\psi:=Z_i\psi, \qquad \psi\in \HH^2(rf),
  $$
  is a bounded left multiplier of $\HH^2(rf)$ and
$M_{Z_i}^{(r)}=(U^{(r)})^{-1}
\varphi_i(\frac{1}{r}S_1,\ldots,\frac{1}{r}S_n) U^{(r)}$, where
   $\varphi_i(\frac{1}{r}S_1,\ldots, \frac{1}{r}S_n)$ is
    in  the noncommutative Hardy algebra $F_n^\infty$ and has the
    Fourier
    representation  $\sum_{\alpha\in \FF_n^+} \frac{1}{r^{|\alpha|}}a_\alpha^{(i)}
    S_\alpha$,
    and
     $U^{(r)}:\HH^2(rf)\to F^2(H_n)$ is the unitary operator defined by
$U^{(r)}(f_\alpha):=\frac{1}{r^{|\alpha|}}e_\alpha$, $\alpha\in
\FF_n^+$.

For each $r\in (\delta, 1]$, let $J_{\cP\circ rf}$ be the WOT-closed
two-sided ideal of $H^\infty(\BB_{rf})$ generated by the operators
$q(rf(M_Z^{(r)}))$, $q\in \cP$. We introduce the subspace
$\cN_{rf,J_{\cP\circ rf}}:=\HH^2(rf)\ominus \cM_{rf,J_{\cP\circ
rf}}$, where $\cM_{rf,J_{\cP\circ rf}}=\overline{J_{\cP\circ rf}
\HH^2(rf)}$, and the operators  $B_i^{(r)}:=P_{\cN_{rf,J_{\cP\circ
rf}}} M_{Z_i}^{(r)}|_{\cN_{rf,J_{\cP\circ rf}}}$, $i=1,\ldots,n$. We
denote by $J_\cP$ the  WOT-closed two-sided ideal of $F_n^\infty$
generated by the operators $q(S_1,\ldots,S_n)$, $q\in \cP$. We also
introduce the subspace $\cN_{J_\cP}=F^2(H_n)\ominus \cM_{J_\cP}$,
where $\cM_{J_\cP}:=\overline{J F^2(H_n)}$.

Our next step is to show that $\psi=\sum_{\alpha\in \FF_n^+}
c_\alpha e_\alpha$ is in $\cN_{J_\cP}$ if and only $(U^{(r)})^{-1}
\psi=\sum_{\alpha\in \FF_n^+} c_\alpha r^{|\alpha|} f_\alpha$ is in
$\cN_{rf,J_{\cP\circ rf}}$. First, note that
$$
J_{\cP\circ rf}=J_\cP\circ rf:=\{\chi(rf(M_Z)):\ \chi\in J_\cP\}.
$$
Due to the definition of $\cN_\cP$, one can see that $\varphi\in
\cN_\cP$ if and only if $\left< \psi,
\chi(S_1,\ldots,S_n)1\right>_{F^2(H_n)}=0$ for any
$\chi(S_1,\ldots,S_n)=\sum_{\alpha\in \FF_n^+} a_\alpha S_\alpha\in
J_\cP$, which is equivalent to $\sum_{\alpha\in \FF_n^+}c_\alpha
\bar{a}_\alpha=0$. On the other hand, note that $\sum_{\alpha\in
\FF_n^+}c_\alpha r^{|\alpha|} f_\alpha$ is in $\cN_{rf,J_{\cP\circ
rf}}$ if and only if
$$
\left<\sum_{\alpha\in \FF_n^+}c_\alpha r^{|\alpha|} f_\alpha,
\chi(rf(M_Z))1\right>_{\HH^2(rf)}=0,\qquad  \chi\in J_\cP.
$$
Since $f$ has the model property, $M_{f_i}=f_i(M_Z)$ and,
consequently,   the relation above  is equivalent to
 $
\sum_{\alpha\in \FF_n^+}c_\alpha \bar{a}_\alpha=0
 $
 for any $\sum_{\alpha\in \FF_n^+} a_\alpha S_\alpha\in
J_\cP$,
 which proves our assertion.
 Now, it is easy to see that
 $$
 U^{(r)}(\cM_{rf,J_{\cP\circ rf}})=\cM_{J_\cP}\quad \text{\rm  and
 }\quad U^{(r)}(\cN_{rf,J_{\cP\circ rf}})=\cN_{J_\cP}.
 $$
Since $M_{Z_i}^{(r)}=(U^{(r)})^{-1}
\varphi_i(\frac{1}{r}S_1,\ldots,\frac{1}{r}S_n) U^{(r)}$,
$i=1,\ldots,n$, we deduce that
\begin{equation*}
\begin{split}
B_i^{(r)}&:=P_{\cN_{rf,J_{\cP\circ rf}}}
M_{Z_i}^{(r)}|_{\cN_{rf,J_{\cP\circ rf}}}\\
&= P_{\cN_{rf,J_{\cP\circ rf}}} (U^{(r)})^{-1} \left(
P_{\cN_{J_\cP}^\perp}+P_{\cN_{J_\cP}}\right)
\varphi_i\left(\frac{1}{r}S_1,\ldots,\frac{1}{r}S_n\right)
|_{\cN_{J_\cP}} \left(U^{(r)}|_{\cN_{rf,J_{\cP\circ rf}}}\right)\\
&=   (U^{(r)})^{-1} P_{\cN_{J_\cP}}
\varphi_i\left(\frac{1}{r}S_1,\ldots,\frac{1}{r}S_n\right)
|_{\cN_{J_\cP}} \left(U^{(r)}|_{\cN_{rf,J_{\cP\circ rf}}}\right)\\
&=   \left(U^{(r)}|_{\cN_{rf,J_{\cP\circ rf}}}\right)^{-1}
P_{\cN_{J_\cP}}
\varphi_i\left(\frac{1}{r}S_1,\ldots,\frac{1}{r}S_n\right)
|_{\cN_{J_\cP}} \left(U^{(r)}|_{\cN_{rf,J_{\cP\circ rf}}}\right),
\end{split}
\end{equation*}
where $U^{(r)}|_{\cN_{rf,J_{\cP\circ rf}}}: \cN_{rf,J_{\cP\circ
rf}}\to \cN_{J_\cP}$ is a unitary operator for each $r\in (\delta,
1]$.

Now, we assume that $T=(T_1,\ldots, T_n)\in \BB_f(\cH)$ has the
property that $ (q_j\circ f)(T)=0$ for  $j=1,\ldots,d$, and $0<r<1$.
Since the $H^\infty(\BB_f)$ functional calculus for the $n$-tuples
of operators in $\BB_f^{cnc}(\cH)$ is a homomorphism and $q\in\cP$ is a
homogenous
  polynomials, we have
  \begin{equation}
  \label{omo}
  (\varphi q(S_1,\ldots, S_n)) (rf_1(T),\ldots, rf_n(T))=r^{\text{deg}\,(q)}
   \varphi(rf_1(T),\ldots, rf_n(T))q(f_1(T),\ldots, f_n(T))=0
  \end{equation}
 for any $\varphi\in F_n^\infty$ and $q\in\cP$, where  deg\,$(q)$ denotes
 the degree
 of the polynomial $q$.
 On the other hand,   $J_{\cP\circ rf}$ is the WOT-closed
two-sided ideal of $H^\infty(\BB_{rf})$ generated by the operators
$q(rf(M_Z))$, $q\in \cP$, for each $r\in (\delta, 1]$. Since  the
$H^\infty(\BB_{rf})$-functional calculus
      for pure $n$-tuples  in $\BB_{rf}(\cH)$ is WOT-continuous, and
      $(T_1,\ldots, T_n)\in \BB_{rf}^{pure}(\cH)$,  relation
      \eqref{omo} implies $\psi(T_1,\ldots, T_n)=0$ for any
       $\psi\in J_{\cP\circ rf}$ and $r\in (\delta, 1)$. Therefore,
       $(T_1,\ldots,T_n)$ is in the noncommutative variety $\cV^{pure}_{rf,J_{\cP\circ
       rf}}(\cH)$.
       Applying Corollary \ref{vN-variety} to the $n$-tuple
       $(T_1,\ldots,T_n)$, we  deduce that there is  completely contractive linear map
       $\Phi:\overline{\text{\rm span}}\left\{
       B_\alpha^{(r)}{B_\beta^{(r)}}^*:\ \alpha,\beta\in
       \FF_n^+\right\} \to B(\cH)$
       uniquely defined by
       $\Phi\left(B_\alpha^{(r)}{B_\beta^{(r)}}^*\right):=T_\alpha
       T_\beta^*$ for all  $\alpha,\beta\in \FF_n^+$. Hence, and
       using
the fact that the $n$-tuple  $(B_1^{(r)}, \ldots,B_n^{(r)}) $ is
unitarily equivalent to $$\left(P_{\cN_{J_\cP}}
\varphi_1\left(\frac{1}{r}S_1,\ldots,\frac{1}{r}S_n\right)
|_{\cN_{J_\cP}}, \ldots, P_{\cN_{J_\cP}}
\varphi_n\left(\frac{1}{r}S_1,\ldots,\frac{1}{r}S_n\right)
|_{\cN_{J_\cP}}\right), $$  and $\cN_{J_\cP}$ is invariant under
$S_1^*,\ldots, S_n^*$, we obtain

\begin{equation*}
\begin{split}
\left\| \left[\sum_{|\alpha|,|\beta|\leq m} a_{\alpha
\beta}^{(ij)}T_\alpha T_\beta^*\right]_{k\times k} \right\| &\leq
\left\| \left[\sum_{|\alpha|,|\beta|\leq m}
a_{\alpha\beta}^{(ij)}B^{(r)}_\alpha
{B^{(r)}_\beta}^*\right]_{k\times k} \right\|\\
&\leq \left\| \left[\sum_{|\alpha|,|\beta|\leq m} a_{\alpha
\beta}^{(ij)}P_{\cN_{J_\cP}}
\varphi_\alpha\left(\frac{1}{r}S_1,\ldots,\frac{1}{r}S_n\right)
\varphi_\beta\left(\frac{1}{r}S_1,\ldots,\frac{1}{r}S_n\right)^*
|_{\cN_{J_\cP}}\right]_{k\times k} \right\|
\end{split}
\end{equation*}
for any $m,k\in \NN$, $a^{(ij)}_{\alpha\beta}\in \CC$, and $i,j\in
\{1,\ldots,k\}$. Since $f=(f_1,\ldots, f_n)$ has the radial
property,  $\varphi_1,\ldots, \varphi_n$ are free holomorphic
functions on a ball $[B(\cH)^n]_\gamma$ with $\gamma>1$.
Consequently, the map $(\delta, 1]\ni r\mapsto
\varphi_i\left(\frac{1}{r}S_1,\ldots,\frac{1}{r}S_n\right)\in
B(F^2(H_n))$ is continuous in the operator norm topology. Passing to
the limit as $r\to 1$ in the inequality above, we obtain
\begin{equation}
\label{ine-phi} \left\| \left[\sum_{|\alpha|,|\beta|\leq m}
a_{\alpha \beta}^{(ij)}T_\alpha T_\beta^*\right]_{k\times k}
\right\| \leq \left\| \left[\sum_{|\alpha|,|\beta|\leq m} a_{\alpha
\beta}^{(ij)}P_{\cN_{J_\cP}} \varphi_\alpha\left( S_1,\ldots,
S_n\right) \varphi_\beta\left( S_1,\ldots, S_n\right)^*
|_{\cN_{J_\cP}}\right]_{k\times k} \right\|.
\end{equation}
Hence, using  the fact (proved above)  that
$$
B_i:=P_{\cN_{f,J_{\cP\circ f}}} M_{Z_i} |_{\cN_{f,J_{\cP\circ f}}} =
\left(U^{(1)}|_{\cN_{f,J_{\cP\circ f}}}\right)^{-1} P_{\cN_{J_\cP}}
\varphi_i\left( S_1,\ldots, S_n\right) |_{\cN_{J_\cP}}
\left(U^{(1)}|_{\cN_{f,J_{\cP\circ f}}}\right)
$$
for each $i=1,\ldots,n$, we obtain
$$
\left\| \left[\sum_{|\alpha|,|\beta|\leq m} a_{\alpha
\beta}^{(ij)}T_\alpha T_\beta^*\right]_{k\times k} \right\| \leq
\left\| \left[\sum_{|\alpha|,|\beta|\leq m}
a_{\alpha\beta}^{(ij)}B_\alpha {B _\beta}^*\right]_{k\times k}
\right\|,
$$
which completes the proof.
\end{proof}

 Let $C^*(\cY)$ be the $C^*$-algebra generated by a set of
  operators
$\cY\subset B(\cK)$ and the identity. A  subspace $\cH\subseteq \cK$
is called $*$-cyclic  for $\cY$    if
$$\cK=\overline{\text{\rm span}}\left\{Xh:\ X\in C^*(\cS), \ h\in
\cH\right\}.
$$

\begin{theorem}\label{dil2} Let $f=(f_1,\ldots, f_n)$ be an $n$-tuple of power series
 in the set $\cM_{rad}\cap \cM^{||}$,
 let $\cP\subset \CC[Z_1,\ldots,Z_n]$ be  a set of homogeneous
 polynomials, and   let $B=(B_1,\ldots, B_n)$ be the universal model
  associated with $f$ and     the WOT-closed two-sided ideal  $J_{\cP\circ f}$
    in
$H^\infty(\BB_f)$.
  If  $\cH$ is a separable Hilbert space and $T=(T_1,\ldots, T_n)\in \BB_f(\cH)$ has the property that
$$
(q\circ f)(T)=0,\qquad q\in\cP,
$$
then there exists a separable Hilbert space $\cK_\pi$ and a
$*$-representation $\pi:C^*(B_1,\ldots, B_n)\to B(\cK_\pi)$ which
annihilates the compact operators and
$$
\sum_{i=1}^n f_i(\pi(B_1),\ldots, \pi(B_n))f_i(\pi(B_1),\ldots,
\pi(B_n))^*= I_{ \cK_\pi},
$$
such that
\begin{enumerate}
\item[(i)]
$\cH$ can be identified with a $*$-cyclic co-invariant subspace of
$\tilde\cK:=(\cN_{f,J_{\cP\circ f}}\otimes
\overline{\Delta_{f,T}\cH})\oplus \cK_\pi$ under the operators
$$
V_i:=\left[\begin{matrix} B_i\otimes
I_{\overline{\Delta_{f,T}\cH}}&0\\0&\pi(B_i)
\end{matrix}\right],\quad i=1,\ldots,n;
$$
\item[(ii)]
$ T_i^*=V_i^*|\cH,\quad i=1,\ldots, n. $
\item[(iii)] $V:=(V_1,\ldots,V_n)\in \BB_f(\widetilde \cK)$ and
$$
(q\circ f)(V)=0,\qquad q\in\cP.
$$
\end{enumerate}
  \end{theorem}
\begin{proof}

 Applying Arveson extension theorem
to the map $\Psi_{f,T,\cP}$ of Theorem \ref{ccmap}, we obtain a
unital completely positive linear map
$\Gamma_{f,T,\cP}:C^*(B_1,\ldots, B_n)\to B(\cH)$ such that
$\Gamma_{f,T,\cP}({ B}_\alpha { B}_\beta):=T_\alpha T_\beta^*$  for
$\alpha,\beta\in \FF_n^+$. Consider  $\tilde\pi:C^*(B_1,\ldots,
B_n)\to B(\tilde\cK)$  to be a minimal Stinespring dilation of
$\Gamma_{f,T,\cP}$, i.e.,
$$\Gamma_{f,T,\cP}(X)=P_{\cH} \tilde\pi(X)|_\cH,\qquad X\in C^*(B_1,\ldots, B_n),
$$
and $\tilde\cK=\overline{\text{\rm span}}\{\tilde\pi(X)h:\  X\in
C^*(B_1,\ldots, B_n), h\in \cH\}.$ It is easy to see  that, for each
$i=1,\ldots, n$,
\begin{equation*}
\begin{split}
\Gamma_{f,T,\cP}(B_i B_i^*)&=P_\cH\tilde\pi(B_i)\tilde\pi(B_i^*)|_\cH\\
&=P_\cH \tilde\pi(B_i)(P_\cH+P_{\cH^\perp})\tilde\pi(B_i^*)|_\cH\\
&=\Gamma_{f,T,\cP}(B_i B_i^*)+ (P_\cH
\tilde\pi(B_i)|_{\cH^\perp})(P_{\cH^\perp} \tilde\pi(B_i^*)|_\cH).
\end{split}
\end{equation*}
Consequently, we have $P_\cH \tilde\pi(B_i)|_{\cH^\perp}=0$ and
\begin{equation}\label{morph}
\begin{split}
\Gamma_{f,T,\cP}(B_\alpha X)&=P_\cH(\tilde\pi(B_\alpha) \tilde\pi(X))|_\cH\\
&=(P_\cH\tilde\pi(B_\alpha)|_\cH)(P_\cH\tilde\pi(X)
|_\cH)\\
&=\Gamma_{f,T,\cP}(B_\alpha) \Gamma_{f,T,\cP}(X)\end{split}
\end{equation}
for any $X\in C^*(B_1,\ldots, B_n)$ and $\alpha\in \FF_n^+$. Note
that the Hilbert space $\tilde\cK$ is separable, since $\cH$ has the
same property. Relation  $P_\cH \tilde\pi(B_i)|_{\cH^\perp}=0$ shows
that $\cH$ is an invariant subspace under each $\tilde\pi(B_i)^*$, \
$i=1,\ldots, n$. Therefore,
\begin{equation}\label{coiso}
\tilde\pi(B_i)^*|_\cH=\Gamma_{f,T,\cP}(B_i^*)=T_i^*,\qquad
i=1,\ldots, n.
\end{equation}

Taking into account that  the subspace $\cN_{f,J_{\cP\circ f}}$
contains the constants, we use   Theorem \ref{compact2} to conclude
that all the compact operators   in $B(\cN_{f,J_{\cP\circ f}})$ are
contained in $C^*(B_1,\ldots, B_n)$. We remark that one can obtain a
version of Theorem \ref{wold} in our new setting, in a similar
manner. Consequently, the representation $\tilde\pi$ decomposes into
a direct sum $\tilde\pi=\pi_0\oplus \pi$ on $\tilde \cK=\cK_0\oplus
\cK_\pi$, where $\pi_0$, $\pi$  are disjoint representations of
$C^*(B_1,\ldots, B_n)$ on the Hilbert spaces $\cK_0$ and $\cK_\pi$,
respectively, such that
\begin{equation}\label{sime}
\cK_0\simeq\cN_{f,J_{\cP\circ f}}\otimes \cG, \quad
\pi_0(X)=X\otimes I_\cG, \quad X\in C^*(B_1,\ldots, B_n),
\end{equation}
 for some Hilbert space $\cG$, and $\pi$ is a representation which
annihilates the compact operators.  Since $P_\CC^{\cN_{f,J_{\cP\circ
f}}}=I_{\cN_{f,J_\cP}}-\sum\limits_{i=1}^n f_i(B)f_i(B)^*$ is a rank
one projection in  the $C^*$-algebra $C^*(B_1,\ldots, B_n)$, we have
$\sum\limits_{i=1}^n f_i(\pi(B_1),\ldots,
\pi(B_n))f_i(\pi(B_1),\ldots, \pi(B_n))^*=I_{\cK_\pi}$ and $ \dim
\cG=\dim (\text{\rm range}\,\tilde\pi(P_\CC^{\cN_{f,J_{\cP\circ
f}}})).$ Using the minimality of the Stinespring representation
$\tilde\pi$,
 the proof of Theorem \ref{compact2},
 and the fact that $P_\CC^{\cN_{f,J_{\cP\circ f}}}B_\alpha=0$ if $|\alpha|\geq 1$,  we
 deduce that
\begin{equation*}\begin{split}
\text{\rm range}\,\tilde\pi(P_\CC^{\cN_{f,J_{\cP\circ f}}})&=
\overline{\text{\rm span}}\{\tilde\pi(P_\CC^{\cN_{f,J_{\cP\circ
f}}})\tilde\pi(X)h:\ X\in C^*(B_1,\ldots, B_n),
 h\in \cH\}\\
&= \overline{\text{\rm span}}\{\tilde\pi(P_\CC^{\cN_{f,J_{\cP\circ
f}}})\tilde\pi(Y)h:\ Y \text{ is compact in }
 B(\cN_{f,J_{\cP\circ f}}), \,h\in \cH\}\\
&= \overline{\text{\rm span}}\{\tilde\pi(P_\CC^{\cN_{f,J_{\cP\circ
f}}})\tilde\pi(B_\alpha P_\CC^{\cN_{f,J_{\cP\circ f}}}
B_\beta^*)h:\ \alpha,\beta\in \FF_n^+, h\in \cH\}\\
&= \overline{\text{\rm span}}\{\tilde\pi(P_\CC^{\cN_{f,J_{\cP\circ
f}}})\tilde\pi(B_\beta^*)h:\ \beta\in \FF_n^+, h\in \cH\}.
\end{split}
\end{equation*}
Now, due to  relation \eqref{morph}, we have
\begin{equation*}\begin{split}
\left<\tilde\pi(P_\CC^{\cN_{f,J_{\cP\circ
f}}})\tilde\pi(B_\alpha^*)h,
\tilde\pi(P_\CC^{\cN_{f,J_{\cP\circ
f}}})\tilde\pi(B_\beta^*)k\right>
&=
\left<h,\pi(B_\alpha)\pi(P_\CC^{\cN_{f,J_{\cP\circ f}}})\pi(B_\beta^*)h\right>\\
&=
\left<h,T_\alpha\left(I_\cH-\sum_{i=1}^n f_i(T)f_i(T)^*\right)T_\beta^*h\right>\\
&= \left<\Delta_{f,T}T_\alpha^*h,\Delta_{f,T}T_\beta^*k\right>
\end{split}
\end{equation*}
for any $h, k \in \cH$ and $\alpha,\beta\in \FF_n^+$. Consequently,
the map $\Lambda:\text{\rm
range}\,\tilde\pi(P_\CC^{\cN_{f,J_{\cP\circ f}}})\to
\overline{\Delta_{f,T}\cH}$  defined by
$$
\Lambda(\tilde\pi(P_\CC^{\cN_{f,J_{\cP\circ
f}}})\tilde\pi(B_\alpha^*)h):=\Delta_{f,T} T_\alpha^*,\quad h\in
\cH,
$$
can  be extended  by linearity and continuity to a unitary operator.
Hence,
$$
\dim[\text{\rm range}\,\pi(P_\CC^{\cN_{f,J_{\cP\circ f}}})]= \dim
\overline{\Delta_{f,T}\cH}=\dim \cG.
$$
Under the appropriate identification of $\cG$ with
$\overline{\Delta_{f,T}\cH}$ and using  relations \eqref{coiso} and
\eqref{sime},
  we  obtain the  required dilation.

  To prove item (iii), note that since
  $B_i:=P_{\cN_{f,J_{\cP\circ f}}} M_{Z_i} |_{\cN_{f,J_{\cP\circ f}}}$,   $i=1,\ldots,
  n$, we have $(q\circ f)(B_1,\ldots,B_n)=0$, $q\in\cP$.
  Taking into account that $q\in\cP$ is a polynomial and $f\in
  \cM^{||}$, the latter equality implies
  $(q\circ f)(\pi(B_1),\ldots,\pi(B_n))=0$.  Therefore, $(q\circ f)(V_1,\ldots,V_n)=0$ for
   $q\in\cP$.
   On the other hand, since $\sum_{i=1}^n
   f_i(B_1,\ldots,B_n)f_i(B_1,\ldots,B_n)^*\leq I$ and $f\in
   \cM^{||}$, we also have  $$\sum_{i=1}^n
   f_i(\pi(B_1),\ldots,\pi(B_n))f_i(\pi(B_1),\ldots,\pi(B_n))^*\leq
   I,$$
which proves that $(\pi(B_1),\ldots,\pi(B_n))\in \BB_f(\cK_\pi)$.
Consequently, $(V_1,\ldots, V_n)\in \BB_f(\widetilde \cK)$. This
completes the proof.
\end{proof}

We remark that if, in addition to the hypotheses of Theorem
$\ref{dil2}$,
\begin{equation*}
\overline{\text{\rm span}}\,\{B_\alpha B_\beta^*:\ \alpha,\beta\in \FF_n^+\}=C^*(B_1,\ldots, B_n),
\end{equation*}
then the map $\Gamma_{f,T,\cP}$ is unique and  the  dilation  is
minimal, i.e., $\tilde\cK=\bigvee\limits_{\alpha\in \FF_n^+}
V_\alpha \cH$. In this case, the minimal dilation of Theorem
\ref{dil2} is unique, due to the uniqueness  of the minimal
Stinespring representation.

\begin{corollary}Let $V:=(V_1,\ldots, V_n)$ be the dilation  of Theorem $\ref{dil2}$. Then,
\begin{enumerate}
\item[(i)]
 $V$ is a pure $n$-tuple if and only if $T$ is  pure;
\item[(ii)]
 $
f_1(V)f_1(V)^*+\cdots +f_n(V)f_n(V)^*=I
$ if  and only if

$$
f_1(T)f_1(T)^*+\cdots +f_n(T)f_n(T)^*=I
$$
\end{enumerate}
\end{corollary}
\begin{proof}
Note that
$$
\sum_{|\alpha|=k}[f(T)]_\alpha [f(T)]_\alpha^*=P_\cH \left[\begin{matrix}
\sum\limits_{|\alpha|=k}[f(B)]_\alpha [f(B)]_\alpha^*\otimes
I_{\overline{\Delta_{f,T}\cH}}&0\\0& I_{\cK_\pi}\end{matrix}\right]|_\cH
$$
and, consequently,
$$
\text{\rm SOT-}\lim_{k\to\infty}\sum_{|\alpha|=k}[f(T)]_\alpha [f(T)]_\alpha^*=P_\cH
\left[\begin{matrix}
 0&0\\0& I_{\cK_\pi}\end{matrix}\right]|_\cH.
$$
Hence, we deduce  that $T$ is a pure $n$-tuple if and only if
$\cH\perp (0\oplus \cK_\pi)$, i.e.,  $\cH\subseteq
\cN_{f,J_{\cP\circ f}}\otimes\overline{\Delta_{f,T}\cH}$. Taking
into account that  $\cN_{f,J_{\cP\circ
f}}\otimes\overline{\Delta_{f,T}\cH}$ is reducing for each operator
$V_i$, \ $i=1,\ldots, n$, and $\tilde \cK$ is the smallest reducing
subspace  for the same operators, which contains $\cH$, we draw the
conclusion that $\tilde\cK=\cN_{f,J_{\cP\circ
f}}\otimes\overline{\Delta_T\cH}$. To prove item (ii), assume that
$\sum\limits_{i=1}^n f_i(V)f_i(V)^*=I_{\tilde\cK}$. Since
$$\sum_{|\alpha|=k} [f(V)]_\alpha [f(V)]_\alpha^*=
\left[\begin{matrix} \sum\limits_{|\alpha|=k}[f(B)]_\alpha
[f(B)]_\alpha^*\otimes I_{\overline{\Delta_{f,T}\cH}}&0\\0&
I_{\cK_\pi}\end{matrix}\right],
$$
we must have  $\sum\limits_{|\alpha|=k}[f(B)]_\alpha
[f(B)]_\alpha^*\otimes I_{\overline{\Delta_{f,T}\cH}}=I_{\cK_0}$ for
any $k=1,2,\ldots$. Since   $$\text{\rm
SOT-}\lim\limits_{k\to\infty} \sum\limits_{|\alpha|=k}[f(B)]_\alpha
[f(B)]_\alpha^*=0, $$
 we deduce that $\cK_0=\{0\}$. Now, using the
proof of Theorem \ref{dil2}, we concluded that $\cG=\{0\}$ and,
therefore,  $\Delta_{f,T}=0$.  The converse is straightforward. The
proof is complete.
\end{proof}

\bigskip

\section{Beurling type theorem and commutant lifting
in noncommutative varieties}

In this section, we provide  a Beurling type theorem characterizing
the invariant subspaces under the universal $n$-tuple   associated
with  a noncommutative variety $\cV_{f,J}^{pure}(\cH)$, and a
commutant lifting theorem for $n$-tuples of operators in
$\cV_{f,J}^{pure}(\cH)$.

\begin{theorem}\label{Beur} Let $f=(f_1,\ldots, f_n)$ be an $n$-tuple of power
series with the model property.
 Let $J\neq H^\infty(\BB_f)$ be a WOT-closed two-sided ideal of the
  Hardy algebra $H^\infty(\BB_f)$ and let $(B_1,\ldots, B_n)\in \cV_{f,J}^{cnc}(\cN_{f,J})$ be the corresponding universal model.
A subspace $\cM\subseteq \cN_{f,J}\otimes \cK$ is invariant under
each operator $B_1\otimes I_\cK,\ldots, B_n\otimes I_\cK$ if and
only if there exists a Hilbert space  $\cG$ and an  operator
$\Theta: \cN_{f,J}\otimes \cG\to \cN_{f,J}\otimes \cK$ with the
following properties:
\begin{enumerate}
\item[(i)]
$\Theta$  is a partial isometry and
$$\Theta(B_i\otimes I_\cG)=(B_i\otimes I_\cK)\Theta,\quad i=1,\ldots,n.
$$
\item[(ii)]
 $\cM=\Theta\left( \cN_{f,J}\otimes \cG\right)$.
\end{enumerate}
\end{theorem}

\begin{proof}
First, note that the subspace $\cN_{f,J}\otimes K$ is invariant
under each operator $M_{Z_i}^*\otimes I_\cK$, $i=1,\ldots, n$, and
$$(M_{Z_i}^*\otimes I_\cK)|_{\cN_{f,J}\otimes \cK}=B_i^*\otimes I_\cK, \quad i=1,\ldots,n. $$
Since the subspace $[\cN_{f,J}\otimes \cK]\ominus \cM$  is invariant under
 $B_i^*\otimes I_\cK$, \ $i=1,\ldots, n$,  it is also invariant under
   each  operator $M_{Z_i}^*\otimes I_\cK$.
Consequently, the subspace
\begin{equation}\label{E}
\cE:=[\HH^2(f)\otimes \cK]\ominus\{[\cN_{f,J}\otimes \cK]\ominus \cM\}=
[\cM_{f,J}\otimes \cK]\oplus \cM
\end{equation}
is invariant under $M_{Z_i}\otimes I_\cK$, \ $i=1,\ldots, n$, where
$\cM_{f,J}:=\HH^2(f)\ominus \cN_{f,J}$. Using the Beurling type
characterization of the invariant subspaces under $M_{Z_1},\ldots,
M_{Z_n}$ (see Theorem 5.2 from \cite{Po-multi}),   we find a Hilbert
space $\cG$ and  an isometric  operator $\Psi: \HH^2(f)\otimes\cG\to
\HH^2(f)\otimes\cK$ such that $\Psi(M_{Z_i}\otimes
I_\cG)=(M_{Z_i}\otimes I_\cK)$ for $i=1,\ldots,n$ and
$$
\cE=\Psi[\HH^2(f)\otimes \cG].
$$
Since $\Psi$ is an isometry, we have $ P_\cE=\Psi\Psi^*, $ where
$P_\cE$ is the orthogonal projection of $\HH^2(f)\otimes \cK$ onto
$\cE$. Note that the subspace $\cN_{f,J}\otimes \cK$ is invariant
under $\Psi^*$. Setting $\Theta:=P_{\cN_{f,J}\otimes \cK}
\Psi|_{\cN_{f,J}\otimes \cG}$, we have
$$
P_{\cN_{f,J}\otimes \cK}P_\cE|_{\cN_{f,J}\otimes
\cK}=\Theta\Theta^*.
$$
 Hence, and using relation \eqref{E},  we deduce that
$ P_\cM=\Theta\Theta^* $, where $P_\cM$ is the orthogonal projection
of $\cN_{f,J}\otimes \cK$  onto $\cM$. Therefore $\Theta$ is a
partial isometry and $\cM=\Theta\left[ \cN_{f,J}\otimes \cG\right]$.
 The proof
is complete.
\end{proof}

  We recall that, due to Theorem \ref{model}, any $n$-tuple $(T_1,\ldots, T_n)$
   in the noncommutative variety $\cV_{f,J}^{pure}(\cH)$
  is unitarily equivalent
   to the compression of $[B_1\otimes I_\cK,\ldots, B_n\otimes I_\cK]$ to a co-invariant subspace $\cE$ under each operator $B_i\otimes I_\cK$, $i=1,\ldots, n$. Therefore, we have
$$T_i=P_\cE(B_i\otimes I_\cK)|_\cE,\qquad i=1,\ldots,n.
$$
The following result is  a commutant lifting theorem for $n$-tuples
of operators in the noncommutative variety $\cV_{f,J}^{pure}(\cH)$.

\begin{theorem}\label{CLT2}  Let $f=(f_1,\ldots, f_n)$ be an $n$-tuple of power
series with the model property. Let $J\neq H^\infty(\BB_f)$ be a
WOT-closed two-sided ideal of  the  Hardy algebra
$H^\infty(\BB_f)$  and let $(B_1,\ldots, B_n)$  and  be the
corresponding universal model acting on $\cN_{f,J}$.  For each
$j=1,2$, let $\cK_j$ be a Hilbert space and $\cE_j\subseteq
\cN_{f,J}\otimes \cK_j$ be a co-invariant subspace  under each
operator $B_1\otimes I_\cK,\ldots, B_n\otimes I_\cK$. If $X:\cE_1\to
\cE_2$ is a bounded operator such that
\begin{equation*}
X[P_{\cE_1}(B_i\otimes I_{\cK_1})|_{\cE_1}]=[P_{\cE_2}(B_i\otimes
I_{\cK_2})]|_{\cE_2}X,\qquad i=1,\ldots,n,
\end{equation*}
then there exists an operator $G:\cN_{f,J}\otimes\cK_1\to
\cN_{f,J}\otimes \cK_2$  with the following properties:
\begin{enumerate}
\item[(i)]
 $G(B_i\otimes I_{\cK_1})=(B_i\otimes I_{\cK_2})G$ for  $i=1,\ldots,n$;
\item[(ii)]
$G^*\cE_2\subseteq \cE_1$,
$
G^*|\cE_2=X^*$,  and
$\|G\|=\|X\|$.
\end{enumerate}
\end{theorem}

\begin{proof}
Note that the subspace  $\cN_{f,J}\otimes \cK_j$
is  invariant under  each operator $M_{Z_i}^*\otimes I_{\cK_j}$, \
$i=1,\ldots, n$,  and
$$
(M_{Z_i}^*\otimes I_{\cK_j})|_{\cN_{f,J}\otimes \cK_j}=B_i^*\otimes
I_{\cK_j},\quad  i=1,\ldots, n.
$$
Since $\cE_j\subseteq \cN_{f,J}\otimes \cK_j$ is invariant   under
$B_i^*\otimes I_{\cK_j}$ it is also invariant under $M_{Z_i}^*\otimes
I_{\cK_j}$  and
$$
(M_{Z_i}^*\otimes I_{\cK_j})|_{\cE_j}=  (B_i^*\otimes
I_{\cK_j})|_{\cE_j}, \quad i=1,\ldots, n.
$$
Consequently,  the intertwining relation  in the hypothesis  implies
\begin{equation*}
XP_{\cE_1}(M_{Z_i}\otimes I_{\cK_1})|_{\cE_1}=P_{\cE_2}(M_{Z_i}\otimes
I_{\cK_2})|_{\cE_2}X,\quad i=1,\ldots,n.
\end{equation*}
We remark that, for each $j=1,2$, the $n$-tuple $(M_{Z_1} \otimes
I_{\cK_j},\ldots, M_{Z_n} \otimes I_{\cK_j})$ is  a dilation of the
$n$-tuple
$$
[P_{\cE_j}(M_{Z_1}\otimes I_{\cK_j})|_{\cE_j},\ldots,
P_{\cE_j}(M_{Z_n}\otimes I_{\cK_j})|_{\cE_j}].
$$
 Applying Theorem 9.1 from \cite{Po-multi},
 we find an
operator  $\Phi: \HH^2(f)\otimes\cK_1\to \HH^2(f)\otimes\cK_2$ with
the following properties:
\begin{enumerate}
\item[(i)]
$\Phi(M_{Z_i}\otimes I_{\cK_1})=(M_{Z_i}\otimes I_{\cK_2})\Phi$ for
$i=1,\ldots,n$;
\item[(ii)]
$\Phi^*\cE_2\subseteq \cE_1$,
$
\Phi^*|\cE_2=X^*$,  and
$\|\Phi\|=\|X\|$.
\end{enumerate}
 Set $G:=P_{\cN_{f,J}\otimes \cK_2} \Phi|_{\cN_J\otimes \cK_1}$.
Since $\Phi^*(\cN_{f,J}\otimes \cK_2)\subseteq \cN_{f,J}\otimes
\cK_1$,  the subspace  $\cN_{f,J}\otimes \cK_j$ is  invariant under
each operator $M_{Z_1}^*\otimes I_{\cK_j},\ldots, M_{Z_n}^*\otimes
I_{\cK_j}$,  and $\cE_j\subseteq \cN_{f,J}\otimes \cK_j$,
 the relations above  imply  $G(B_i\otimes I_{\cK_1})=(B_i\otimes I_{\cK_2})G$ for  $i=1,\ldots,n$,
$
G^*\cE_2\subseteq \cE_1$,  and
$G^*|_{\cE_2}=X^*$.
Hence,  we deduce that
$
\|X\|\leq \|G\|\leq \|\Phi\|=\|X\|.
$
Therefore, $\|G\|=\|X\|$.
 The proof is complete.
\end{proof}

{\bf The commutative case.}  Let $f=(f_1,\ldots, f_n)$ be an
$n$-tuple of power series with the model property. Let  $J_c$ be the
$WOT$-closed two-sided ideal of the Hardy algebra $H^\infty(\BB_f)$
generated by the commutators
$$ M_{Z_i}M_{Z_j}-M_{Z_j}M_{Z_i},\qquad i,j=1,\ldots, n.
$$
According to \cite{Po-multi},  the subspace $\cN_{f,J_c}$ coincides
with  the symmetric Hardy  space associated with $\BB_f$,
$\HH^2_s(f)$,  which can be identified with the Hilbert space
$H^2(\BB_{f}^{<}(\CC))$ of holomorphic functions on $\BB_f^<(\CC)$,
namely, the reproducing kernel Hilbert space with reproducing kernel
$K_f:\BB_f^<(\CC)\times \BB_f^<(\CC)\to \CC$ defined by
$$
K_f(\mu,\lambda):=\frac{1}{1-\sum_{ i=1}^n   f_i(\mu)
\overline{f_i(\lambda)}},\qquad \lambda,\mu\in \BB_f^<(\CC).
$$
We recall that
$$\BB_f^<(\CC):=\{\lambda=(\lambda_1,\ldots,
\lambda_n)\in \CC^n : \  \lambda=g(f(\lambda)) \  \text{ and } \
\sum_{i=1}^n |f_i(\lambda)|^2<1\}=g({\bf B}_n),
$$
where  ${\bf B}_n:=\{(z_1,\ldots, z_n)\in \CC^n: \ \sum_{i=1}^n
|z_i|^2<1\}$ and $g=(g_1,\ldots, g_n)$ is the inverse of $f$ with
respect to the composition.
 The algebra $P_{\HH_s^2(f)} H^\infty(\BB_f)|_{\HH_s^2(f)}$  coincides with the
 $WOT$-closed algebra  generated by  the operators
  $L_i:=P_{\HH_s^2(f)} M_{Z_i}|_{\HH_s^2(f)}$, $i=1,\dots, n$, and
 can be identified
 with the algebra of  all multipliers of the  Hilbert space $H^2(\BB_f^<(\CC))$.
 Under this identification, the  operators $L_1,\ldots, L_n$ become
 the multiplication operators $M_{z_1},\ldots, M_{z_n}$ by the coordinate functions $z_1,\ldots, z_n$.

Under the above-mentioned identifications, if $T:=(T_1,\ldots,
T_n)\in \BB_f(\cH)$ is  such that
$$
T_iT_j=T_jT_i, \qquad i,j=1,\ldots, n,
$$
then
 the  characteristic
function of $T$  is the multiplier
$M_{\Theta_{f,J_c,T}}:\HH^2(\BB_f^<(\CC))\otimes \cD_{f,T^*}\to
\HH^2(\BB_f^<(\CC))\otimes \cD_{f,T} $ defined by    the
operator-valued analytic function
\begin{equation*}
\begin{split}
\Theta_{f,J_c,T}(z):=
  - f(T)+
 \Delta_{f,T}\left(I-\sum_{i=1}^n f_i(z) f_i(T)^*\right)^{-1}
\left[f_1(z)I_\cH,\ldots, f_n(z) I_\cH \right] \Delta_{f,T^*},\qquad
z\in \BB_f^<(\CC).
\end{split}
\end{equation*}
We remark that all   the results  of  the last three sections can be
written in this commutative setting.

\bigskip

       %

      \end{document}